\documentclass{amsart}

\usepackage{amssymb,amsmath,amsaddr}
\usepackage{tikz,hyperref}

\newcommand{\class}{\mathsf}
\newcommand{\logic}{\mathcal}
\newcommand{\alg}{\mathbf}
\newcommand{\matr}{\mathbb}
\newcommand{\filter}{\mathcal}

\newcommand{\assign}{:=}
\newcommand{\napprox}{\not\approx}
\newcommand{\set}[2]{\{ #1 \mid #2 \}}
\newcommand{\pair}[2]{\langle #1, #2 \rangle}

\newcommand{\dmneg}{{-}}
\newcommand{\dual}{\partial}

\newcommand{\True}{\mathsf{t}}
\newcommand{\False}{\mathsf{f}}

\newcommand{\logleq}{\leq}
\newcommand{\loggeq}{\geq}
\newcommand{\nlogleq}{\nleq}
\newcommand{\notto}{\not\to}

\newcommand{\unit}{\eta}
\newcommand{\counit}{\varepsilon}

\newcommand{\sigmapushp}{\sigma_{p}}
\newcommand{\sigmapopp}{\tau_{p}}

\DeclareMathOperator{\Alg}{Alg}
\DeclareMathOperator{\Ker}{Ker}
\DeclareMathOperator{\Log}{Log}
\DeclareMathOperator{\Mod}{Mod}
\DeclareMathOperator{\Exp}{Exp}
\DeclareMathOperator{\Ext}{Ext}

\newcommand{\Leibniz}[2]{\Omega^{#1} (#2)}

\newcommand{\CL}{\logic{CL}}
\newcommand{\K}{\logic{K}}
\newcommand{\LP}{\logic{LP}}
\newcommand{\KO}{\logic{KO}}
\newcommand{\BD}{\logic{BD}}
\newcommand{\ETL}{\logic{ETL}}
\newcommand{\ECQ}{\logic{ECQ}}
\newcommand{\Kminus}{\K_{-}}
\newcommand{\KOminus}{\KO_{-}}
\newcommand{\ETLplus}{\logic{EDS}}

\newcommand{\mc}[1]{#1_{\mathrm{mc}}}
\newcommand{\CLmc}{\mc{\CL}}
\newcommand{\Kmc}{\mc{\K}}
\newcommand{\LPmc}{\mc{\LP}}
\newcommand{\KOmc}{\mc{\KO}}
\newcommand{\BDmc}{\mc{\BD}}

\newcommand{\Btwo}{\alg{B_{2}}}
\newcommand{\Kthree}{\alg{K_{3}}}

\newcommand{\DMfour}{\alg{DM_{4}}}

\newcommand{\CLmatrix}{\matr{B}_{\alg{2}}}
\newcommand{\Kmatrix}{\matr{K}_{\alg{3}}}
\newcommand{\LPmatrix}{\matr{P}_{\alg{3}}}
\newcommand{\BDmatrix}{\matr{BD}_{\alg{4}}}
\newcommand{\ETLmatrix}{\matr{ETL}_{\alg{4}}}
\newcommand{\Kminusmatrix}{\matr{ETL}_{\alg{8}}}

\newcommand{\AlgH}{\mathbb{H}}
\newcommand{\AlgInvH}{\mathbb{H}^{-1}}
\newcommand{\AlgHS}{\mathbb{H}_{\mathrm{S}}}
\newcommand{\AlgInvHS}{\mathbb{H}^{-1}_{\mathrm{S}}}
\newcommand{\AlgI}{\mathbb{I}}
\newcommand{\AlgS}{\mathbb{S}}
\newcommand{\AlgP}{\mathbb{P}}
\newcommand{\AlgPu}{\mathbb{P}_{\mathrm{U}}}
\newcommand{\AlgPuStar}{\mathbb{P_{\mathrm{U}}^{*}}}

\newcommand{\boldmu}{\boldsymbol{\mu}}
\newcommand{\boldgamma}{\boldsymbol{\gamma}}

\newtheorem{theorem}{Theorem}[section]
\newtheorem{proposition}[theorem]{Proposition}
\newtheorem{lemma}[theorem]{Lemma}
\newtheorem{fact}[theorem]{Fact}
\newtheorem{corollary}[theorem]{Corollary} 
\theoremstyle{definition}
\newtheorem{definition}[theorem]{Definition}

\author{Adam \texorpdfstring{P\v{r}enosil}{Prenosil}}
\title{The lattice of super-Belnap logics}
\address{Institute of Computer Science, Czech Academy of Sciences, Czechia}
\email{adam.prenosil@gmail.com}


\thanks{The author acknowledges the support of the project P202/12/G061 of the Czech Science Foundation. The author is grateful to Umberto Rivieccio for sharing his research notes, to Petr Cintula and Carles Noguera for a helpful discussion about non-finitary logics, and to the three anonymous referees for their useful comments.}

\begin{document}

\begin{abstract}
  We study the lattice of extensions of four-valued Belnap--Dunn logic, called super-Belnap logics by analogy with superintuitionistic logics. We describe the global structure of this lattice by splitting it into several subintervals, and prove some new completeness theorems for super-Belnap logics. The crucial technical tool for this purpose will be the so-called antiaxiomatic (or explosive) part operator. The antiaxiomatic (or explosive) extensions of Belnap--Dunn logic turn out to be of particular interest owing to their connection to graph theory: the lattice of finitary antiaxiomatic extensions of Belnap--Dunn logic is iso\-morphic to the lattice of upsets in the homomorphism order on finite graphs (with loops allowed). In particular, there is a continuum of finitary super-Belnap logics. Moreover, a non-finitary super-Belnap logic can be constructed with the help of this isomorphism. As algebraic corollaries we obtain the existence of a continuum of antivarieties of De \mbox{Morgan} algebras and the existence of a prevariety of De Morgan algebras which is not a quasivariety.
\end{abstract}

\maketitle

\section{Introduction}

  The present paper is an attempt to map out the landscape of extensions of the four-valued logic introduced in the 1960's and 1970's by Dunn~\cite{dunn66,dunn69,dunn76} as the so-called first-degree fragment of the logic of \mbox{entailment} of Anderson and Belnap~\cite{anderson+belnap75} and later proposed by Belnap \cite{belnap77a,belnap77b} as a logic which a computer could use to handle inconsistent and incomplete information. This logic will be called Belnap--Dunn logic here and denoted~$\BD$. It is also known as the logic of first-degree entailment (FDE). For a more complete account of the origins of this logic, see \cite{dunn10,dunn16}.

  Belnap--Dunn logic has attracted considerable attention from researchers in logic and computer science since its introduction in the 1970's. However, few of its non-classical extensions have been investigated in detail. Most prominent among these are Kleene's strong three-valued logic $\K$~\cite{kleene38,kleene52} and Priest's Logic of Paradox $\LP$~\cite{priest79}. These two logics have been widely used by philosophers who accept truth gaps or truth gluts in their accounts of truth. The intersection of these two logics, which we call Kleene's logic of order and denote $\KO$, was occasionally studied as well. It is mentioned by Makinson~\cite{makinson73}, who calls it Kalman implication, and identified by Dunn~\cite{dunn76b} as the first-degree fragment of the relevance logic R-Mingle. More recently, Exactly True Logic was introduced and studied by Pietz \& Rivieccio~\cite{pietz+rivieccio13}. This seems to exhaust the list of non-classical super-Belnap logics which have been studied in any detail.

  The idea of studying extensions of $\BD$ as a family of logics in its own right was first proposed by Rivieccio~\cite{rivieccio12}, who called such extensions \emph{super-Belnap logics} by analogy with super-intuitionistic logics. Among other things, Rivieccio proved that there are infinitely many super-Belnap logics. We take up his proposal and study the structure of the lattice of super-Belnap logics.

  Unlike in the case of intuitionistic logic, where the axiomatic extensions are the main objects of interest, here it is the antiaxiomatic extensions (extensions by rules stating that a certain set of formulas is inconsistent) which are of interest. Indeed, $\BD$ has only one non-trivial proper axiomatic extension, namely~$\LP$, while it turns out that it has a continuum of finitary antiaxiomatic extensions. Before engaging in the study of super-Belnap logics, we therefore establish some basic facts about antiaxiomatic (or explosive) extensions (Section~\ref{sec: explosive}).

  In particular, the antiaxiomatic (or explosive) part of a logic turns out to be a useful tool in this context. The explosive part $\Exp_{\logic{B}} \logic{L}$ of an extension $\logic{L}$ of a base logic~$\logic{B}$ is the strongest antiaxiomatic extension of $\logic{B}$ which lies below~$\logic{L}$. The~logic determined by a product of matrices $\prod_{i \in I} \matr{A}_{i}$ can then be computed from the logics determined by the matrices $\matr{A}_{i}$ and their explosive parts.

  Computing the explosive parts of known super-Belnap logics will enable us, after reviewing their basic properties (Section~\ref{sec: known super-belnap}), to prove some new completeness theorems for super-Belnap logics (Section~\ref{sec: completeness}). For example, the logic $\ECQ$ which extends $\BD$ by the rule of \emph{ex contradictione quodlibet} $p, \dmneg p \vdash q$ is precisely the explosive part of the Exactly True Logic~$\ETL$ of Pietz \& Rivieccio~\cite{pietz+rivieccio13} which extends $\BD$ by the rule of disjunctive syllogism $p, \dmneg p \vee q \vdash q$. (The rule of \emph{ex contradictione quodlibet} $p, \dmneg p \vdash q$ is an example of an antiaxiomatic rule: it~states that the set of formulas $\{ p, \dmneg p \}$ is inconsistent.) As~a consequence, we obtain a completeness theorem for $\ECQ$.

  We then describe the large-scale structure of the lattice of extensions of $\BD$ (Section~\ref{sec: super-belnap lattice}). It has a smallest proper extension $\LP \cap \ECQ$, as well as a largest non-trivial extension, namely classical logic $\CL$. The interval $[\LP \cap \ECQ, \CL]$ decomposes into three disjoint intervals: $[\LP \cap \ECQ, \LP]$, $[\ECQ, \LP \vee \ECQ]$, and $[\ETL, \CL]$. This~last interval moreover has the structure $\ETL < [\ETL_{2}, \Kminus] < \K < \CL$, where $\ETL_{2}$ is the extension of $\ETL$ by the rule $(p \wedge \dmneg p) \vee (q \wedge \dmneg q) \vdash r$ and $\Kminus$ extends $\ETL$ by the rules $(p_1 \wedge \dmneg p_1) \vee \dots \vee (p_n \wedge \dmneg p_n) \vee q, \dmneg q \vee r \vdash r$ for each $n \in \omega$. We then determine which super-Belnap logics enjoy various metalogical properties such as structural completeness or the proof by cases property.

  While lattices of logics have long been studied, especially in the context of modal, super-intuitionistic, and substructural logics~\cite{blok80,bezhanishvili06,glimpse07}, the present study differs from most of these investigations in two respects. Firstly, we consider all extensions of $\BD$ rather than only axiomatic extensions. Secondly, the link between logic and algebra is too weak in the realm of super-Belnap logics to permit a straightforward application of algebraic techniques. In the case of super-intuitionistic logics, there is a straightforward correspondence between axiomatic extensions of intuitionistic logic and varieties of Heyting algebras. In~contrast, there is no such straightforward bijective correspondence between super-Belnap logics and quasivarieties of De~Morgan algebras, which form the algebraic counter\-part of Belnap--Dunn logic in the sense of \cite{font16,font+jansana09}. 

  In~technical terms, intuitionistic logic is algebraizable (its consequence relation is equivalent, in a suitable sense, to equational consequence in Heyting algebras), while Belnap--Dunn logic fails to satisfy even the much weaker property of being proto\-algebraic (it lacks an implication satisfying the axiom of Reflexivity and the rule of Modus Ponens). The present investigation therefore also has some value as a contribution to the study of lattices of non-protoalgebraic logics.

  The above results, it turns out, do not substantially depend on whether the truth and falsity constants $\True$ and $\False$ are taken to be part of the signature of the logic (Section~\ref{sec: different frameworks}). While Belnap--Dunn logic has typically been studied without these constants, their inclusion changes the picture only marginally. On the other hand, moving to a multiple-conclusion setting changes the picture dramatically: the multiple-conclusion form of $\BD$ only has finitely many extensions, namely the multiple-conclusion forms of $\BD$, $\KO$, $\K$, $\LP$, and $\CL$. This is because the move to the multiple-conclusion setting amounts to forcing the proof by cases property: $\Gamma, \varphi \vee \psi \vdash \chi$ holds if and only if $\Gamma, \varphi \vdash \chi$ and $\Gamma, \psi \vdash \chi$ hold. To~go beyond these well-studied extensions of~$\BD$, one must be ready to abandon this~property.

  The second half of the paper is devoted to working out the relationship between super-Belnap logics and finite graphs (we allow for loops). Each finite reduced model of $\BD$ in the sense of abstract algebraic logic is determined up to isomorphism by a pair of graphs and a non-negative integer (Section~\ref{sec: graph duality}). Even better, each finite reduced model of Exactly True Logic $\ETL$ is determined up to logical equivalence by a single graph and a single bit $k \in \{ 0, 1 \}$. Ultimately, this follows from the duality theory for De Morgan algebras~\cite{cornish+fowler77}. As a consequence, we obtain certain graph-theoretic completeness theorems (Section~\ref{sec: graph completeness}).

  A somewhat unexpected connection between explosive super-Belnap logics and the homo\-morphism order on finite graphs now arises (Section~\ref{sec: hom order}): the lattice of finitary antiaxiomatic extensions of $\BD$ is dually isomorphic to the lattice of upsets in the homomorphism order on finite graphs. It immediately follows that there is a continuum of finitary anti\-axiomatic extensions of $\BD$ (and consequently a continuum of antivarieties of De~Morgan algebras), improving on the result of Rivieccio that there are infinitely many finitary super-Belnap logics. Moreover, we can use the countable universality of the homomorphism order on graphs to construct a non-finitary super-Belnap logic. We can also use this graph-theoretic connection to prove that the super-Belnap logics $\ECQ_{n}$ and $\ETL_{n}$, defined as extensions of $\BD$ and $\ETL$ by the rule $(p_1 \wedge \dmneg p_1) \vee \dots \vee (p_n \wedge \dmneg p_n) \vdash q$, are not complete with respect to any finite set of finite matrices for $n \geq 2$.

 Finally (Section~\ref{sec: graphs to logics}), we describe the lattice of all finitary extensions of $\ETL$ in terms of graphs. In particular, its interval $[\ETL, \ETL_{\omega}]$ is isomorphic to the lattice of classes of non-empty graphs without loops closed under homomorphic images, disjoint unions, and contracting isolated edges. A description of the full lattice of finitary super-Belnap logics in terms of classes of triples $\langle G, H, k \rangle$, where $G$ and $H$ are graphs and $k \in \{ 0, 1 \}$, is possible but rather cumbersome.

  The bulk of this paper is based on the author's thesis~\cite{prenosil18thesis}. Some of the results proved here, including the fact that $\Kminus$ is a lower cover of $\K$ and $\ETL_{2}$ is an upper cover of $\ETL$, were already established by Rivieccio in his unpublished research notes~\cite{rivieccio11}, which he kindly shared with the present author. Some of the research presented here was also summarized in~\cite{albuquerque+prenosil+rivieccio17}.

\section{Logical preliminaries}
\label{sec: preliminaries}

  This preliminary section introduces the basic notions of abstract algebraic logic which will be used throughout the paper. For a more thorough introduction to the field, the reader may consult the textbook~\cite{font16}, the monographs \cite{czelakowski01} and \cite{wojcicki88}, or the survey paper \cite{font+jansana+pigozzi03}. Towards the end of the section, we also recall some universal algebraic preliminaries.

  The \emph{signature} of a logic is given by an infinite set of propositional variables (also called atoms) and a set of connectives of finite arities. The algebra of formulas is then the absolutely free algebra generated by these variables. Less abstractly, the set of formulas is obtained by closing the set of atoms under the given connectives in the obvious way. Atoms will be denoted by $p$, $q$, $r$, formulas by $\varphi$, $\psi$, $\chi$, and sets of formulas by $\Gamma$, $\Delta$. A \emph{substitution} is an endomorphism of the algebra of formulas. Equivalently, substitutions may be viewed as mappings which assign a formula to each atom. Each such mapping then naturally extends to a function $\sigma$ which assigns to each formula $\varphi$ its substitution instance $\sigma(\varphi)$. Let us consider a certain fixed signature in the following definitions.

  A \emph{rule} is a pair consisting of a set of formulas $\Gamma$ and a formula $\varphi$, written as $\Gamma \vdash \varphi$. A \emph{logic}~$\logic{L}$ is a set of rules which satisfies the following conditions:
\begin{itemize}
\item $\varphi \vdash_{\logic{L}} \varphi$ (reflexivity),
\item $\text{if } \Gamma \vdash_{\logic{L}} \varphi \text{, then } \Gamma, \Delta \vdash_{\logic{L}} \varphi$ (monotonicity),
\item $\text{if } \Gamma \vdash_{\logic{L}} \delta \text{ for all } \delta \in \Delta \text{ and } \Delta \vdash_{\logic{L}} \varphi \text{, then } \Gamma \vdash_{\logic{L}} \varphi$ (cut),
\item $\text{if }\Gamma \vdash_{\logic{L}} \varphi \text{, then } \sigma[\Gamma] \vdash_{\logic{L}} \sigma (\varphi) \text{ for each substitution } \sigma$ (structurality),
\end{itemize}
  where $\Gamma \vdash_{\logic{L}} \varphi$ means that the rule $\Gamma \vdash \varphi$ belongs to (holds in, is valid in) the logic $\logic{L}$. If $\Gamma \vdash_{\logic{L}} \varphi$ implies that $\Gamma' \vdash_{\logic{L}} \varphi$ for some finite set of formulas $\Gamma' \subseteq \Gamma$, then $\logic{L}$ is called \emph{finitary}. The \emph{finitary part} of $\logic{L}$ is the finitary logic where $\Gamma \vdash \varphi$ holds if and only if there is some finite set of formulas $\Gamma' \subseteq \Gamma$ such that $\Gamma' \vdash_{\logic{L}} \varphi$. Rules of the form $\emptyset \vdash \varphi$ are called \emph{axiomatic}. A formula~$\varphi$ is a \emph{theorem} of~$\logic{L}$ if $\emptyset \vdash_{\logic{L}} \varphi$. The \emph{trivial logic} is the logic where $\Gamma \vdash \varphi$ holds for each $\Gamma$ and $\varphi$.

  A logic $\logic{L}$ is called an \emph{extension} of a logic $\logic{B}$ (in the same signature), sym\-bolically $\logic{B} \logleq \logic{L}$, if each rule valid in $\logic{B}$ also holds in $\logic{L}$. The extensions of $\logic{B}$ form a complete lattice denoted $\Ext \logic{B}$. The finitary extensions of a finitary logic $\logic{B}$ form an algebraic lattice denoted $\Ext_{\omega} \logic{B}$. If $\logic{L}_{1} \logleq \logic{L}_{2}$, the interval of $\Ext \logic{L}_{1}$ (or $\Ext_{\omega} \logic{L}_{1}$, depending on context) between $\logic{L}_{1}$ and $\logic{L}_{2}$ will be denoted $[\logic{L}_{1}, \logic{L}_{2}]$.

  A logic $\logic{L}$ is \emph{axiomatized} by a set of rules $\rho$ (relative to some logic $\logic{B}$) if it is the least logic which validates each rule of $\rho$ (and extends $\logic{B}$). We also say that $\logic{L}$ is the extension of $\logic{B}$ by the set of rules $\rho$. It is \emph{finitely axiomatizable} (relative to $\logic{B}$) if it is axiomatized (relative to $\logic{B}$) by a finite set of rules. If $\logic{L}_{1}$ and $\logic{L}_{2}$ are extensions of $\logic{B}$ by the sets of rules $\rho_{1}$ and $\rho_{2}$ respectively, then their join $\logic{L}_{1} \vee \logic{L}_{2}$ in $\Ext \logic{B}$ is axiomatized by~$\rho_{1} \cup \rho_{2}$.

  The above notion of axiomatization may be given a more proof-theoretic inter\-pretation. By a \emph{proof} of $\varphi$ from $\Gamma$ using the rules $\rho$, we mean a well-founded tree (i.e.\ a tree with no infinite branches) where the root is labelled by $\varphi$, each terminal node is labelled either by some $\gamma \in \Gamma$ or by a substitution instance of the conclusion of an axiomatic rule in $\rho$, and each non-terminal node is labelled by a formula obtained from the labels of its parents by a substitution instance of a rule in $\rho$. Saying that a logic $\logic{L}$ is axiomatized by $\rho$ is then equivalent to saying that $\Gamma \vdash_{\logic{L}} \varphi$ if and only if $\varphi$ has a proof from $\Gamma$ using the rules $\rho$.

  The models of consequence relations are \emph{(logical) matrices}. A matrix $\matr{A} = \pair{\alg{A}}{F}$ consists of an algebra $\alg{A}$ and a set $F \subseteq \alg{A}$ of \emph{designated values}. A~matrix is \emph{trivial} if $F = \alg{A}$, and it is \emph{almost trivial} if $F = \emptyset$. A \emph{valuation} on $\matr{A}$ is a homo\-morphism from the algebra of formulas into~$\alg{A}$. A rule $\Gamma \vdash \varphi$ is \emph{valid} in~$\matr{A}$ if $v[\Gamma] \subseteq F$ implies $v(\varphi) \in F$ for each valuation $v$ on $\matr{A}$. Each matrix $\matr{A}$ thus determines a logic $\Log \matr{A}$ such that $\Gamma \vdash \varphi$ in $\Log \matr{A}$ if and only if $\Gamma \vdash \varphi$ is valid in $\matr{A}$. If $\class{K}$ is a class of matrices in the given signature, then $\Log \class{K}$ will denote the logic $\bigcap_{\alg{A} \in \class{K}} \Log \alg{A}$. The finitary part of $\Log \class{K}$ will be denoted $\Log_{\omega} \class{K}$.

  A logic $\logic{L}$ is \emph{complete} with respect to a class of matrices $\class{K}$ if $\logic{L} = \Log \class{K}$. Likewise, a finitary logic $\logic{L}$ is \emph{complete as a finitary logic} (or \emph{$\omega$-complete}) with respect to a class of matrices $\class{K}$ if $\logic{L} = \Log_{\omega} \class{K}$, i.e.\ if a finitary rule holds in $\logic{L}$ if and only if it holds in each matrix in $\class{K}$. The two notions coincide if $\class{K}$ is a finite class of finite matrices: in that case $\Log \class{K}$ is always finitary.

  A matrix $\matr{A}$ is a \emph{model} of a logic $\logic{L}$ if each rule of $\logic{L}$ is valid in $\matr{A}$, i.e.\ if $\logic{L} \logleq \Log \matr{A}$. The class of all models of $\logic{L}$ is denoted $\Mod \logic{L}$. Each logic $\logic{L}$ is the logic determined by the class of its models: $\logic{L} = \Log \Mod \logic{L}$. However, $\Mod \logic{L}$ is usually too broad a class to give us a good grip on the properties of $\logic{L}$. Models of a particular kind, called reduced models, will be needed.

  We call a congruence $\theta$ of an algebra $\alg{A}$ \emph{compatible} with $F \subseteq \alg{A}$ if
\begin{align*}
  a \in F \text{ and } \pair{a}{b} \in \theta \text{ implies } b \in F.
\end{align*}
  If $\theta$ is compatible with $F$, we define $F / \theta \assign \set{a / \theta}{a \in F}$, where $a / \theta$ denotes the equivalence class of $a$ with respect to $\theta$. For each $F \subseteq \alg{A}$ there is a largest congruence of $\alg{A}$ compatible with $F$, called the \emph{Leibniz congruence} of $F$ and denoted $\Leibniz{\alg{A}}{F}$. The~Leibniz congruence of a matrix $\matr{A} = \pair{\alg{A}}{F}$ is $\Leibniz{\alg{A}}{F}$ and the \emph{Leibniz reduct} of $\matr{A}$ is the matrix $\matr{A}^{*} \assign \pair{\alg{A} / \Leibniz{\alg{A}}{F}}{ F / \Leibniz{\alg{A}}{F}}$.

  A matrix is called \emph{reduced} if its Leibniz congruence is the identity relation. The Leibniz reduct of $\matr{A}$ is always a reduced matrix. The class of all reduced models of a logic $\logic{L}$ will be denoted $\Mod^{*} \logic{L}$. Crucially, each matrix is \emph{logically equivalent} to (i.e.\ yields the same logic as) its Leibniz reduct. Each logic $\logic{L}$ is thus determined by the class of its reduced models: $\logic{L} = \Log \Mod^{*} \logic{L}$.

  Each matrix is a structure in the sense of model theory, therefore we may define submatrices and products and ultraproducts of matrices in the usual model-theoretic way. More explicitly, consider the matrices $\matr{A} = \pair{\alg{A}}{F}$, $\matr{B} = \pair{\alg{B}}{G}$, and $\matr{A}_{i} = \pair{\alg{A}_{i}}{F_{i}}$ for $i \in I$. Then $\matr{A}$ is a \emph{submatrix} of $\matr{B}$, symbolically $\matr{A} \leq \matr{B}$, if $\alg{A}$ is a subalgebra of $\alg{B}$, symbolically $\alg{A} \leq \alg{B}$, and $F = G \cap \alg{A}$. The matrix $\matr{A}$ is the \emph{direct product} of the matrices $\matr{A}_{i}$ if $\alg{A} = \prod_{i \in I} \alg{A}_{i}$ and $F = \prod_{i \in I} F_{i}$. Given two classes of matrices $\class{K}_{1}$ and $\class{K}_{2}$, the class of all matrices $\matr{A}_{1} \times \matr{A}_{2}$ such that $\matr{A}_{1} \in \class{K}_{1}$ and $\matr{A}_{2} \in \class{K}_{2}$ will be denoted $\class{K}_{1} \times \class{K}_{2}$, with $\matr{A} \times \class{K} \assign \{ \matr{A} \} \times \class{K}$.

  A \emph{matrix homomorphism} $h\colon \matr{A} \rightarrow \matr{B}$ is an algebraic homomorphism $h\colon \alg{A} \to \alg{B}$ such that $h[F] \subseteq G$. It is \emph{strict} if in fact $F = h^{-1}[G]$. If the homomorphism~$h$ is surjective (and strict), we call $\matr{B}$ a (strict) homomorphic image of the \mbox{matrix}~$\matr{A}$, and we call $\matr{A}$ a (strict) homomorphic preimage of~$\matr{B}$.

  If $\class{K}$ is a class of matrices in the given signature, the classes of all homo\-morphic preimages, strict homomorphic images, strict homo\-morphic preimages, sub\-matrices, products, and ultra\-products of \mbox{matrices} in $\class{K}$ will respectively be denoted $\AlgInvH (\class{K})$, $\AlgHS (\class{K})$, $\AlgInvHS (\class{K})$, $\AlgS (\class{K})$, $\AlgP (\class{K})$, and~$\AlgPu (\class{K})$.

  The~class $\Mod \logic{L}$ is always closed under submatrices, products of matrices, strict homomorphic images, and strict homomorphic preimages. Conversely, the following analogue of the $\AlgI \AlgS \AlgP$ theorem for quasi\-varieties characterizes those classes of matrices which arise as $\Mod \logic{L}$ for some finitary logic $\logic{L}$. The~theorem is due to Czelakowski~\cite{czelakowski80} for languages with countably many connectives and to Dellunde \& Jansana~\cite{dellunde+jansana96} for arbitrary languages.

\begin{theorem} \label{thm: equality-free quasivarieties}
  $\Mod \Log_{\omega} \class{K} = \AlgInvHS \AlgHS \AlgS \AlgP \AlgPu (\class{K})$ for any class of matrices $\class{K}$.
\end{theorem}

  In particular, the map $\logic{L} \mapsto \Mod \logic{L}$ is an isomorphism between the lattice $\Ext_{\omega} \logic{B}$ of finitary extensions of a finitary logic $\logic{B}$ and the lattice of classes of models of $\logic{B}$ closed under the appropriate constructions. If the class $\Alg^{*} \logic{B}$ of algebraic reducts of reduced models of $\logic{B}$ moreover generates a locally finite variety, then the following theorem states that this map yields an isomorphism between the lattice of finitary extensions of $\logic{B}$ and a certain lattice of classes of finite reduced models of $\logic{B}$. This theorem is merely the matrix version of a theorem of Gr\"{a}tzer \& Quackenbush~\cite[Theorem~2.3]{gratzer+quackenbush10}. Here $\AlgS^{*} (\class{K})$ and $\AlgP_{\omega}^{*} (\class{K})$ denote respectively the class of all Leibniz reducts of submatrices of matrices in~$\class{K}$ and the class of finite products of matrices in~$\class{K}$. $\Mod_{\omega}^{*} \logic{L}$ denotes the class of all finite reduced models of $\logic{L}$.

\begin{theorem} \label{thm: gratzer-quackenbush}
  Let $\logic{B}$ be a finitary logic such that $\Alg^{*} \logic{B}$ generates a locally finite variety. Then $\Ext_{\omega} \logic{B}$ is dually isomorphic to the lattice of subclasses of $\Mod_{\omega}^{*} \logic{B}$ closed under $\AlgS^{*}$ and $\AlgP_{\omega}^{*}$ via the maps $\logic{L} \mapsto \Mod_{\omega}^{*} \logic{L}$ and $\class{K} \mapsto \Log_{\omega} \class{K}$.
\end{theorem}

\begin{proof}
  The two maps are isotone, and $\logic{L} = \Log_{\omega} \Mod_{\omega}^{*} \logic{L}$ for each $\logic{L}$ in $\Ext_{\omega} \logic{B}$ by the finitarity of $\logic{L}$ and the local finiteness of $\Alg^{*} \logic{B}$. Moreover, $\class{K} \subseteq \Mod_{\omega}^{*} \Log_{\omega} \class{K}$ for each class $\class{K} \subseteq \Mod_{\omega}^{*} \logic{B}$. It therefore remains to prove that $\Mod_{\omega}^{*} \Log_{\omega} \class{K} \subseteq \AlgS^{*} \AlgP^{*} (\class{K})$.

  Let $\matr{A}$ be a finite reduced model of $\Log_{\omega} \class{K}$ of cardinality $n$ with $\class{K} \subseteq \Mod_{\omega}^{*} \logic{B}$. Then $\matr{A} \in \AlgHS \AlgS \AlgP \AlgPu (\class{K})$ by Theorem~\ref{thm: equality-free quasivarieties}. There is an $n$-generated (hence finite) matrix $\matr{B} \in \AlgS \AlgP \AlgPu (\class{K})$ such that $\matr{A} \in \AlgHS (\matr{B})$. The~condition $\matr{B} \in \AlgS \AlgP \AlgPu (\class{K})$ implies that there are finitely many $\matr{C}_{i} \in \AlgPu (\class{K})$ for $i \in I$ with strict homomorphisms $h_{i}\colon \matr{B} \rightarrow \matr{C}_{i}$ such that $\bigcap \set{\Ker h_{i}}{i \in I} = \Delta_{\alg{B}}$, where $\Delta_{\alg{B}}$ is the identity relation on $\alg{B}$ and $\Ker h_{i}$ is the kernel of the homomorphism $h_{i}$. It follows that there are $n$-generated (hence finite) matrices $\matr{D}_{i} \leq \matr{C}_{i}$ such that $\matr{D}_{i}$ is the range of~$h_{i}$. Each such matrix $\matr{D}_{i}$ is an $n$-generated submatrix of an ultra\-product of matrices in $\class{K}$, and therefore embeds into an ultr\-aproduct of $n$-generated sub\-matrices of matrices in $\class{K} \subseteq \Alg^{*} \class{K}$. Since $\Alg^{*} \logic{L}$ generates a locally finite variety, there are only finitely many such $n$-generated submatrices. The matrices $\matr{D}_{i}$ are thus submatrices of matrices in $\class{K}$, i.e.\ $\matr{D}_{i} \in \AlgS (\class{K})$ for $i \in I$. Therefore $\matr{A} \in \AlgHS \AlgS \AlgP_{\omega} \AlgS (\class{K}) \subseteq \AlgHS \AlgS \AlgP_{\omega} (\class{K})$, where $\AlgP_{\omega}$ stands for finite products. Since the matrix $\matr{A}$ is reduced, we in fact have $\matr{A} \in \AlgS^{*} \AlgP_{\omega} (\class{K}) \subseteq \AlgS^{*} \AlgP_{\omega}^{*} (\class{K})$.
\end{proof}

  Finally, we recall some basic notions of universal algebra. The reader may consult the textbook \cite{burris+sankappanavar81} for an intro\-duction to universal algebra and the monograph~\cite{gorbunov98} for an introduction to the study of quasivarieties and antivarieties.

  If $\class{K}$ is a class of algebras in a given signature, then $\AlgInvH (\class{K})$, $\AlgS (\class{K})$, $\AlgP (\class{K})$, and $\AlgPu (\class{K})$ denote the algebraic analogues of the corresponding matrix constructions. We use $\AlgH (\class{K})$ to denote the class of all homomorphic images of algebras in $\class{K}$ and we use $\AlgPuStar(\class{K})$ to denote the class of all ultraproducts of \emph{non-empty} families of algebras in $\class{K}$. An \emph{equation} in a given signature is a formula of the form $t \approx u$, where $t$ and~$u$ are terms in the given signature and $\approx$ is the equality predicate. A \emph{quasiequation} has the form $t_{1} \approx u_{1} ~ \& ~ \dots ~ \& ~ t_{n} \approx u_{n} \implies t \approx u$. Finally, a \emph{negative clause} has the form $t_{1} \napprox u_{1} \vee \dots \vee t_{n} \napprox u_{n}$.

  A~\emph{variety (quasivariety, antivariety)} is a class of algebras in a given signature axiomatized by a set of universally quantified equations (quasi\-equations, negative clauses). The variety (quasivariety, antivariety) \emph{generated} by $\class{K}$ is the smallest variety (quasivariety, antivariety) which contains $\class{K}$.

\newcommand{\auxcit}{\cite[Theorem 2.1.12, Corollary 2.3.4, and Theorem 2.3.11]{gorbunov98}}

\begin{theorem}[\auxcit] \label{thm: generating antivarieties}
  ~
\begin{enumerate}
\item The variety generated by $\class{K}$ is $\AlgH \AlgS \AlgP (\class{K})$.
\item The quasi\-variety generated by $\class{K}$ is $\AlgI \AlgS \AlgP \AlgPu (\class{K})$.
\item The antivariety generated by $\class{K}$ is $\AlgInvH \AlgS \AlgPuStar (\class{K})$.
\end{enumerate}
\end{theorem}

\section{Explosive extensions}
\label{sec: explosive}

  In our study of the lattice of super-Belnap logics, the extensions of Belnap--Dunn logic by an antiaxiomatic (or explosive) rule will play an important role. An antiaxiomatic rule is, roughly speaking, a rule which states that a certain set of propositions is inconsistent. In this section, we study such antiaxiomatic extensions of a given base logic in full generality.

\begin{definition}
  A set of formulas $\Gamma$ is an antitheorem of the logic $\logic{L}$, symbolically $\Gamma \vdash_{\logic{L}} \emptyset$, if no valuation on a non-trivial model of $\logic{L}$ designates all of~$\Gamma$.
\end{definition}

\begin{fact}
  Let $p$ be a variable which does not occur in $\Gamma$. Then $\Gamma$ is an anti\-theorem of $\logic{L}$ if and only if $\Gamma \vdash_{\logic{L}} p$.
\end{fact}

  If all the variables of $\logic{L}$ occur in $\Gamma$, one has to resort to renaming the variables. Pick a variable $p$ and substitutions $\sigmapushp$ and $\sigmapopp$ such that $(\sigmapopp \circ \sigmapushp) (\varphi) =\varphi$ for each formula~$\varphi$, $\sigmapopp(p) = p$, and moreover $p$ does not occur in $\sigmapushp(\varphi)$ for any $\varphi$.

\begin{fact}
  The following are equivalent:
\begin{enumerate}
\item $\Gamma$ is an antitheorem of $\logic{L}$,
\item $\sigmapushp [\Gamma] \vdash_{\logic{L}} p$,
\item $\sigma [\Gamma] \vdash_{\logic{L}} \varphi$ for each formula $\varphi$ and each substitution $\sigma$.
\end{enumerate}
\end{fact}

\begin{proof}
  If $\sigma[\Gamma] \nvdash_{\logic{L}} \varphi$, then there is a model $\pair{\alg{A}}{F}$ of $\logic{L}$ and a valuation $v$ on it such that $v[\sigma[\Gamma]] \subseteq F$ but $v(\varphi) \notin F$. Thus $\pair{\alg{A}}{F}$ is a non-trivial model of $\logic{L}$ and the valuation $w(\varphi) = v(\sigma(\varphi))$ witnesses that $\Gamma$ is not an antitheorem.

  Conversely, if $\Gamma$ is not an antitheorem of $\logic{L}$, then there is a valuation $v$ on some non-trivial model $\pair{\alg{A}}{F}$ of $\logic{L}$ such that $v[\Gamma] \subseteq F$. Consider the valuation $w$ on $\pair{\alg{A}}{F}$ such that $w(p) \notin F$ and $w(q) = v(q)$ otherwise. Then $(w \circ \sigmapopp)[\sigmapushp [\Gamma]] = w[(\sigmapopp \circ \sigmapushp) [\Gamma]] = w[\Gamma] \subseteq F$ while $(w \circ \sigmapopp)(p) = w(p) \notin F$. Thus $\sigmapushp [\Gamma] \nvdash_{\logic{L}} p$.

  The remaining implication is trivial: we instantiate $\sigma$ by $\sigma_{p}$ and $\varphi$ by $p$.
\end{proof}

  When we talk about the \emph{explosive rule} $\Gamma \vdash \emptyset$, we mean the rule $\sigmapushp [\Gamma] \vdash_{\logic{L}} p$. If~$p$ is a variable which does not occur in $\Gamma$, we may identify $\Gamma \vdash \emptyset$ with the rule $\Gamma \vdash p$. For logics which validate the rule $\False \vdash p$ for some constant $\False$, we may identify explosive rules with rules of the form $\Gamma \vdash \False$.

  An \emph{explosive extension} of $\logic{B}$ is an extension of $\logic{B}$ by a set of explosive rules. The~following lemma describes the consequence relation of such an extension.

\begin{lemma} \label{lemma: consequence in antiaxiomatic extensions}
  Let $\logic{L}$ be the extension of $\logic{B}$ by a set of explosive rules $\Delta_{i} \vdash \emptyset$ for $i \in I$. Then $\Gamma \vdash_{\logic{L}} \varphi$ if and only if either $\Gamma \vdash_{\logic{B}} \varphi$ or there is some substitution~$\sigma$ and some $i \in I$ such that $\Gamma \vdash_{\logic{B}} \sigma (\delta)$ for each $\delta \in \Delta_{i}$.
\end{lemma}

\begin{proof}
  The right-to-left direction is obvious, since $\Gamma \vdash_{\logic{L}} \emptyset$ implies $\sigma[\Gamma] \vdash_{\logic{L}} \emptyset$. To prove the opposite direction it suffices to verify that the condition on the right-hand side of the equivalence indeed defines a logic.
\end{proof}

  The explosive part of a logic (relative to some base logic) will turn out to be a very useful construction in the following. Throughout this section, we assume that $\logic{L}$ and $\logic{L}_{i}$ for $i \in I$ are extensions of some base logic $\logic{B}$.

\begin{definition}
  The \emph{explosive part of $\logic{L}$ relative to $\logic{B}$}, denoted $\Exp_{\logic{B}} \logic{L}$, is the logic such that $\Gamma \vdash \varphi$ holds in $\Exp_{\logic{B}} \logic{L}$ if and only if either $\Gamma \vdash_{\logic{B}} \varphi$ or $\Gamma \vdash_{\logic{L}} \emptyset$.
\end{definition}

\begin{fact}
  $\logic{L}$ is an explosive extension of $\logic{B}$ if and only if $\logic{L} = \Exp_{\logic{B}} \logic{L}$.
\end{fact}

  The logic $\Exp_{\logic{B}} \logic{L}$ is the largest extension of $\logic{B}$ by a set of explosive rules which lies below $\logic{L}$. Two extensions of $\logic{B}$ have the same explosive part if and only if they have the same anti\-theorems. Let us now make some basic observations about the explosive part operator $\Exp_{\logic{B}}$.

\begin{fact}
  $\Exp_{\logic{B}}$ is an interior operator on $\Ext \logic{B}$. That is, it is isotone and $\Exp_{\logic{B}} \Exp_{\logic{B}} \logic{L} = \Exp_{\logic{B}} \logic{L} \logleq \logic{L}$.
\end{fact}

\begin{fact}
  $\Exp_{\logic{B}} \bigcap_{i \in I} \logic{L}_{i} = \bigcap_{i \in I} \Exp_{\logic{B}} \logic{L}_{i}$.
\end{fact}

\begin{proof}
  The inequality $\Exp_{\logic{B}} \bigcap_{i \in I} \logic{L}_{i} \logleq \bigcap_{i \in I} \Exp_{\logic{B}} \logic{L}_{i}$ holds because $\Exp_{\logic{B}}$ is an interior operator. Conversely, suppose that $\Gamma \vdash \varphi$ is valid in $\Exp_{\logic{B}} \logic{L}_{i}$ for each $i \in I$. Then either $\Gamma \vdash_{\logic{B}} \varphi$ or $\Gamma \vdash_{\logic{L}_{i}} \emptyset$ for each $i \in I$. But then either $\Gamma \vdash_{\logic{B}} \varphi$ or $\Gamma$ is an antitheorem of $\bigcap_{i \in I} \logic{L}_{i}$.
\end{proof}

\begin{fact}
  $\bigvee_{i \in I} \Exp_{\logic{B}} \logic{L}_{i} = \bigcup_{i \in I} \Exp_{\logic{B}} \logic{L}_{i}$.
\end{fact}

\begin{proof}
  This is an immediate consequence of Lemma~\ref{lemma: consequence in antiaxiomatic extensions}.
\end{proof}

\begin{fact} \label{fact: join with lexp}
  Let $\logic{L}_{1} \logleq \logic{L}_{2}$. Then $\logic{L}_{1} \vee \Exp_{\logic{B}} \logic{L}_{2} = \logic{L}_{1} \cup \Exp_{\logic{B}} \logic{L}_{2}$.
\end{fact}

\begin{proof}
  By Lemma~\ref{lemma: consequence in antiaxiomatic extensions}, $\Gamma \vdash \varphi$ holds in $\logic{L}_{1} \vee \Exp_{\logic{B}} \logic{L}_{2}$ if and only if either $\Gamma \vdash_{\logic{L}_{1}} \varphi$ or there is some antitheorem $\Delta$ of $\logic{L}_{2}$ and some substitution $\sigma$ such that $\Gamma \vdash_{\logic{L}_{1}} \sigma(\delta)$ for each $\delta \in \Delta$. But then $\logic{L}_{1} \logleq \logic{L}_{2}$ implies that $\Gamma$ is an antitheorem of $\logic{L}_{2}$, therefore $\Gamma \vdash \varphi$ holds in $\Exp_{\logic{B}} \logic{L}_{2}$.
\end{proof}

\begin{proposition}
  The explosive extensions of a logic~$\logic{B}$ form a completely distributive complete sublattice of $\Ext \logic{B}$. We denote it $\Exp \Ext \logic{B}$.
\end{proposition}

\begin{proposition}
  The finitary explosive extensions of a finitary logic $\logic{B}$ form an algebraic distributive sublattice of $\Ext_{\omega} \logic{B}$. We denote it $\Exp \Ext_{\omega} \logic{B}$.
\end{proposition}

  The lattices of explosive extensions of any logic $\logic{L}_{0}$ and of its explosive part $\Exp_{\logic{B}} \logic{L}_{0}$ are in fact isomorphic via the maps $\logic{L} \mapsto \Exp_{\logic{B}} \logic{L}$ and $\logic{L} \mapsto \logic{L}_{0} \vee \logic{L} = \logic{L}_{0} \cup \logic{L}$. If $\logic{B}$ and~$\logic{L}_{0}$ are finitary, this isomorphism restricts to an isomorphism between the lattices of finitary explosive extensions of $\logic{L}_{0}$ and $\Exp_{\logic{B}} \logic{L}_{0}$.

\begin{theorem} \label{thm: explosive isomorphism}
  Let $\logic{L}_{0}$ be an extension of $\logic{B}$. Then the lattices $\Exp \Ext \logic{L}_{0}$ and $\Exp \Ext \Exp_{\logic{B}} \logic{L}_{0}$ are isomorphic via the maps $\logic{L} \mapsto \Exp_{\logic{B}} \logic{L}$ and $\logic{L} \mapsto \logic{L}_{0} \cup \logic{L}$.
\end{theorem}

\begin{proof}
  The two maps are isotone and clearly $\logic{L}_{0} \vee \Exp_{\logic{B}} \logic{L} \logleq \logic{L}$ for $\logic{L}$ in $\Exp \Ext \logic{L}_{0}$. Conversely, if $\Gamma \vdash_{\logic{L}} \varphi$ for $\logic{L} \in \Exp \Ext \logic{L}_{0}$, then either $\Gamma \vdash_{\logic{L}_{0}} \varphi$ or $\Gamma \vdash_{\logic{L}} \emptyset$. In either case $\Gamma \vdash \varphi$ holds in $\logic{L}_{0} \vee \Exp_{\logic{B}} \logic{L}$. Thus $\logic{L} = \logic{L}_{0} \vee \Exp_{\logic{B}} \logic{L}$. Fact~\ref{fact: join with lexp} now implies that $\logic{L} = \logic{L}_{0} \cup \Exp_{\logic{B}} \logic{L}$ for each $\logic{L} \in \Exp \Ext \logic{L}_{0}$.

  On the other hand, $\logic{L} \logleq \Exp_{\logic{B}} (\logic{L}_{0} \vee \logic{L})$ for $\logic{L} \in \Exp \Ext (\Exp_{\logic{B}} \logic{L}_{0})$. Conversely, if $\Gamma \vdash \varphi$ holds in $\Exp_{\logic{B}} (\logic{L}_{0} \vee \logic{L})$ for $\logic{L} \in \Exp \Ext (\Exp_{\logic{B}} \logic{L}_{0})$, then either $\Gamma \vdash_{\logic{B}} \varphi$ or $\Gamma \vdash \emptyset$ holds in $\logic{L}_{0} \vee \logic{L}$. But $\logic{L}_{0} \vee \logic{L}$ is the extension of $\logic{L}_{0}$ by the explosive rules $\Gamma \vdash \emptyset$ valid in $\logic{L}$, therefore $\Gamma \vdash_{\logic{L}_{0} \vee \logic{L}} \emptyset$ implies $\Gamma \vdash_{\logic{L}_{0}} \emptyset$ or $\Gamma \vdash_{\logic{L}} \emptyset$. If $\Gamma \vdash \emptyset$ holds in $\logic{L}_{0}$, then it holds in $\Exp_{\logic{B}} \logic{L}_{0}$, and therefore also in $\logic{L}$. Thus $\Gamma \vdash_{\logic{B}} \varphi$ or $\Gamma \vdash_{\logic{L}} \emptyset$. In either case $\Gamma \vdash_{\logic{L}} \varphi$. Thus $\Exp_{\logic{B}} (\logic{L}_{0} \vee \logic{L}) = \logic{L}$ for $\logic{L} \in \Exp \Ext (\Exp_{\logic{B}} \logic{L}_{0})$.
\end{proof}

  Although axiomatizing the intersection of two logics may be a non-trivial task in general, axiomatizing the intersection of an explosive extension of $\logic{B}$ with an arbitrary extension of $\logic{B}$ turns out to be much easier.

\newcommand{\Lexp}{\logic{L}_{\mathrm{exp}}}

  We call two rules $\Gamma \vdash \varphi$~and~$\Delta \vdash \psi$ \emph{variable disjoint} if no propositional atom occurs as a sub\-formula both in $\Gamma \cup \{ \varphi \}$ and in $\Delta \cup \{ \psi \}$. The following two propositions hold, \emph{mutatis mutandis}, for the intersection $\bigcap_{i \in I} \logic{L}_{i}$ of a family $\logic{L}_{i}$ with $i \in I$ of explosive extensions of~$\logic{B}$ (instead of $\logic{L} \cap \Lexp$), provided that we can find axiomatizations $\rho_{i}$ of $\logic{L}_{i}$ such that each rule in $\rho_{i}$ is variable disjoint from each rule in $\rho_{j}$ if $i$ and $j$ are distinct.

\begin{proposition} \label{prop: axiomatizing cap exp}
  Let $\logic{L}$ be an extension of $\logic{B}$ by the rules $\Gamma_{i} \vdash \varphi_{i}$ for $i \in I$. Let $\Lexp$ be an explosive extension of $\logic{B}$ by the rules $\Delta_{j} \vdash \emptyset$ for $j \in I$ such that $\Gamma_{i} \vdash \varphi_{i}$ is variable disjoint from $\Delta_{j} \vdash \emptyset$ for each $i \in I$, $j \in J$. Then $\logic{L} \cap \Lexp$ is the extension of $\logic{B}$ by the rules $\Gamma_{i}, \Delta_{j} \vdash \varphi_{i}$ for $i \in I$, $j \in J$.
\end{proposition}

\begin{proof}
  Clearly $\Gamma_{i}, \Delta_{j} \vdash_{\logic{L} \cap \Lexp} \varphi_{i}$ for each $i \in I$, $j \in J$. Conversely, $\Gamma \vdash_{\Lexp} \varphi$ implies that $\Gamma \vdash_{\logic{B}} \varphi$ or for some $\sigma$ and some $j \in J$ we have $\Gamma \vdash_{\logic{B}} \sigma (\delta_{j})$ for all $\delta_{j} \in \Delta_{j}$, thus $\Gamma \vdash \varphi$ holds in the extension of $\logic{B}$ by the rules $\Gamma, \Delta_{j} \vdash \varphi$ if $\Gamma \vdash_{\logic{L}} \varphi$. But the rule $\Gamma, \Delta_{j} \vdash \varphi$ can be derived from the rules $\Gamma_{i}, \Delta_{j} \vdash \varphi_{i}$ if $\Gamma \vdash_{\logic{L}} \varphi$.
\end{proof}

\begin{proposition} \label{prop: mod cap lexp}
  $\Mod (\logic{L} \cap \Lexp) = \Mod \logic{L} \cup \Mod \Lexp$ for each $\logic{L} \in \Ext \logic{B}$ and $\Lexp \in \Exp \Ext \logic{B}$.
\end{proposition}

\begin{proof}
  Clearly $\Mod \logic{L} \subseteq \Mod (\logic{L} \cap \Lexp)$ and $\Mod \Lexp \subseteq \Mod (\logic{L} \cap \Lexp)$. Conversely, suppose that a non-trivial matrix $\pair{\alg{A}}{F}$ is a model of neither $\logic{L}$ nor $\Lexp$. Then there are rules $\Gamma \vdash \varphi$ and $\Delta \vdash \emptyset$, without loss of generality variable disjoint, and valuations $v$ and $w$ on $\matr{A}$ such that $\Gamma \vdash_{\logic{L}} \varphi$ and $\Delta \vdash_{\Lexp} \emptyset$ and moreover $v[\Gamma] \subseteq F$, $v(\varphi) \notin F$, and $w[\Delta] \subseteq F$. Any valuation $u$ such that $u(p) = v(p)$ if $p$ occurs in $\Gamma$~or~$\varphi$ and $u(p) = w(p)$ if $p$ occurs in $\Delta$ then witnesses that the rule $\Gamma, \Delta \vdash \varphi$, which is valid in $\logic{L} \cap \Lexp$, fails in the matrix $\pair{\alg{A}}{F}$.
\end{proof}

  The logic determined by a product of matrices may be described in terms of the logics determined by the factors and their explosive parts. In the following proposition and its corollaries, the matrices $\matr{A}$, $\class{K}$, and $\matr{A}_{i}$ for $i \in I$ are assumed to be \emph{non-trivial} models of~$\logic{B}$.

\begin{proposition} \label{prop: log of product}
  $\Log \prod_{i \in I} \matr{A}_{i} = \bigcap_{i \in I} \Log \matr{A}_{i} \cup \bigcup_{i \in I} \left( \Exp_{\logic{B}} \Log \matr{A}_{i} \right)$.
\end{proposition}

\begin{proof}
  The right-to-left inclusion is clear. Conversely, suppose that $\Gamma \vdash \varphi$ holds in $\Log \prod_{i \in I} \matr{A}_{i}$. If no valuation on $\matr{A}_{i}$ designates $\Gamma$, then $\Gamma \vdash \varphi$ holds in $\Exp_{\logic{B}} \Log \matr{A}_{i}$. Otherwise, take a valuation $v_{i}$ on $\matr{A}_{i}$ which designates~$\Gamma$ for each~$i \in I$. If there were some $j \in I$ such that $\Gamma \nvdash \varphi$ in $\Log \matr{A}_{j}$, as witnessed by a valuation $w_{j}$, then the product of the valuation $w_{j}$ with the valuations $v_{i}$ for $i \neq j$ would witness that $\Gamma \nvdash \varphi$ in $\Log \prod_{i \in I} \matr{A}_{i}$. Thus $\Gamma \vdash \varphi$ in each $\Log \matr{A}_{i}$.
\end{proof}

  This formula for computing the logic determined by a product of matrices will be used throughout the paper. We recommend that the reader keep it in mind. We now state some of its immediate corollaries.

\begin{corollary}
  $\Log (\class{K} \times \matr{A}) = \Exp_{\logic{B}} \Log \class{K} \cup \Exp_{\logic{B}} \Log \matr{A} \cup (\Log \class{K} \cap \Log \matr{A})$.
\end{corollary}

\begin{proof}
\begin{align*}
  \Log (\class{K} \times \matr{A}) & = \bigcap_{\matr{B} \in \class{K}} \Log \left( \matr{B} \times \matr{A} \right) \\
  & = \bigcap_{\matr{B} \in \class{K}} (\Exp_{\logic{B}} \Log \matr{B} \cup \Exp_{\logic{B}} \Log \matr{A} \cup (\Log \matr{B} \cap \Log \matr{A})) \\
  & = \Exp_{\logic{B}} \Log \matr{B} \cup \left(\bigcap_{\matr{B} \in \class{K}} (\Exp_{\logic{B}} \Log \matr{B} \cup \Log \matr{A}) \cap \bigcap_{\matr{B} \in \class{K}} \Log \matr{B}\right) \\
  & = \Exp_{\logic{B}} \Log \matr{A} \cup \left((\Log \matr{A} \cup \bigcap_{\matr{B} \in \class{K}} \Exp_{\logic{B}} \Log \matr{B}) \cap \bigcap_{\matr{B} \in \class{K}} \Log \matr{B}\right) \\
  & = \Exp_{\logic{B}} \Log \matr{A} \cup \Exp_{\logic{B}} \Log \class{K} \cup (\Log \matr{A} \cap \Log \class{K}).
\end{align*}
\end{proof}

\begin{corollary}
  If $\logic{B} = \Log \matr{A}$ and $\logic{L} = \Log \class{K}$, then $\Exp_{\logic{B}} \logic{L} = \Log (\class{K} \times \matr{A})$.
\end{corollary}

\begin{corollary}
  Let $\logic{L}$ be an explosive extension of $\logic{B}$. Then $\prod_{i \in I} \matr{A}_{i}$ is a model of $\logic{L}$ if and only if $\matr{A}_{i}$ is a model of $\logic{L}$ for some $i \in I$.
\end{corollary}

  The following corollary describes the opposite extreme case. Let us call $\Gamma$ a \emph{potential antitheorem} of $\logic{B}$ if the extension of $\logic{B}$ by $\Gamma \vdash \emptyset$ is a non-trivial logic.

\begin{corollary} \label{cor: potential antitheorems}
  Let $\logic{L}$ be the extension of $\logic{B}$ by a set of rules of the form $\Gamma \vdash \varphi$ where $\Gamma$ is not a potential antitheorem of $\logic{B}$. Then $\prod_{i \in I} \matr{A}_{i}$ is a model of~$\logic{L}$ if and only if each $\matr{A}_{i}$ is a model of $\logic{L}$.
\end{corollary}

\begin{proof}
  A rule $\Gamma \vdash \varphi$ holds in $\Log \prod_{i \in I} \matr{A}_{i}$ if and only if it either holds in each~$\matr{A}_{i}$ or $\Gamma$ is an antitheorem of some $\Log \matr{A}_{i}$. But $\Log \matr{A}_{i}$ is a non-trivial extension of~$\logic{B}$.
\end{proof}

\section{Belnap--Dunn logic and its closest kin}
\label{sec: known super-belnap}

  It is now time to turn our attention from the general theory towards super-Belnap logics. In this section, we review some known facts about Belnap--Dunn logic and its closest relatives.

  The algebraic counterpart of Belnap--Dunn logic is the variety of De~Morgan algebras. A \emph{De~Morgan algebra} $\langle A, \wedge, \vee, \True, \False, \dmneg \rangle$ is a bounded \mbox{distributive} lattice $\langle A, \wedge, \vee, \True, \False \rangle$ with an order-inverting involution~$\dmneg x$ called De~Morgan negation. The~constants $\True$ and $\False$ denote the top and bottom elements respectively.

  The variety of De~Morgan algebras is axiomatized by the equation $\dmneg \dmneg x \approx x$ and either of the De Morgan laws $\dmneg (x \vee y) \approx \dmneg x \wedge \dmneg y$ or $\dmneg (x \wedge y) \approx \dmneg x \vee \dmneg y$ relative to an axiomatization of the variety of bounded distributive lattices.

  The only subdirectly irreducible De~Morgan algebras are the two-element Boolean algebra $\Btwo$, the three-element chain~$\Kthree$ with the unique order-inverting involution, and the four-element diamond $\DMfour$ with the unique order-inverting involution with two fixpoints (see~\cite{kalman58,pynko99c}). Clearly $\Btwo \leq \Kthree \leq \DMfour$.

  Each De~Morgan algebra is therefore a subdirect power of $\DMfour$. Moreover, each \emph{Kleene algebra} (De~Morgan algebra which satisfies $x \wedge \dmneg x \leq y \vee \dmneg y$) is a subdirect power of $\Kthree$, and each \emph{Boolean algebra} (De~Morgan algebra which satisfies $x \wedge \dmneg x \leq y$) is a subdirect power of $\Btwo$. There are no other varieties of De~Morgan algebras, apart from the trivial one.

\begin{figure}
\caption{Some important models of $\BD$}
\label{fig: matrices}

\bigskip

\begin{center}
\begin{tikzpicture}[scale=1,
  dot/.style={circle,fill,inner sep=2.5pt,outer sep=2.5pt}]
  \node (DM4a) at (0,-1) [dot] {};
  \node (DM4b) at (-1,0) [dot] {};
  \node (DM4c) at (1,0) [dot] {};
  \node (DM4d) at (0,1) [dot] {};
  \draw[-] (DM4a) edge (DM4b);
  \draw[-] (DM4a) edge (DM4c);
  \draw[-] (DM4b) edge (DM4d);
  \draw[-] (DM4c) edge (DM4d);
  \draw[rotate around={45:(0.5,0.5)}] (0.5,0.5) ellipse (0.5 and 1);
  \node at (0,-2) {$\BDmatrix$ ($\BD$)};
\end{tikzpicture}
\quad
\begin{tikzpicture}[scale=1,
  dot/.style={circle,fill,inner sep=2.5pt,outer sep=2.5pt}]
  \node (B2a) at (0,0) [dot] {};
  \node (B2b) at (0,2) [dot] {};
  \draw[-] (B2a) edge (B2b);
  \draw (0,2) ellipse (0.5 and 0.5);
  \node at (0,-1) {$\CLmatrix$ ($\CL$)};
\end{tikzpicture}
\quad
\begin{tikzpicture}[scale=1,
  dot/.style={circle,fill,inner sep=2.5pt,outer sep=2.5pt}]
  \node (K3a) at (0,0) [dot] {};
  \node (K3b) at (0,1) [dot] {};
  \node (K3c) at (0,2) [dot] {};
  \draw[-] (K3a) edge (K3b);
  \draw[-] (K3b) edge (K3c);
  \draw (0,2) ellipse (0.5 and 0.5);
  \node at (0,-1) {$\Kmatrix$ ($\K$)};
\end{tikzpicture}
\quad
\begin{tikzpicture}[scale=1,
  dot/.style={circle,fill,inner sep=2.5pt,outer sep=2.5pt}]
  \node (K3a) at (0,0) [dot] {};
  \node (K3b) at (0,1) [dot] {};
  \node (K3c) at (0,2) [dot] {};
  \draw[-] (K3a) edge (K3b);
  \draw[-] (K3b) edge (K3c);
  \draw (0,1.5) ellipse (0.5 and 1);
  \node at (0,-1) {$\LPmatrix$ ($\LP$)};
\end{tikzpicture}
\quad
\begin{tikzpicture}[scale=1,
  dot/.style={circle,fill,inner sep=2.5pt,outer sep=2.5pt}]
  \node (DM4a) at (0,-1) [dot] {};
  \node (DM4b) at (-1,0) [dot] {};
  \node (DM4c) at (1,0) [dot] {};
  \node (DM4d) at (0,1) [dot] {};
  \draw[-] (DM4a) edge (DM4b);
  \draw[-] (DM4a) edge (DM4c);
  \draw[-] (DM4b) edge (DM4d);
  \draw[-] (DM4c) edge (DM4d);
  \draw (0,1) ellipse (0.5 and 0.5);
  \node at (0,-2) {$\ETLmatrix$ ($\ETL$)};
\end{tikzpicture}
\end{center}

\end{figure}
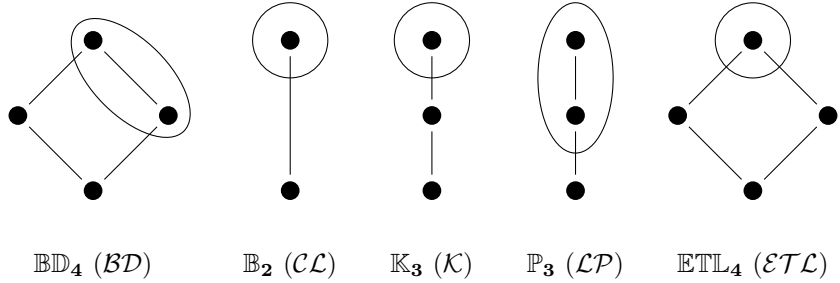

  The best-known super-Belnap logics are determined by matrices over one of the algebras $\Btwo$, $\Kthree$, $\DMfour$ where the designated elements form a lattice filter. These matrices are shown in Figure~\ref{fig: matrices}, where De~Morgan negation is interpreted by reflection across the horizontal axis of symmetry. The logics determined by these matrices are recorded in parentheses.

  \emph{Belnap--Dunn logic} $\BD$ itself is determined by the matrix $\BDmatrix$. The four elements of this matrix can be interpreted as the truth values True, False, Neither (True nor False), and Both (True and False). The matrices $\Kmatrix$ and $\LPmatrix$ are submatrices of $\BDmatrix$: the former drops the truth value Both, the latter drops the truth value Neither. The familiar matrix $\CLmatrix$ which determines classical logic $\CL$ is a submatrix of both $\Kmatrix$ and $\LPmatrix$. It is obtained from $\BDmatrix$ by restricting to the two classical values True and False.

  A~Hilbert-style axiomatization of Belnap--Dunn logic (i.e.\ an axiomatization in the sense of Section~\ref{sec: preliminaries}) was provided independently by Pynko~\cite{pynko95a} and Font~\cite{font97}. Both papers also contain sequent calculi for~$\BD$. More precisely, Pynko and Font study Belnap--Dunn logic without the constants $\True$~and~$\False$. However, to obtain an axiomatization of $\BD$ with the constants it suffices to add the rules
\begin{align*}
  \False \vee p & \vdash p, & \emptyset & \vdash \True, \\
  \dmneg \True \vee p & \vdash p, & \emptyset & \vdash \dmneg \False.
\end{align*}

  We shall see in Section~\ref{sec: different frameworks} that the presence or absence of these constants makes very little difference. One benefit of including them is that doing so collapses the distinction between the trivial logic axiomatized by $\emptyset \vdash p$ and the almost trivial logic axiomatized by $p \vdash q$, as well as between classical logic~$\CL$ and almost classical logic~$\CL^{-}$ where $\Gamma \vdash_{\CL^{-}} \varphi$ if and only if $\Gamma$ is non-empty and $\Gamma \vdash_{\CL} \varphi$.

  \emph{Kleene's strong three-valued logic} $\K$ is determined by the matrix $\Kmatrix$. This logic, or at least the three-valued semantics for its connectives, was introduced by Kleene~\cite{kleene38,kleene52} in connection with partial recursive functions. It was later used by Kripke~\cite{kripke75} in his theory of truth. The logic $\K$ is axiomatized relative to~$\BD$ by $(p \wedge \dmneg p) \vee q \vdash q$, or equivalently by the rule of resolution $p \vee q, \dmneg q \vee r \vdash p \vee r$,  as observed by Rivieccio~\cite{rivieccio12} and proved in~\cite{albuquerque+prenosil+rivieccio17}.

  The \emph{Logic of Paradox} $\LP$ is determined by the matrix $\LPmatrix$. It was introduced by Priest~\cite{priest79}, who proposed to use it to handle semantic paradoxes such as the Liar Paradox. Pynko~\cite{pynko95a} later proved that $\LP$ is axiomatized relative to~$\BD$ by the law of the excluded middle $\emptyset \vdash p \vee \dmneg p$. This~logic is the only non-trivial proper axiomatic extension of $\BD$.

  The intersection of $\LP$ and $\K$ is the logic determined by the set of matrices $\{ \Kmatrix, \LPmatrix \}$. We call it \emph{Kleene's logic of order}, following Rivieccio~\cite{rivieccio12}, and we denote it~$\KO$. This logic was called Kalman implication by Makinson~\cite{makinson73} and studied by Dunn~\cite{dunn76b}, who identified it as the so-called first-degree fragment of the relevance logic R-Mingle. (Recall that $\BD$ itself is the first-degree fragment of the logic of entailment~\cite{dunn76}.) Kleene's logic of order is axiomatized by the rule $(p \wedge \dmneg p) \vee r \vdash (q \vee \dmneg q) \vee r$ relative to $\BD$, as observed by Rivieccio~\cite{rivieccio12} and proved in \cite{albuquerque+prenosil+rivieccio17}. It can also be axiomatized by a rule in two variables, namely $(p \wedge \dmneg p) \vee q \vdash q \vee \dmneg q$.\footnote{Dunn~\cite{dunn00} provides an axiomatization of $\KO$ which relies on a metarule which allows one to infer $\varphi \vee \psi \vdash \chi$ from $\varphi \vdash \psi$ and $\psi \vdash \chi$. It is therefore not a (Hilbert-style) axiomatization in our sense of the word.}

  In addition to the above logics, super-Belnap logics of course also include \emph{classical logic} $\CL$, determined by the matrix $\CLmatrix$. Figure~\ref{fig: classical extensions} shows the super-Belnap logics introduced so far ordered by their logical strength.

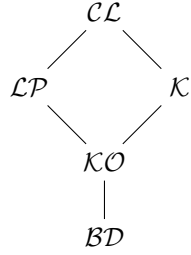
\begin{figure}
\caption{Some super-Belnap logics}
\label{fig: classical extensions}

\begin{center}
\begin{tikzpicture}[scale=1]
  \node (Belnap) at (0,0) {$\BD$};
  \node (KO) at (0,1) {$\KO$};
  \node (LP) at (-1,2) {$\LP$};
  \node (K) at (1,2) {$\K$};
  \node (CL) at (0,3) {$\CL$};
  \draw[-] (Belnap) -- (KO) -- (K) -- (CL);
  \draw[-] (KO) -- (LP) -- (CL);
\end{tikzpicture}
\end{center}

\end{figure}

  A more recent addition to the super-Belnap family is the \emph{Exactly True Logic} $\ETL$ introduced by Pietz \& Rivieccio~\cite{pietz+rivieccio13} as the logic of the matrix~$\ETLmatrix$. It~was also studied by Rivieccio in~\cite{rivieccio12}, where it was denoted $\logic{B}_{1}$.\footnote{The idea of preserving exact truth (truth and non-falsity) had previously been considered by Marcos~\cite{marcos11}, although the signature of his logic was larger than the signature of $\ETL$.} This logic is axiomatized relative to $\BD$ by the rule of disjunctive syllogism $p, \dmneg p \vee q \vdash q$. Classical logic is precisely the extension of $\ETL$ by the law of the excluded middle. That is, $\CL = \LP \vee \ETL$. We shall see that, in a way, this is the canonical decomposition of $\CL$ in the lattice of super-Belnap logics.

  We now review some known properties of these logics, which will be used throughout the paper. The logics $\CL$, $\KO$, and $\BD$ are directly related to the equational theories of Boolean, Kleene, and De~Morgan algebras.

\newlength{\auxlength}
\settowidth{\auxlength}{$\vdash_{\KO}$}

\begin{fact} \label{fact: semilattice based}
  Let $\Gamma$ be a finite set of formulas. Then:
\begin{enumerate}
\item $\Gamma \mathrel{\makebox[\auxlength][c]{$\vdash_{\BD}$}} \varphi$ if and only if \,$\bigwedge \Gamma \leq \varphi$ holds in all De Morgan algebras.
\item $\Gamma \mathrel{\makebox[\auxlength][c]{$\vdash_{\KO}$}} \varphi$ if and only if \,$\bigwedge \Gamma \leq \varphi$ holds in all Kleene algebras.
\item $\Gamma \mathrel{\makebox[\auxlength][c]{$\vdash_{\CL}$}} \varphi$ if and only if \,$\bigwedge \Gamma \leq \varphi$ holds in all Boolean algebras.
\end{enumerate}
\end{fact}

  The following observations follow immediately from the semantic definitions of the logics in question.

\begin{fact} \label{fact: contraposition}
  The logics $\BD$, $\KO$, and $\CL$ enjoy the contraposition property:
\begin{align*}
  \varphi \vdash_{\logic{L}} \psi ~ \implies ~ \dmneg \psi \vdash_{\logic{L}} \dmneg \varphi.
\end{align*}
  The logics $\K$ and $\LP$ are related by contraposition as follows:
\begin{align*}
  \varphi \vdash_{\K} \psi & ~ \implies ~ \dmneg \psi \vdash_{\LP} \dmneg \varphi, & \varphi \vdash_{\LP} \psi & ~ \implies ~ \dmneg \psi \vdash_{\K} \dmneg \varphi.
\end{align*}
\end{fact}

\begin{fact} \label{fact: pcp}
  The logics $\BD$, $\K$, $\LP$, $\KO$, $\CL$ enjoy the proof by cases property:
\begin{align*}
  \Gamma, \varphi \vee \psi \vdash_{\logic{L}} \chi & \iff \Gamma, \varphi \vdash_{\logic{L}} \chi \text{ and } \Gamma, \psi \vdash_{\logic{L}} \chi.
\end{align*}
\end{fact}

  The reduced models of $\BD$ were described by Font~\cite{font97}, and the reduced models of $\ETL$ by Rivieccio~\cite{rivieccio12}. We will only need the following observations here.

\begin{proposition}
  Each reduced model of $\BD$ is a De Morgan algebra with a lattice filter. Conversely, each De Morgan algebra equipped with a lattice filter is a model of $\BD$ (although it need not be a reduced model).
\end{proposition}

\begin{proposition} \label{prop: reduced models of etl}
  Each reduced model of $\ETL$ is a De Morgan algebra with $F = \{ \True \}$. Conversely, each De Morgan algebra equipped with $F = \{ \True \}$ is a model of $\ETL$ (although it need not be a reduced model).
\end{proposition}

  The following propositions shows that consequence in $\LP$, $\K$, and $\ETL$ may be reduced to consequence in $\BD$. Throughout the paper, by (classical) \emph{tautologies} and \emph{contradictions} we mean the tautologies and contradictions of classical logic.

\begin{proposition} \label{prop: consequence in k} \label{prop: consequence in lp} \label{prop: consequence in etl}
  ~
\begin{enumerate}
\item $\Gamma \vdash_{\LP} \varphi$ if and only if $\Gamma, \tau \vdash_{\BD} \varphi$ for some classical tautology $\tau$.
\item $\Gamma \vdash_{\K} \varphi$ if and only if $\Gamma \vdash_{\BD} \varphi \vee \chi$ for some classical contradiction $\chi$.
\item $\Gamma \vdash_{\ETL} \varphi$ if and only if $\Gamma \vdash_{\BD} \psi$ and $\psi \vdash_{\BD} \dmneg \psi \vee \varphi$ for some formula $\psi$.
\end{enumerate}
\end{proposition}

\begin{proof}
  The claim for $\LP$ is equivalent to the fact that $\LP$ is axiomatized by the law of the excluded middle relative to $\BD$. (For a direct semantic proof of the equivalence, see~\cite[Proposition~3.5]{prenosil18thesis}.) The claim for $\K$ then follows from the contraposition relation between $\K$ and $\LP$. Finally, the claim for $\ETL$ was proved by Pietz \& Rivieccio~\cite[Lemma~3.2]{pietz+rivieccio13}.
\end{proof}

  We define conjunctive and disjunctive normal forms of formulas of $\BD$ as in classical logic: a \emph{literal} is an atom or a negated atom, a \emph{conjunctive (disjunctive) clause} is a conjunction (disjunction) of literals, and a formula is in \emph{conjunctive (disjunctive) normal form} if it is a conjunction of disjunctive clauses (a disjunction of conjunctive clauses). The empty conjunction (dis\-junction) is identified with $\True$~($\False$). A clause is \emph{positive} if it does not contain negated atoms.

  Each formula is equivalent in $\BD$ to a formula in conjunctive normal form, and therefore also to a formula in disjunctive normal form (see~\cite{font97}). More precisely, each formula is equivalent to $\True$ or to $\False$ or to a non-empty disjunction (conjunction) of non-empty conjunctions (disjunctions) of literals.

\begin{proposition} \label{prop: consequence in bd}
  Let $\varphi$ be a disjunctive clause. Then $\Gamma \vdash_{\BD} \varphi$ if and only if $\gamma \vdash_{\BD} \varphi$ for some $\gamma \in \Gamma$.
\end{proposition}

\begin{proof}
  Since each formula is equivalent in $\BD$ to a conjunction of disjunctive clauses, we may assume without loss of generality that each formula of $\Gamma$ is a disjunctive clause. If $\gamma \nvdash_{\BD} \varphi$ for each $\gamma \in \Gamma$, then each $\gamma \in \Gamma$ contains a literal which does not occur in $\varphi$. The unique valuation on $\BDmatrix$ which assigns an undesignated value to every literal which occurs in $\varphi$ and a designated value to every other literal then witnesses that $\Gamma \nvdash_{\BD} \varphi$.
\end{proof}

  Unlike in classical logic, the equivalent conjunctive and disjunctive normal form of a formula is essentially unique in $\BD$ (see~\cite[Theorem~3.15]{prenosil18thesis}).

\section{Completeness theorems and explosive parts}
\label{sec:completeness}

\label{sec: completeness}

  In this section, we prove several new completeness theorems for super-Belnap logics. The explosive part operator $\Exp$ turns out to be a useful tool for this purpose.

  Let us first introduce two related sequences of super-Belnap logics. The logic $\ECQ_{n}$ ($\ETL_{n}$) for $n \geq 1$ extends $\BD$ ($\ETL$) by the explosive rule
\begin{equation*}
  (p_{1} \wedge \dmneg p_{1}) \vee \dots \vee (p_{n} \wedge \dmneg p_{n}) \vdash \emptyset.
\end{equation*}
  We use $\ECQ$ as a synonym for $\ECQ_{1}$. Clearly $\ECQ \logleq \ETL$, so $\ETL_{1} = \ETL$. These logics are ordered as follows:
\begin{align*}
  \ECQ_{n} & \logleq \ETL_{n}, & \ECQ_{n} & \logleq \ECQ_{n+1}, & \ETL_{n} & \logleq \ETL_{n+1}.
\end{align*}
  The joins (unions) of these sequences of logics will be denoted $\ECQ_{\omega}$ and $\ETL_{\omega}$:
\begin{align*}
 \ECQ_{\omega} & \assign \bigcup_{n \geq 1} \ECQ_{n}, & \ETL_{\omega} & \assign \bigcup_{n \geq 1} \ETL_{n}.
\end{align*}

  The logics $\ETL_{n}$ and their union $\ETL_{\omega}$ were first introduced by Rivieccio~\cite{rivieccio12} under the names $\logic{B}_{n}$ and $\logic{B}_{\omega}$. Rivieccio provided a completeness theorem for $\ETL_{\omega}$ and proved that $\ETL_{n} <\ETL_{n+1}$. It follows that $\ECQ_{n} < \ECQ_{n+1}$ and that the logics $\ECQ_{\omega}$ and $\ETL_{\omega}$ are not finitely axiomatizable. The inequality $\ECQ_{n} < \ETL_{n}$ also holds: if $p, \dmneg p \vee q \vdash q$ were valid in $\ECQ_{n}$, then we would have either $p, \dmneg p \vee q \vdash_{\BD} q$ or $p, \dmneg p \vee q \vdash_{\ECQ_{n}} \emptyset$. But $\ECQ_{n} \logleq \CL$ and $p, \dmneg p \vee q \nvdash_{\CL} \emptyset$.

  We now determine the explosive parts of $\LP$, $\ETL$, and $\CL$. These logics are finitary, therefore it only suffices to consider antitheorems of the form $\{ \gamma \}$.

\begin{proposition}
  $\Exp_{\BD} \LP = \BD$.
\end{proposition}

\begin{proof}
  If $\gamma \vdash_{\LP} \emptyset$, then $\gamma, \tau \vdash_{\BD} \emptyset$ for some classical tautology~$\tau$ by Proposition~\ref{prop: consequence in lp}, so either $\gamma \vdash_{\BD} \emptyset$ or $\tau \vdash_{\BD} \emptyset$ by Proposition~\ref{prop: consequence in bd}. But $\tau \nvdash_{\BD} \emptyset$ because $\tau \nvdash_{\CL} \emptyset$.
\end{proof}

\begin{proposition}
  $\Exp_{\BD} \ETL = \ECQ$.
\end{proposition}

\begin{proof}
  If $\gamma \vdash_{\ETL} \emptyset$, then $\gamma \vdash_{\ETL} \dmneg \gamma$, so $\gamma \vdash_{\BD} \dmneg \gamma$ by Proposition~\ref{prop: consequence in etl} and $\gamma \vdash_{\ECQ} \emptyset$. Thus $\Exp_{\BD} \ETL \logleq \ECQ$. Conversely, $\ECQ \leq \ETL$.
\end{proof}

\begin{proposition}
  $\Exp_{\BD} \CL = \ECQ_{\omega}$.
\end{proposition}

\begin{proof}
  The inclusion $\ECQ_{\omega} \logleq \Exp_{\BD} \CL$ is clear. Conversely, suppose that $\gamma \vdash_{\CL} \emptyset$. Let $\gamma_{1} \vee \dots \vee \gamma_{n}$ be a disjunction of conjunctive clauses which is equivalent $\gamma$ in $\BD$. Then $\gamma_{i} \vdash_{\CL} \emptyset$ for each~$\gamma_{i}$ by the proof by cases property. It~follows that $\gamma_{i}$ is equivalent in $\BD$ to $p_{i} \wedge \dmneg p_{i} \wedge \varphi_{i}$ for some atom $p_{i}$ and some formula $\varphi_{i}$. Thus $\gamma \vdash_{\BD} (p_{1} \wedge \dmneg p_{1}) \vee \dots \vee (p_{n} \wedge \dmneg p_{n})$ and $\gamma \vdash_{\ECQ_{\omega}} \emptyset$.
\end{proof}

  The following lemma will help us identify classical contradictions. Throughout the paper, we take
\begin{align*}
  \chi_{n} & \assign (p_{1} \wedge \dmneg p_{1}) \vee \dots \vee (p_{n} \wedge \dmneg p_{n}).
\end{align*}

\begin{lemma}
  A formula $\chi$ is a classical contradiction if and only if there is some substitution $\sigma$ such that $\chi \vdash_{\BD} \sigma(\chi_{n})$.
\end{lemma}

\begin{proof}
  This follows from the fact that $\Exp_{\BD} \CL = \ECQ_{\omega}$ by Lemma~\ref{lemma: consequence in antiaxiomatic extensions}.
\end{proof}

\begin{proposition}
  $(\LP \cap \ECQ_{\omega}) \vee \ECQ = \ECQ_{\omega}$.
\end{proposition}

\begin{proof}
  We first prove that $\ECQ_{\omega} \logleq \LP \vee \ECQ$. It suffices to show that the rule $(p \wedge \dmneg p) \vee (q \wedge \dmneg q) \vee r \vdash (\varphi \wedge \dmneg \varphi) \vee r$ is derivable in $\LP$ for some $\varphi$. In~particular, let $\varphi = (p \vee q) \wedge (\dmneg p \vee \dmneg q)$. Then $(\varphi \wedge \dmneg \varphi) \vee r$ is equivalent to the conjunction of the formulas $p \vee q \vee r$, $p \vee \dmneg q \vee r$, $\dmneg p \vee q \vee r$, $\dmneg p \vee \dmneg q \vee r$, $p \vee \dmneg p \vee r$, $q \vee \dmneg q \vee r$. But the last two formulas are theorems of $\LP$ and the rest are derivable from $(p \wedge \dmneg p) \vee (q \wedge \dmneg q)$ in $\BD$.

  It remains to prove that $\ECQ_{\omega} \logleq (\LP \cap \ECQ_{\omega}) \vee \ECQ$. Consider a model $\matr{A}$ of $(\LP \cap \ECQ_{\omega}) \vee \ECQ$. Then $\matr{A}$ is a model of $\ECQ$, as well as a model of either $\LP$ or $\ECQ_{\omega}$ by Proposition~\ref{prop: mod cap lexp}. In the former case, $\matr{A}$ is still a model of $\ECQ_{\omega}$ because $\ECQ_{\omega} \logleq \LP \vee \ECQ$.
\end{proof}

  Cashing in our general observations about the explosive part operator from Section~\ref{sec: explosive}, we can immediately infer that
\begin{align*}
  \LP \vee \ECQ = \LP \vee \ECQ_{\omega} = \LP \vee \Exp_{\BD} \CL = \LP \cup \Exp_{\BD} \CL = \Exp_{\LP} \CL.
\end{align*}
  Similarly, $\LP \cap \ECQ_{\omega} \logleq \KO$ ($\LPmatrix$ is a model of $\LP$ and $\Kmatrix$ of $\ECQ_{\omega}$), therefore
\begin{align*}
  \KO \vee \ECQ = \KO \vee \ECQ_{\omega} = \KO \vee \Exp_{\BD} \CL = \KO \cup \Exp_{\BD} \CL = \Exp_{\KO} \CL.
\end{align*}
  Now recall how a completeness theorem for $\Exp_{\logic{B}} \logic{L}$ is obtained from completeness theorems for $\logic{L} = \Log \matr{A}$ and $\logic{B} = \Log \matr{B}$ if $\logic{L}$ is an extension of $\logic{B}$:
\begin{align*}
  \Exp_{\logic{B}} \logic{L} = \Log \matr{A} \times \matr{B}.
\end{align*}
  We immediately obtain the following batch of completeness theorems. The completeness theorem for $\LP \vee \ECQ$ was already proved by Pynko~\cite{pynko00}.\footnote{A completeness theorem for $\ETL_{\omega}$ was already proved by Rivieccio~\cite{rivieccio12} with respect to the slightly more complicated matrix $\ETLmatrix \times \Kmatrix$. This is no contra\-diction: the logic $\ETL_{\omega}$ is complete with respect to any matrix of the form $\matr{A} \times \matr{B}$ such that $\Log \matr{A} \logleq \ETL_{\omega} \logleq \Log \matr{B} \logleq \CL$. Observe that the algebraic reducts of $\CLmatrix \times \ETLmatrix$ and $\Kmatrix \times \ETLmatrix$ do not generate the same quasivariety. In particular, the algebraic reduct of $\ETLmatrix \times \CLmatrix$ satisfies the quasiequation $x \approx \dmneg x \implies x \approx y$, while the algebraic reduct of $\ETLmatrix \times \Kmatrix$ does not.}

% page break
\pagebreak

\begin{proposition}
  $\ECQ = \Log \ETLmatrix \times \BDmatrix$.
\end{proposition}

\begin{proposition}
  $\ECQ_{\omega} = \Log \CLmatrix \times \BDmatrix$.
\end{proposition}

\begin{proposition}
  $\ETL_{\omega} = \Log \CLmatrix \times \ETLmatrix$.
\end{proposition}

\begin{proposition}
  $\LP \vee \ECQ = \Log \CLmatrix \times \LPmatrix$.
\end{proposition}

\begin{proposition}
  $\KO \vee \ECQ = \Log \{ \CLmatrix \times \LPmatrix, \Kmatrix \}$.
\end{proposition}

\begin{proof}
  This holds because $\LP \vee \ECQ = \Log \CLmatrix \times \LPmatrix$ and $(\LP \vee \ECQ) \cap \K = (\LP \cup \ECQ_{\omega}) \cap \K = (\LP \cap \K) \cup (\ECQ_{\omega} \cap \K) = \KO \cup \ECQ_{\omega} = \KO \vee \ECQ$.
\end{proof}

  The observations that $\Exp_{\BD} \LP = \BD$ and $\Exp_{\BD} \ETL = \ECQ$ are easy to prove but crucial for understanding the structure of the lattice of explosive extensions of $\BD$, as we now show.

\begin{proposition} \label{prop: ecq lp splitting}
  For each $\logic{L} \loggeq \BD$ either $\ECQ \logleq \logic{L}$ or ${\logic{L} \logleq \LP}$.
\end{proposition}

\begin{proof}
  If $\ECQ \nlogleq \logic{L}$, then $\logic{L}$ has a non--trivial reduced model $\pair{\alg{A}}{F}$ such that $a \in F$ for some $a \leq \dmneg a$. But then the three- or four-element submatrix $\False < a \leq \dmneg a  < \True$ of $\pair{\alg{A}}{F}$ is a model of $\logic{L}$. In either case the Leibniz reduct of this submatrix is isomorphic to $\LPmatrix$, therefore $\LPmatrix$ is a model of $\logic{L}$ and $\logic{L} \logleq \LP$.
\end{proof}

\begin{proposition}
  $\ECQ$ is the smallest proper explosive extension of $\BD$.
\end{proposition}

\begin{proof}
  If $\logic{L}$ is an explosive extension of $\BD$ such that $\logic{L} \logleq \LP$, then $\logic{L} = \Exp_{\BD} \logic{L} \logleq \Exp_{\BD} \LP = \BD$.
\end{proof}

  In other words, $\Exp \Ext \ECQ$ is the lattice of proper explosive extensions of~$\BD$. Because $\Exp_{\BD} \ETL = \ECQ$, this lattice is isomorphic to $\Exp \Ext \ETL$.

\begin{theorem} \label{thm: iso exp ext etl}
  The lattices $\Exp \Ext \ECQ$ and $\Exp \Ext \ETL$ are isomorphic via the maps $\logic{L} \mapsto \ETL \vee \logic{L} = \ETL \cup \logic{L}$ and $\logic{L} \mapsto \Exp_{\BD} \logic{L} = \Exp_{\ECQ} \logic{L}$.
\end{theorem}

\begin{proof}
  This is a particular instance of Theorem~\ref{thm: explosive isomorphism}.
\end{proof}

\begin{corollary}
  $\Exp_{\BD} \ETL_{n} = \ECQ_{n}$ and $\ETL_{n} = \ETL \cup \ECQ_{n}$.
\end{corollary}

  $\Exp \Ext \ECQ$ is also isomorphic to the lattice $\LP \cap \Exp \Ext \ECQ$ of all intersections of $\LP$ with an explosive extension of $\ECQ$.

\begin{theorem} \label{thm: iso exp ext ecq}
  The lattices $\Exp \Ext \ECQ$ and $\LP \cap \Exp \Ext \ECQ$ are iso\-morphic via the maps $\logic{L} \mapsto \LP \cap \logic{L}$ and $\logic{L} \mapsto \logic{L} \vee \ECQ$.
\end{theorem}

\begin{proof}
  Consider an extension $\logic{L}$ of $\ECQ$ by the explosive rules $\Gamma_{i} \vdash \emptyset$ for $i \in I$. Intersecting with $\LP$ yields a logic axiomatized by the rules $\Gamma_{i} \vdash p_{i} \vee \dmneg p_{i}$ for $i \in I$, where the atom $p_{i}$ does not occur in $\Gamma_{i}$ by Proposition~\ref{prop: axiomatizing cap exp}. (We rename the variables of~$\Gamma$ if necessary.) But a matrix validates $\Gamma_{i} \vdash p_{i} \vee \dmneg p_{i}$ if and only if it validates either $\Gamma_{i} \vdash \emptyset$ or $\emptyset \vdash p_{i} \vee \dmneg p_{i}$. Since $\ECQ$ is the smallest proper explosive extension of $\BD$, it follows that a matrix validates both $p, \dmneg p \vdash \emptyset$ and $\Gamma_{i} \vdash p_{i} \vee \dmneg p_{i}$ if and only if it validates either $\Gamma_{i} \vdash \emptyset$ or both $\emptyset \vdash p \vee \dmneg p$ and $p, \dmneg p \vdash \emptyset$. But $\logic{L} \logleq \ECQ_{\omega} \logleq \LP \vee \ECQ$, therefore a matrix is a model of both $\ECQ$ and $\LP \cap \logic{L}$ if and only if it is a model of $\logic{L}$, i.e.\ $(\LP \cap \logic{L}) \vee \ECQ = \logic{L}$.
\end{proof}

  In the rest of this section, we prove some more completeness theorems for super-Belnap logics. The methodology used above will not apply here, since these logics will not be identified as explosive parts of other super-Belnap logics.

  We introduce $\Kminus$ as the logic determined by the eight-element matrix $\Kminusmatrix$ shown in Figure~\ref{fig: kminus matrix}, where De~Morgan negation again corresponds to reflection across the horizontal axis of symmetry. This may seem like a very \emph{ad hoc} logic to study at first sight, but we shall see that this logic is one of the two lower covers of $\K$ in $\Ext \BD$ (hence the name), the other being $\KO \vee \ECQ$.

\begin{figure}
\caption{The matrix $\Kminusmatrix$}
\label{fig: kminus matrix}

\bigskip

\begin{center}
\begin{tikzpicture}[scale=1,
  dot/.style={circle,fill,inner sep=2.5pt,outer sep=2.5pt}]
  \node (bot) at (0,0) [dot] {};
  \node (a) at (-1,1) [dot] {};
  \node (b) at (1,1) [dot] {};
  \node (c) at (-2,2) [dot] {};
  \node (d) at (0,2) [dot] {};
  \node (e) at (-1,3) [dot] {};
  \node (f) at (1,3) [dot] {};
  \node (top) at (0,4) [dot] {};
  \node (aname) at (-2.5,2) {$\mathsf{a}$};
  \node (bname) at (1.5,3) {$\mathsf{b}$};
  \node (bname) at (1.5,1) {$\mathsf{c}$};
  \draw[-] (bot) -- (a) -- (c) -- (e) -- (top);
  \draw[-] (bot) -- (b) -- (d) -- (f) -- (top);
  \draw[-] (a) -- (d) -- (e);
  \draw (0,4) ellipse (0.5 and 0.5);
  \node at (0,-1) {$\Kminusmatrix$ ($\Kminus$)};
\end{tikzpicture}
\end{center}

\end{figure}

\begin{proposition}[Consequence in $\Kminus$] \label{prop: consequence in kminus}
  $\Gamma \vdash_{\Kminus} \varphi$ if and only if $\Gamma \vdash_{\BD} \chi \vee \psi$ and $\Gamma \vdash_{\BD} \dmneg \psi \vee \varphi$ for some formula $\psi$ and some classical contradiction $\chi$.
\end{proposition}

\begin{proof}
  Right to left, it suffices to verify that the rule $\chi_{n} \vee q, \dmneg q \vee r \vdash r$ holds in $\Kminusmatrix$ for each $n \geq 1$, where $\chi_{n} \assign (p_{1} \wedge \dmneg p_{1}) \vee \dots \vee (p_{n} \wedge \dmneg p_{n})$. This is true because for each valuation $v$ on $\Kminusmatrix$ we have $v(\chi_{n}) \leq \mathsf{a} \vee \mathsf{c}$, so $v(\chi_{n} \vee q) = \mathsf{t}$ implies $v(q) \geq \mathsf{b}$. But then $v(\dmneg q) \leq \mathsf{c}$, so $v(\dmneg q \vee r) = \True$ implies $v(r) = \True$.

  Conversely, suppose that $\Gamma \vdash_{\Kminus} \varphi$. By finitarity, we may assume that $\Gamma = \{ \gamma \}$. We first prove two auxiliary claims. Firstly, we show that if the left-to-right implication holds for each $\gamma \in \{ \delta_{1}, \delta_{2}, \delta_{3} \}$, where
\begin{align*}
  \delta_{1} & \assign \gamma_{2} \vee \gamma_{3}, &
  \delta_{2} & \assign \gamma_{3} \vee \gamma_{1}, &
  \delta_{3} & \assign \gamma_{1} \vee \gamma_{2}
\end{align*}
  then it holds for $\gamma \assign \gamma_{1} \vee \gamma_{2} \vee \gamma_{3}$. If the implication holds in these three cases, then we have formulas $\psi_{i}$ and classical contradictions $\chi_{i}$ for $1 \leq i \leq 3$ such that
\begin{align*}
  \delta_{i} & \vdash_{\BD} \chi_{i} \vee \psi_{i} & & \text{ and } & \delta_{i} & \vdash_{\BD} \dmneg \psi_{i} \vee \varphi.
\end{align*}
  Observe that
\begin{align*}
  \gamma_{1} & \vdash_{\BD} \delta_{2} \vee \delta_{3}, & \gamma_{2} & \vdash_{\BD} \delta_{3} \vee \delta_{1}, & \gamma_{3} & \vdash_{\BD} \delta_{1} \vee \delta_{2}.
\end{align*}
  Now take
\begin{align*}
  \psi & \assign (\psi_{1} \vee \psi_{2}) \wedge (\psi_{2} \vee \psi_{3}) \wedge (\psi_{3} \vee \psi_{1}), &
  \chi & \assign \chi_{1} \vee \chi_{2} \vee \chi_{3}.
\end{align*}
  Then
\begin{align*}
  \gamma_{1} \vdash_{\BD} \delta_{2} \wedge \delta_{3} \vdash_{\BD} (\chi_{2} \vee \psi_{2}) \wedge (\chi_{3} \vee \psi_{3}) \vdash_{\BD} \chi_{2} \vee \chi_{3} \vee (\psi_{2} \wedge \psi_{3}) \vdash_{\BD} \chi \vee \psi,
\end{align*}
  and likewise $\gamma_{2} \vdash_{\BD} \chi \vee \psi$ and $\gamma_{3} \vdash_{\BD} \chi \vee \psi$. Moreover,
\begin{align*}
  \gamma_{1} \vdash_{\BD} \delta_{2} \wedge \delta_{3} \vdash_{\BD} (\dmneg \psi_{2} \wedge \dmneg \psi_{3}) \vee \varphi \vdash_{\BD} \dmneg \psi \vee \varphi,
\end{align*}
  and likewise $\gamma_{2} \vdash_{\BD} \dmneg \psi \vee \varphi$ and $\gamma_{3} \vdash_{\BD} \dmneg \psi \vee \varphi$. By the proof by cases property for $\BD$ (Fact~\ref{fact: pcp}) we have
\begin{align*}
  \gamma_{1} \vee \gamma_{2} \vee \gamma_{3} & \vdash_{\BD} \chi \vee \psi & & \text{ and } &
  \gamma_{1} \vee \gamma_{2} \vee \gamma_{3} & \vdash_{\BD} \dmneg \psi \vee \varphi,
\end{align*}
  therefore the implication holds for $\gamma \assign \gamma_{1} \vee \gamma_{2} \vee \gamma_{3}$.

  Secondly, we show that if the left-to-right implication holds for $\varphi_{1}$ and  $\varphi_{2}$, then it holds for $\varphi \assign \varphi_{1} \wedge \varphi_{2}$. The assumption yields formulas $\psi_{1}, \psi_{2}$ and contradictions $\chi_{1}$, $\chi_{2}$ such that
\begin{align*}
  \gamma & \vdash_{\BD} \chi_{i} \vee \psi_{i} & & \text{ and } & \gamma & \vdash_{\BD} \dmneg \psi_{i} \vee \varphi_{i}.
\end{align*}
  But then taking $\chi \assign \chi_{1} \vee \chi_{2}$ and $\psi \assign \psi_{1} \wedge \psi_{2}$ yields that
\begin{align*}
  \gamma & \vdash_{\BD} \chi \vee \psi & & \text{ and } & \gamma & \vdash_{\BD} \dmneg \psi \vee \varphi.
\end{align*}

  We now prove the left-to-right implication for arbitrary $\gamma$ using these two auxiliary claims. By the first claim, it suffices to prove the implication for $\gamma \assign \gamma_{1} \vee \gamma_{2}$, where $\gamma_{1}$ and $\gamma_{2}$ are conjunctive clauses. By the second claim, it suffices to prove the implication under the assumption that $\varphi$ is a disjunctive clause.

  If $\gamma$ is a classical contradiction, the implication holds trivially for $\chi \assign \gamma$ and $\psi \assign \dmneg \gamma$. Otherwise, we may suppose that without loss of generality the conjunctive clause $\gamma_{2}$ is not a classical contradiction.

  Suppose now that the right-hand side of the implication fails and $\gamma_{1}$ is not a classical contradiction. Taking $\psi \assign \gamma$, either $\gamma_{1} \nvdash_{\BD} \dmneg \gamma \vee \varphi$ or $\gamma_{2} \nvdash_{\BD} \dmneg \gamma \vee \varphi$. In particular, either $\gamma_{1} \nvdash_{\BD} \varphi$ or $\gamma_{2} \nvdash_{\BD} \varphi$. Suppose without loss of generality that $\gamma_{2} \nvdash_{\BD} \varphi$. Then $\gamma_{2}$ has no literal in common with $\varphi$, therefore there is a valuation $v$ on $\Kminusmatrix$ such that $v(l) = \True$ for each literal $l$ of $\gamma_{2}$ while $v(l) \in \{ \False, \mathsf{b}, \mathsf{c} \}$ for each literal $l$ of $\varphi$. This valuation $v$ witnesses that $\gamma \nvdash_{\Kminus} \varphi$.

  On the other hand, suppose that the right-hand side of the implication fails and $\gamma_{1}$ is a classical contradiction. Taking $\chi \assign \gamma_{1}$ and $\psi \assign \gamma_{2}$, either $\gamma_{1} \nvdash_{\BD} \dmneg \gamma_{2} \vee \varphi$ or $\gamma_{2} \nvdash_{\BD} \dmneg \gamma_{2} \vee \varphi$. The latter case, where $\gamma_{2} \nvdash_{\BD} \varphi$, has already been dealt with. Suppose therefore that $\gamma_{1} \nvdash_{\BD} \dmneg \gamma_{2} \vee \varphi$.

  Now consider the following valuation $v$ on $\Kminusmatrix$. If $p$ and $\dmneg p$ are both literals of $\gamma_{1}$, take $v(p) \assign \mathsf{a}$. If $p$ but not $\dmneg p$ is a literal of $\gamma_{1}$, take $v(p) \assign \True$, while if $\dmneg p$ but not $p$ is a literal of $\gamma_{1}$, take $v(p) \assign \False$. For atoms such that neither $p$ nor $\dmneg p$ is a literal of $\gamma_{1}$, take $v(p) \assign \mathsf{b}$ if $p$ is a literal of $\gamma_{2}$ and $v(p) \assign \mathsf{c}$ if $\dmneg p$ is a literal of $\gamma_{2}$. (These two subcases are mutually exclusive, since $\gamma_{2}$ is not a classical contradiction.) For other atoms $p$ take arbitrary $v(p) \in \{ \mathsf{b}, \mathsf{c} \}$.

  We have $v(\gamma_{1}) = \mathsf{a}$, since $\gamma_{1}$ contains both $p$ and $\dmneg p$ for some atom $p$. Moreover, $v(\gamma_{2}) \in \{ \True, \mathsf{b} \}$, since $\gamma_{2}$ is a conjunction of literals $l$ with ${v(l) \in \{ \True, \mathsf{b} \}}$: if $l$ is a literal of both $\gamma_{1}$ and $\gamma_{2}$, then $\dmneg l$ is not a literal of $\gamma_{1}$, since $\gamma_{1} \nvdash_{\BD} \dmneg \gamma_{2}$. Thus $v(\gamma) = v(\gamma_{1} \vee \gamma_{2}) = \True$. But $v(\varphi) \in \{ \False, \mathsf{b}, \mathsf{c} \}$ because all literals take values in $\{ \False, \mathsf{b}, \mathsf{c}, \True \}$ and no literal of $\varphi$ takes the value $\True$ because $\gamma_{1} \nvdash_{\BD} \varphi$, so $\gamma \nvdash_{\Kminus} \varphi$.
\end{proof}

  A completeness theorem for $\Kminus$ now follows as a corollary. Let $\ETLplus_{n}$ with $n \geq 1$ be the extension of $\BD$ by the rule
\begin{align*}
  \chi_{n} \vee q, \dmneg q \vee r & \vdash r, & \text{where } \chi_{n} \assign (p_{1} \wedge \dmneg p_{1}) \vee \dots (p_{n} \wedge \dmneg p_{n}).
\end{align*}
  We call this rule the $n$-explosive disjunctive syllogism: as special cases it subsumes both the ordinary disjunctive syllogism $p, \dmneg p \vee q \vdash q$ and the rule of \emph{ex contradictione quodlibet} in the form $\chi_{n} \vdash q$.

   In particular, $\ETLplus_{1}$ is axiomatized by the rule $(p \wedge \dmneg p) \vee q, \dmneg q \vee r \vdash r$. Clearly $\ETLplus_{n} \logleq \ETLplus_{n+1}$. The join (union) of this chain of logics will be denoted $\ETLplus_{\omega}$.

\begin{fact}
  $\ETL_{n+1} < \ETLplus_{n}$.
\end{fact}

\begin{proof}
  We have $\chi_{n+1} \vdash_{\BD} \chi_{n} \vee p_{n+1}$ and $\chi_{n+1} \vdash_{\BD} \dmneg p_{n+1} \vee \chi_{n}$, hence $\chi_{n+1} \vdash_{\ETLplus_{n}} \chi_{n}$. But $\ETL_{n} \logleq \ETLplus_{n}$, so $\chi_{n+1} \vdash_{\ETLplus_{n}} \emptyset$. On the other hand, $\ETLplus_{1} \nleq \ETL_{\omega}$ because $\chi_{n} \vee q, \dmneg q \vee r$ is not a classical contradiction.
\end{proof}

  The inequalities $\ETLplus_{n} \logleq \ETLplus_{n+1}$ are in fact strict (Fact~\ref{fact: etlplus separation}). We postpone the proof of this fact until we have the appropriate tools to separate these logics.

\begin{proposition}[Completeness for $\Kminus$] \label{prop: kminus}
  The logic $\Kminus$ is axiomatized by the infinite set of rules ${\chi_{n} \vee q, \dmneg q \vee r \vdash r}$ for $n \geq 1$, i.e.\ $\Kminus = \ETLplus_{\omega}$.
\end{proposition}

  We may also axiomatize the intersections of the logics $\ETL$, $\ETLplus_{n}$, and $\Kminus$ with $\LP$. We shall use the notation $\KOminus \assign \LP \cap \Kminus$. This is because $\KOminus$ turns out to be the only lower cover of $\KO$.

\begin{proposition} \label{prop: lp cap etl}
  $\LP \cap \ETL$ is axiomatized by $p, \dmneg p \vee q \vee \dmneg q \vdash q \vee \dmneg q$.
\end{proposition}

\begin{proof}
  Suppose that $\Gamma \vdash_{\LP \cap \ETL} \varphi$ and $\Gamma \nvdash_{\BD} \varphi$. In $\BD$ the formula $\varphi$ is equivalent to a conjunction of disjunctive clauses $\varphi_{i}$ for $i \in I$. Then $\Gamma, \tau \vdash_{\BD} \varphi$ for some classical tautology $\tau$ by Proposition~\ref{prop: consequence in lp}. But $\Gamma \nvdash_{\BD} \varphi$, so ${\tau \vdash_{\BD} \varphi_{i}}$ by Proposition \ref{prop: consequence in bd}. Each $\varphi_{i}$ is thus a tautology, hence it is equivalent to $q_{i} \vee \dmneg q_{i} \vee \psi_{i}$ for some atom $q_{i}$ and some formula $\psi_{i}$. By Proposition \ref{prop: consequence in etl}, $\Gamma \vdash_{\ETL} \varphi$ implies that $\Gamma \vdash_{\BD} \delta$ and $\Gamma \vdash_{\BD} \dmneg \delta \vee q \vee \dmneg q \vee \psi_{i}$ for some $\delta$. The rule $p, \dmneg p \vee q \vee \dmneg q \vee r \vdash q \vee \dmneg q \vee r$ then suffices to derive $\varphi_{i}$ from $\Gamma$. The equivalence of this rule and the rule $p, \dmneg p \vee q \vee \dmneg q \vdash q \vee \dmneg q$ follows from the fact that in any De Morgan algebra, the elements of the form $q \vee \dmneg q \vee r$ are precisely the elements of the form $q \vee \dmneg q$. Finally, we can derive $\varphi$ from the set of formulas $\set{\varphi_{i}}{i \in I}$ in $\BD$.
\end{proof}

\begin{proposition}[Completeness for $\LP \cap \ETLplus_{n}$]
  The logic $\LP \cap \ETLplus_{n}$ is axiomatized by the infinite set of rules $\chi_{n} \vee q, \dmneg q \vee r \vee \dmneg r \vdash r \vee \dmneg r$ for $n \geq 1$.
\end{proposition}

\begin{proof}
  Suppose that $\Gamma \vdash_{\KOminus} \varphi$ and $\Gamma \nvdash_{\BD} \varphi$. We may assume that $\varphi$ is a disjunctive clause, as in the proof of Proposition~\ref{prop: lp cap etl}. Then $\Gamma, \tau \vdash_{\BD} \varphi$ for some tautology $\tau$ by Proposition~\ref{prop: consequence in lp}, hence $\tau \vdash_{\BD} \varphi$ by Proposition~\ref{prop: consequence in bd}. It~follows that $\varphi$ is equivalent to $r \vee \dmneg r \vee \alpha$ for some atom $r$ and some formula~$\alpha$. Then $\Gamma \vdash_{\Kminus} \varphi$ implies that $\Gamma \vdash_{\BD} \chi \vee \psi$ and $\Gamma \vdash_{\BD} \dmneg \psi \vee r \vee \dmneg r \vee \alpha$ for some formula $\psi$ and some contradiction $\chi$ by Proposition~\ref{prop: consequence in kminus}. The~formula $\varphi$ is thus derivable from $\Gamma$ in the extension of $\BD$ by the rules $\chi_{n} \vee q, \dmneg q \vee r \vee \dmneg r \vee s \vdash r \vee \dmneg r \vee s$ for $n \in \omega$. But these rules are derivable from the simpler rules $\chi_{n} \vee q, \dmneg q \vee r \vee \dmneg r \vdash r \vee \dmneg r$, as in the proof of Proposition~\ref{prop: lp cap etl}.
\end{proof}

\begin{proposition}[Completeness for $\KOminus$] \label{prop: kominus}
  The logic $\KOminus \assign \LP \cap \Kminus$ is axiomatized by the infinite set of rules $\chi_{n} \vee q, \dmneg q \vee r \vee \dmneg r \vdash r \vee \dmneg r$ for $n \geq 1$.
\end{proposition}

\section{The lattice of super-Belnap logics}
\label{sec: super-belnap lattice}

  In this section, we study the global structure of the lattice of super-Belnap logics. We prove that lattice of non-trivial proper extensions of $\BD$ \emph{splits} into the three disjoint intervals $[\LP \cap \ECQ, \LP]$, $[\ECQ, \LP \vee \ECQ]$, and $[\ETL, \CL]$. That is, each logic in $[\BD, \CL]$ lies in exactly one of these intervals. Moreover, the lattice of non-trivial proper extensions of $\ETL$ has the structure $[\ETL_{2}, \Kminus] < \K < \CL$. We then identify some other splittings of $\Ext \BD$ and use these results to list all super-Belnap logics which satisfy various natural metalogical properties.

  We first provide an example which shows that the lattice of super-Belnap logics is non-modular (and therefore non-distributive). Recall that a lattice is \emph{modular} if it satisfies the equation $(a \wedge b) \vee c = (a \vee c) \wedge b$ for $c \leq b$.

\begin{proposition}
  $(\LP \cap \ETL) \vee \ECQ < (\LP \vee \ECQ) \cap \ETL$.
\end{proposition}

\begin{proof}
  By Fact~\ref{fact: join with lexp} we have $(\LP \cap \ETL) \vee \ECQ = (\LP \cap \ETL) \vee \Exp \ETL = (\LP \cap \ETL) \cup \Exp \ETL = (\LP \cap \ETL) \cup \ECQ$ and $(\LP \vee \ECQ) \cap \ETL = (\LP \cup \ECQ_{\omega}) \cap \ETL = (\LP \cap \ETL) \cup (\ECQ_{\omega} \cap \ETL)$. It therefore suffices to find $\Gamma$ and $\varphi$ such that $\Gamma \vdash_{\ECQ_{\omega} \cap \ETL} \varphi$ but $\Gamma \nvdash_{\ECQ} \varphi$ and $\Gamma \nvdash_{\LP} \varphi$.

  The rule $(p_1 \wedge \dmneg p_1) \vee (p_2 \wedge \dmneg p_2), q, \dmneg q \vee r \vdash r$ holds in $\ECQ_{\omega} \cap \ETL$ but not in $\LP = \Log \LPmatrix$. Moreover, $\ECQ = \Log \ETLmatrix \times \BDmatrix$ and there is a valuation on $\ETLmatrix \times \BDmatrix$ which designates $(p_1 \wedge \dmneg p_1) \vee (p_2 \wedge \dmneg p_2)$, therefore the rule in question holds in $\ECQ$ only if $q, \dmneg q \vee r \vdash_{\ECQ} r$. But $q, \dmneg q \vee r \nvdash_{\ECQ} r$.
\end{proof}

\begin{corollary}
  $\Ext_{\omega} \BD$ is not a modular lattice.
\end{corollary}

  We now identify some natural splittings of $\Ext \BD$. We already know that $[\BD, \CL]$ splits into $[\BD, \LP]$ and $[\ECQ, \CL]$ (Proposition~\ref{prop: ecq lp splitting}).

  In the following, $\theta(a, b)$ for $a, b \in \alg{A}$ will denote the principal congruence of the De~Morgan algebra $\alg{A}$ generated by the pair $\pair{a}{b}$. A result of Sankappanavar~\cite{sankappanavar80} states that for $a \leq b$ we have $\pair{x}{y} \in \theta(a, b)$ if and only if the following equations are satisfied:
\begin{align*}
  x \wedge a \wedge \dmneg b & = y \wedge a \wedge \dmneg b, & (x \wedge a) \vee \dmneg a & = (y \wedge a) \vee \dmneg a, \\
  x \vee b \vee \dmneg a & =  y \vee b \vee \dmneg a, & (x \vee b) \wedge \dmneg b & = (y \vee b) \wedge \dmneg b.
\end{align*}

\begin{proposition}
  $\LP \cap \ECQ$ is the smallest proper extension of $\BD$.
\end{proposition}

\begin{proof}
  Suppose that $\LP \cap \ECQ \nlogleq \logic{L}$. Then $p \wedge \dmneg p \nvdash_{\logic{L}} q \vee \dmneg q$, so $\logic{L}$ has a reduced model $\pair{\alg{A}}{F}$ with $a \in F$ and $b \notin F$ such that $a \leq \dmneg a$ and~$\dmneg b \leq b$. The~congruence $\theta(a, \dmneg a)$ is compatible with $F$: if $x \in F$ and $\pair{x}{y} \in \theta(a, \dmneg a)$, then $x \wedge a = y \wedge a$ by the above description of principal congruences on De~Morgan algebras. Since $x \wedge a \in F$, we have $y \in F$. Since the matrix $\pair{\alg{A}}{F}$ is reduced, the congruence $\theta(a, \dmneg a)$ is the identity relation, therefore $a = \dmneg a$.

  Consider the submatrix $\pair{\alg{B}}{G}$ of $\pair{\alg{A}}{F}$ generated by $a$ and $c = (a \wedge b) \vee \dmneg b$. Note that $c = \dmneg c = (a \vee \dmneg b) \wedge b$. The elements $a$ and $c$ are distinct, since $a \in F$ but $c \notin F$ (because $b \notin F$). The universe of $\alg{B}$ is the set $\{ \False, a \wedge c, a, c, a \vee c,  \True \}$ and $G \assign F \cap \alg{B} = \{ a, a \vee c, \True \}$. The congruence $\theta(a \vee c, \True)$ is compatible with $G$ and the matrix $\pair{\alg{B} / \theta(a \vee c, \True)}{G / \theta(a \vee c, \True)}$ is isomorphic to~$\BDmatrix$. Therefore $\BDmatrix$ is a model of $\logic{L}$ and $\logic{L} \logleq \BD$.
\end{proof}

% page break
\pagebreak

\begin{proposition}
  $\CL$ is the largest non-trivial extension of $\BD$.
\end{proposition}

\begin{proof}
  If $\logic{L}$ is a non-trivial extension of $\BD$, then it has a non-trivial reduced model $\pair{\alg{A}}{F}$. Then $\True \in F$ and $\False \notin F$, so the submatrix of $\pair{\alg{A}}{F}$ with the universe $\{ \False, \True \}$ is isomorphic to $\CLmatrix$. Therefore $\CLmatrix$ is a model of $\logic{L}$ and $\logic{L} \logleq \CL$.
\end{proof}

\begin{proposition} \label{prop: k lp splitting}
  The interval $[\BD, \CL]$ splits into $[\BD, \K]$ and $[\LP, \CL]$. All~logics in $[\BD, \K]$ have the same theorems, as do all logics in $[\LP, \CL]$.
\end{proposition}

\begin{proof}
  Suppose that $\LP \nlogleq \logic{L}$. Then $\emptyset \nvdash_{\logic{L}} p \vee \dmneg p$, so $\logic{L}$ has a reduced model $\pair{\alg{A}}{F}$ such that $a \notin F$ for some $a \in \alg{A}$ such that $\dmneg a \leq a$. Consider the submatrix $\pair{\alg{B}}{G}$ of $\pair{\alg{A}}{F}$ generated by $a$. The universe of $\alg{B}$ is the set $\{ \False, \dmneg a, a, \True\}$ and $G \assign F \cap \alg{A} = \{ \True \}$. The congruence $\theta(\dmneg a, a)$ on $\alg{B}$ is compatible with $G$ and the matrix $\pair{\alg{B} / \theta(\dmneg a, a)}{G / \theta(\dmneg a, a)}$ is isomorphic to $\Kmatrix$. Therefore $\Kmatrix$ is a model of $\logic{L}$ and $\logic{L} \logleq \K$.

  The claim that $\LP$ and $\CL$ have the same theorems was proved by Priest~\cite{priest79}. To prove that $\K$ and $\BD$ also have the same theorems, recall the contrapositive relation between $\K$ and $\LP$: $\emptyset \vdash_{\K} \varphi$ implies $\dmneg \varphi \vdash_{\LP} \emptyset$. But $\Exp_{\BD} \LP = \BD$, so $\dmneg \varphi \vdash_{\BD} \emptyset$, and by contraposition $\emptyset \vdash_{\BD} \varphi$.
\end{proof}

  Note that the constants $\True$ and $\False$ are part of our signature, therefore $\BD$ does have theorems. We omit the straightforward verification of the following fact.

\begin{lemma} \label{lemma: three blocks}
  The algebra in Figure~\ref{fig: three blocks} is the free De Morgan algebra generated by $a$~and~$b$ modulo the inequalities $b \leq a$ and $a \leq \dmneg a \vee b$.
\end{lemma}

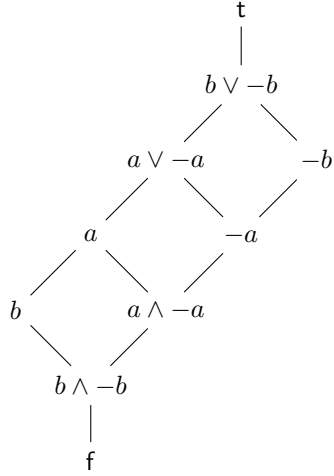
\begin{figure}[t]
\caption{The free two-generated De~Morgan algebra modulo $b \leq a$ and $a \leq \dmneg a \vee b$}
\label{fig: three blocks}

\begin{center}
\begin{tikzpicture}
  \node (bot) at (0,0) {$\False$};
  \node (bnegb) at (0,1) {$b \wedge \dmneg b$};
  \node (anega) at (1,2) {$a \wedge \dmneg a$};
  \node (nega) at (2,3) {$\dmneg a$};
  \node (negb) at (3,4) {$\dmneg b$};
  \node (b) at (-1,2) {$b$};
  \node (a) at (0,3) {$a$};
  \node (aveenega) at (1,4) {$a \vee \dmneg a$};
  \node (bveenegb) at (2,5) {$b \vee \dmneg b$};
  \node (top) at (2,6) {$\True$};
  \draw[-] (bot) -- (bnegb) -- (anega) -- (nega) -- (negb) -- (bveenegb);
  \draw[-] (bnegb) -- (b) -- (a) -- (aveenega) -- (bveenegb) -- (top);
  \draw[-] (anega) -- (a);
  \draw (nega) -- (aveenega);
\end{tikzpicture}
\end{center}

\end{figure}

\begin{proposition}
  $[\BD, \CL]$ splits into $[\BD, \LP \vee \ECQ]$ and $[\ETL, \CL]$.
\end{proposition}

\begin{proof}
  Suppose that $\ETL \nlogleq \logic{L}$. Then $p, \dmneg p \vee q \nvdash_{\logic{L}} q$ and $\logic{L}$ has a reduced model $\pair{\alg{A}}{F}$ such that $a \in F$ and $\dmneg a \vee b \in F$ but $b \notin F$ for some $a, b \in \alg{A}$. Without loss of generality we may take $b \assign a \wedge b$ and $a \assign a \wedge (\dmneg a \vee b)$, i.e.\ we may assume that $b \leq a$ and $a \leq \dmneg a \vee b$.

  Consider the submatrix $\pair{\alg{B}}{G}$ of $\pair{\alg{A}}{F}$ generated by $a$ and $b$. Let $\alg{C}$ be the algebra shown in Figure \ref{fig: three blocks}. By Lemma \ref{lemma: three blocks} there is a homomorphism $h\colon \alg{C} \rightarrow \alg{B}$. Take $H \assign h^{-1}[G]$. Then the matrix $\pair{\alg{C}}{H}$ is a model of $\logic{L}$, being an strict homomorphic preimage of a model of~$\logic{L}$. We have $a \in H$ and $b \notin H$.

  We distinguish two cases. If $a \wedge \dmneg a \notin H$, then $H = \{ a, a \vee \dmneg a, b \vee \dmneg b, \True \}$, hence $\theta(a, a \vee \dmneg a)$ is compatible with $H$ and the matrix $\pair{\alg{C} / \theta(a, a \vee \dmneg a)}{H / \theta(a, a \vee \dmneg a)}$ is isomorphic to $\LPmatrix \times \CLmatrix$. On the other hand, if $a \wedge \dmneg a \in H$, then $H = \{ a, a \vee \dmneg a, \linebreak b \vee \dmneg b, \True, a \wedge \dmneg a, \dmneg a, \dmneg b \}$, hence $\theta(a \wedge \dmneg a, \True)$ is compatible with $H$ and the matrix $\pair{\alg{C} / \theta(a \wedge \dmneg a, \True)}{H / \theta(a \wedge \dmneg a, \True)}$ is isomorphic to $\LPmatrix$. Thus either $\LPmatrix \times \CLmatrix$ is a model of $\logic{L}$ and $\logic{L} \logleq \LP \vee \ECQ$, or $\LPmatrix$ is a model of $\logic{L}$ and $\logic{L} \logleq \LP \logleq \LP \vee \ECQ$.
\end{proof}

  It follows that the join $\CL = \LP \vee \ETL$ is canonical in the following sense.

\begin{corollary}
  If $\CL = \logic{L}_{1} \vee \logic{L}_{2}$ with $\logic{L}_{1} < \CL$ and $\logic{L}_{2} < \CL$, then either $\LP \logleq \logic{L}_{1}$ and $\ETL \logleq \logic{L}_{2}$ or $\ETL \logleq \logic{L}_{1}$ and $\LP \logleq \logic{L}_{2}$.
\end{corollary}

\begin{proof}
  If $\LP \nlogleq \logic{L}_{1}$ and $\LP \nlogleq \logic{L}_{2}$, then $\logic{L}_{1} \vee \logic{L}_{2} \logleq \K$. Likewise, if $\ETL \nlogleq \logic{L}_{1}$ and $\ETL \nlogleq \logic{L}_{2}$, then $\logic{L}_{1} \vee \logic{L}_{2} \logleq \LP \vee \ECQ$. But if $\LP \logleq \logic{L}_{1}$ and $\ETL \logleq \logic{L}_{1}$, then $\CL \logleq \logic{L}_{1}$, and likewise for $\logic{L}_{2}$.
\end{proof}

  Taking the above splittings together yields the following theorem.

\begin{theorem} \label{thm: three blocks}
  Each non-trivial proper extension of $\BD$ lies in one of the disjoint intervals $[\LP \cap \ECQ, \LP]$, $[\ECQ, \LP \vee \ECQ]$, or $[\ETL, \CL]$.
\end{theorem}

\begin{proof}
  These intervals are indeed disjoint: $\ECQ \nlogleq \LP$ and $\ETL \nlogleq \LP \vee \ECQ$ because $p, \dmneg p \vdash \emptyset$ fails in $\LPmatrix$ and $p, \dmneg p \vee q \vdash q$ fails in $\CLmatrix \times \LPmatrix$.
\end{proof}

  Each of these three intervals in fact contains a continuum of finitary logics (Corollary~\ref{cor: continuum in three intervals}). We can also split the lattice of super-Belnap logics into a finite upper part and an infinite lower part. We omit the tedious but straightforward verification of the following claim.

\begin{lemma} \label{lemma: ko splitting}
  The algebra shown in Figure \ref{fig: ko splitting} is the free algebra generated by $a$ and $b$ modulo the inequalities $a \leq \dmneg a$ and $b \leq \dmneg b$.
\end{lemma}

\begin{figure}[!t]
\caption{The free two-generated De~Morgan algebra modulo $a \leq \dmneg a$ and $b \leq \dmneg b$}
\label{fig: ko splitting}

\medskip

\begin{center}
\begin{tikzpicture}[x={(2,1)}, y={(1,2)}, z={(-2,2)}, every node/.style={outer sep=2},scale=0.8]
  \node (bot) at (3.5,-1,0) {$\False$};
  \node (pnegr) at (1.5,0,0) {$a \wedge b$};
  \node (negpnegr) at (3,0,0) {$\dmneg a \wedge b$};
  \node (negr) at (5,0,0) {$b$};
  \node (pr) at (-0.5,1,0) {$a \wedge \dmneg b$};
  \node (alpha) at (1,1,0) {$(a \wedge \dmneg b) \vee (\dmneg a \wedge b)$};
  \node (negrveepr) at (3,1,0) {$b \vee (a \wedge \dmneg b)$};
  \node (negpr) at (1,2,0) {$\dmneg a \wedge \dmneg b$};
  \node (negrveenegpr) at (3,2,0) {$b \vee (\dmneg a \wedge \dmneg b)$};
  \node (r) at (4.5,2,0) {$\dmneg b$};
  \node (p) at (-0.5,1,1) {$a$};
  \node (pveenegpnegr) at (1,1,1) {$a \vee (\dmneg a \wedge b)$};
  \node (pveenegr) at (3,1,1) {$a \vee b$};
  \node (pveenegpr) at (1,2,1) {$a \vee (\dmneg a \wedge \dmneg b)$};
  \node (beta) at (3,2,1) {$a \vee b \vee (\dmneg a \wedge \dmneg b)$};
  \node (pveer) at (4.5,2,1) {$a \vee \dmneg b$};
  \node (negp) at (-1,3,1) {$\dmneg a$};
  \node (negpveenegr) at (1,3,1) {$\dmneg a \vee b$};
  \node (negpveer) at (2.5,3,1) {$\dmneg a \vee \dmneg b$};
  \node (top) at (0.5,4,1) {$\True$};
  \draw[-] (pnegr) -- (negpnegr) -- (negr);
  \draw[-] (pr) -- (alpha) -- (negrveepr);
  \draw[dashed] (negpr) -- (negrveenegpr);
  \draw[-] (negrveenegpr) -- (r);
  \draw[-] (bot) -- (pnegr) -- (pr) -- (p);
  \draw[-] (negpnegr) -- (alpha);
  \draw[dashed] (alpha) -- (negpr) -- (pveenegpr);
  \draw[-] (negr) -- (negrveepr) -- (negrveenegpr) -- (beta);
  \draw[-] (r) -- (pveer) -- (negpveer) -- (top);
  \draw[-] (p) -- (pveenegpnegr) -- (pveenegr);
  \draw[-] (pveenegpr) -- (beta) -- (pveer);
  \draw[-] (negp) -- (negpveenegr) -- (negpveer);
  \draw[-] (alpha) -- (pveenegpnegr) -- (pveenegpr) -- (negp);
  \draw[-] (negrveepr) --(pveenegr) -- (beta) -- (negpveenegr);
\end{tikzpicture}
\end{center}
\end{figure}

  The following proposition extends the unpublished result of Rivieccio that $\K$ is an upper cover of $\Kminus$ (defined semantically).

\begin{proposition} \label{prop: ko splitting}
  $[\BD, \CL]$ splits into $[\BD, \Kminus]$ and $[\KO, \CL]$.
\end{proposition}

\begin{proof}
  Suppose that $\KO \nlogleq \logic{L}$. Then $(p \wedge \dmneg p) \vee r \nvdash_{\logic{L}} (q \vee \dmneg q) \vee r$ and $\logic{L}$ has a reduced model $\pair{\alg{A}}{F}$ such that $a \vee d \in F$ and $c \vee d \notin F$ for some $a, c, d \in \alg{A}$ such that $a \leq \dmneg a$ and $\dmneg c \leq c$. Let $b \assign \dmneg d$. Without loss of generality we may take $d \assign c \vee d$. It follows that $b \leq \dmneg b$ and $\dmneg b \notin F$.

  Consider the submatrix $\pair{\alg{B}}{G}$ of $\pair{\alg{A}}{F}$ generated by the elements $a$ and $b$. Let $\alg{C}$ be the algebra shown in Figure \ref{fig: ko splitting}. By Lemma \ref{lemma: ko splitting} there is a surjective homomorphism $h\colon \alg{C} \rightarrow \alg{B}$. Then the matrix $\pair{\alg{C}}{H}$ with $H \assign h^{-1}[G]$ is a model of $\logic{L}$, being a strict homomorphic preimage of a model of $\logic{L}$. We have $a \vee \dmneg b \in H$ and $\dmneg b \notin H$. The congruence $\theta(\dmneg a \vee \dmneg b, \True)$ is then compatible with $H$, thus $\pair{\alg{D}}{I} \assign \pair{\alg{C} / \theta(\dmneg a \vee \dmneg b, \True)}{H / \theta(\dmneg a \vee \dmneg b, \True)}$ is a model of $\logic{L}$.

  There are now several cases to consider. If $\dmneg a \vee b \notin I$, then the congruence $\theta(a, \dmneg a)$ is compatible with $I$ and $\pair{\alg{D} / \theta(a, \dmneg a)}{I / \theta(a, \dmneg a)}$ is isomorphic to the matrix $\Kminusmatrix$ (recall Figure \ref{fig: kminus matrix}). In that case $\Kminusmatrix$ is a model of $\logic{L}$ and ${\logic{L} \logleq \Kminus}$. On~the other hand, if $a \vee (\dmneg a \wedge b) \in I$, then the rule $p \wedge \dmneg p \vdash q \vee \dmneg q$ fails in $\pair{\alg{D}}{I}$, hence $\LP \cap \ECQ \nlogleq \Log \alg{D}$ and $\logic{L} \logleq \Log \alg{D} = \BD$.

  Finally, if $\dmneg a \vee b \in I$ and $a \vee (\dmneg a \wedge b) \notin I$,  then ${\theta(a \vee b \vee (\dmneg a \wedge \dmneg b), \True)}$ is a congruence compatible with $I$ and it yields either the matrix $\BDmatrix \times \CLmatrix$ or the matrix $\ETLmatrix \times \CLmatrix$. In the former case $\logic{L} \logleq \Log \BDmatrix \times \CLmatrix = \ECQ_{\omega} \logleq \Kminus$, while in the latter case $\logic{L} \logleq \Log \ETLmatrix \times \CLmatrix = \ETL_{\omega} \logleq \Kminus$.
\end{proof}

  Recall the definition of $\KOminus$ as $\LP \cap \Kminus$.

\begin{proposition} \label{prop: kominus vee etl}
  $(\LP \cap \ETLplus_{n}) \vee \ETL = \ETLplus_{n}$. $\KOminus \vee \ETL = \Kminus$.
\end{proposition}

\begin{proof}
  Let $\chi_{n} \assign (p_{1} \wedge \dmneg p_{1}) \vee \dots \vee (p_{n} \wedge \dmneg p_{n})$ and $\Gamma = \{ \chi_{n} \vee q, \dmneg q \vee r \}$. Because $\chi_{n} \vee q, \dmneg (\chi_{n} \vee q) \vee r \vdash_{\ETL} r$, to prove that $\Gamma \vdash q$ holds in $(\LP \cap \ETLplus_{n}) \vee \ETL$ it will suffice to show that $\Gamma \vdash \dmneg (\chi_{n} \vee q) \vee r$ holds in $\LP \cap \ETLplus_{n}$. But $\Gamma \vdash_{\BD} \dmneg q \vee r$, therefore it suffices to show that $\Gamma \vdash p \vee \dmneg p \vee r$ holds in $\LP \cap \ETLplus_{n}$. Let $\psi \assign r \vee (p \wedge \dmneg p)$. Then $\Gamma \vdash_{\BD} \dmneg q \vee \psi \vee \dmneg \alpha$ and $\Gamma \vdash_{\BD} \chi_{n} \vee q$, so $\Gamma \vdash_{\LP \cap \ETLplus_{n}} \psi \vee \dmneg \psi$. But $\psi \vee \dmneg \psi \vdash_{\BD} p \vee \dmneg p \vee r$.
\end{proof}

\begin{theorem} \label{thm: upper part}
  The interval $[\KOminus, \CL]$ splits into the intervals $[\LP, \CL]$, $[\KO, \K]$, and $[\KOminus, \Kminus]$, where
\begin{align*}
  [\LP, \CL] & = \{ \LP, \LP \vee \ECQ, \CL \}, \\ [\KO, \K] & = \{ \KO, \KO \vee \ECQ, \K \}, \\ [\KOminus, \Kminus] & = \{ \KOminus, \KOminus \vee \ECQ, \Kminus \}.
\end{align*}
\end{theorem}

\begin{proof}
  The claim for $[\LP, \CL]$ holds because each non-trivial super-Belnap logic lies in one of the intervals $[\BD, \LP]$, $[\ECQ, \LP \vee \ECQ]$, $[\ETL, \CL]$ and $\CL = \LP \vee \ETL$. The~claim for $[\KO, \K]$ holds because $\KO = \LP \cap \K$ and $\KO \vee \ECQ = (\LP \vee \ECQ) \cap \K$ and $\K = \KO \vee \ETL$, therefore $[\KO, \K] \cap [\BD, \LP] = \{ \KO \}$ and $[\KO, \K] \cap [\ECQ, \LP \vee \ECQ] = \{ \KO \vee \ECQ \}$ and $[\KO, \K] \cap [\ETL, \CL] = \{ \K \}$. Likewise, the claim for $[\KOminus, \Kminus]$ holds because $\KOminus = \LP \cap \Kminus$ and $\KOminus \vee \ECQ = (\LP \vee \ECQ) \cap \Kminus$ and $\KOminus \vee \ETL = \Kminus$. The second equality holds because $(\LP \vee \ECQ) \cap \Kminus = (\LP \cup \ECQ_{\omega}) \cap \Kminus = (\LP \cap \Kminus) \cup \ECQ_{\omega} = \KOminus \vee \ECQ$, since $(\LP \cap \ECQ_{\omega}) \vee \ECQ = \ECQ_{\omega}$ and $\ECQ_{\omega} \logleq \Kminus$.
\end{proof}

  The interval $[\LP, \CL]$ was already described by Pynko~\cite{pynko00}.

\begin{theorem} \label{thm: lower part}
  Each non-trivial proper extension of $\BD$ lies in one of the disjoint intervals $[\KO, \CL]$, $[\LP \cap \ECQ, \KOminus]$, $[\ECQ, \KOminus \vee \ECQ]$, or $[\ETL, \Kminus]$.
\end{theorem}

\begin{proof}
  This follows immediately from Theorems~\ref{thm: three blocks} and~\ref{thm: upper part}.
\end{proof}

  The following proposition extends the unpublished result of Rivieccio that $\ETL_{2}$ is the smallest proper extension of $\ETL$.

\begin{proposition} \label{prop:lp-cap-ecq2-splitting}
  $[\BD, \CL]$ splits into $[\BD, \ETL]$ and $[\LP \cap \ECQ_{2}, \CL]$.
\end{proposition}

\begin{proof}
  Suppose that $\LP \cap \ECQ_{2} \nlogleq \logic{L}$. Then $(p \wedge \dmneg p) \vee (q \wedge \dmneg q) \nvdash_{\logic{L}} r \vee \dmneg r$ by Proposition \ref{prop: axiomatizing cap exp} and $\logic{L}$ has a reduced model $\pair{\alg{A}}{F}$ such that $a \vee b \in F$ for some $a, b \in \alg{A}$ with $a \leq \dmneg a$ and $b \leq \dmneg b$, and $c \notin F$ for some $c \in \alg{A}$ with $\dmneg c \leq c$.

  Consider the submatrix $\pair{\alg{B}}{G}$ of $\pair{\alg{A}}{F}$ generated by the elements $a$ and $b$. Let $\alg{C}$ be the algebra shown in Figure~\ref{fig: ko splitting}. As in the proof of Proposition \ref{prop: ko splitting} there is a surjective homomorphism $h\colon \alg{C} \rightarrow \alg{B}$, therefore $\pair{\alg{C}}{H}$ is a model of~$\logic{L}$ for $H \assign h^{-1}[G]$. We have $a \vee b \in H$.

  If $\dmneg a \in H$ or $\dmneg b \in H$, then there is some $d \in F$ such that $d \leq \dmneg d$: either $d = h(a \vee (\dmneg a \wedge \dmneg b))$ or $d = h(b \vee (a \wedge \dmneg b))$. Since $c \notin F$ for some $c \in \alg{A}$ such that $\dmneg c \leq c$, it follows that the rule $p \wedge \dmneg p \vdash q \vee \dmneg q$ fails in $\pair{\alg{A}}{F}$, hence $\LP \cap \ECQ \nlogleq \Log \pair{\alg{A}}{F}$ and $\logic{L} \logleq \Log \pair{\alg{A}}{F} = \BD$.

  Finally, if $\dmneg a \notin H$ and $\dmneg b \notin H$, then $H$ is the principal filter generated by $a \vee b$ and $\theta(a \vee b, \True)$ is compatible with $H$. The matrix $\pair{\alg{C} / \theta(a \vee b, \True)}{H / \theta(a \vee b, \True)}$ is then isomorphic to $\ETLmatrix$, hence $\ETLmatrix$ is a model of $\logic{L}$ and $\logic{L} \logleq \ETL$.
\end{proof}

\begin{theorem}
  $[\ETL, \CL]$ has the structure $\ETL < [\ETL_{2}, \Kminus] < \K < \CL$.
\end{theorem}

  In other words, the rule schema $\chi \vee p, \dmneg p \vee q \vdash q$, where $\chi$ ranges over all classical contradictions, is the strongest set of rules which lies properly between the disjunctive syllogism and the resolution rule.

\begin{proposition} \label{prop: etl omega splitting}
  For each super-Belnap logic $\logic{L}$ either $\logic{L} \logleq \ETL_{\omega}$ or the rule $(p \wedge \dmneg p) \vee q \vee \dmneg q, (q \wedge \dmneg q) \vee p \vee \dmneg p \vdash p \vee \dmneg p$ holds in $\logic{L}$.
\end{proposition}

\begin{proof}
  Suppose that $(p \wedge \dmneg p) \vee q \vee \dmneg q, (q \wedge \dmneg q) \vee p \vee \dmneg p \nvdash_{\logic{L}} p \vee \dmneg p$. Then $\logic{L}$ has a reduced model $\pair{\alg{A}}{F}$ such that $a \vee \dmneg b \in F$ and $b \vee \dmneg a \in F$ but $\dmneg b \notin F$ for some $a, b \in \alg{A}$ such that $a \leq \dmneg a$ and $b \leq \dmneg b$.

  We proceed as in the proofs of the previous propositions. Again, if we have $a \vee (\dmneg a \wedge b) \in H$, then the rule $p \wedge \dmneg p \vdash q \vee \dmneg q$ fails in $\pair{\alg{C}}{H}$, hence $\logic{L} \logleq \Log \pair{\alg{C}}{H} = \BD$. Suppose therefore that $a \vee (\dmneg a \wedge b) \notin H$. Then $H$ is a principal filter generated either by $a \vee (\dmneg a \wedge \dmneg b)$ or by $a \vee b$ or by $a \vee b \vee (\dmneg a \wedge \dmneg b)$. In the first two cases, the Leibniz reduct of $\pair{\alg{C}}{H}$ is isomorphic to the matrix $\BDmatrix \times \CLmatrix$, while in the third case it is isomorphic to the matrix $\ETLmatrix \times \CLmatrix$. But we know that $\Log \BDmatrix \times \CLmatrix = \ECQ_{\omega} \logleq \ETL_{\omega}$ and $\Log \ETLmatrix \times \CLmatrix = \ETL_{\omega}$. Therefore $\logic{L} \logleq \Log \pair{\alg{C}}{H} \logleq \ETL_{\omega}$.
\end{proof}

\begin{figure}
\caption{Part of the lattice of super-Belnap logics}
\label{fig: ext bd}

\bigskip

\begin{center}
\begin{tikzpicture}[scale=1.15]
  \node (Belnap) at (0,0) {$\BD$};
  \node (LPcapECQ) at (0,1) {$\LP \cap \ECQ$};
  \node (ECQ) at (1,2) {$\ECQ$};
  \node (ETL) at (2,3) {$\ETL$};
  \node (ETLplus) at (2,5) {$\ETLplus_{1}$};
  \node (ETLplus2) at (1,6) {$\ETLplus_{2}$};
  \node (LPcapECQ2) at (-1,2) {$\LP \cap \ECQ_{2}$};
  \node (ECQ2) at (0,3) {$\ECQ_{2}$};
  \node (ETL2) at (1,4) {$\ETL_{2}$};
  \node (dots1) at (-2,3) {$\dots$};
  \node (dots2) at (-1,4) {$\dots$};
  \node (dots3) at (0,5) {$\dots$};
  \node (dots4) at (-0.5,6.5) {$\dots$};
  \node (LPcapECQomega) at (-3,4) {$\LP \cap \ECQ_{\omega}$};
  \node (ECQomega) at (-2,5) {$\ECQ_{\omega}$};
  \node (ETLomega) at (-1,6) {$\ETL_{\omega}$};
  \node (LPcapETLplusomega) at (-4,5) {$\KOminus$};
  \node (LPcapETLplusomegaveeECQ) at (-3,6) {$\KOminus \vee \ECQ$}; 
  \node (ETLplusomega) at (-2,7) {$\Kminus$};
  \node (KO) at (-5,6) {$\KO$};
  \node (KOveeECQ) at (-4,7) {$\KO \vee \ECQ$};
  \node (K) at (-3,8) {$\K$};
  \node (LP) at (-6,7) {$\LP$};
  \node (LPveeECQ) at (-5,8) {$\LP \vee \ECQ$};
  \node (CL) at (-4,9) {$\CL$};
  \draw[-] (Belnap) -- (LPcapECQ) -- (LPcapECQ2) -- (dots1) -- (LPcapECQomega) -- (LPcapETLplusomega) -- (KO) -- (LP);
  \draw[-] (ECQ) -- (ECQ2) -- (dots2) -- (ECQomega) -- (LPcapETLplusomegaveeECQ) -- (KOveeECQ) -- (LPveeECQ);
  \draw[-] (ETL) -- (ETL2) -- (dots3) -- (ETLomega);
  \draw[-] (ETLplus) -- (ETLplus2) -- (dots4) -- (ETLplusomega) -- (K) -- (CL);
  \draw[-] (LPcapECQ) -- (ECQ) -- (ETL);
  \draw[-] (LPcapECQ2) -- (ECQ2) -- (ETL2) -- (ETLplus);
  \draw[-] (LPcapECQomega) -- (ECQomega) -- (ETLomega) -- (ETLplusomega);
  \draw[-] (LPcapETLplusomega) -- (LPcapETLplusomegaveeECQ) -- (ETLplusomega);
  \draw[-] (KO) -- (KOveeECQ) -- (K);
  \draw[-] (LP) -- (LPveeECQ) -- (CL);
\end{tikzpicture}
\end{center}

\end{figure}
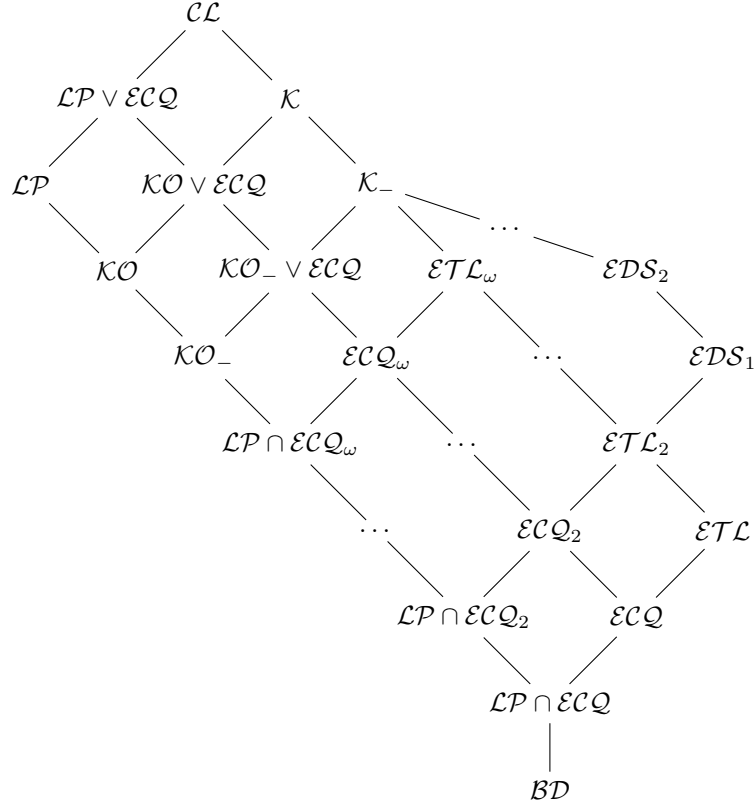

  Proving an informative completeness theorem for the logic axiomatized by the above rule remains an open problem. Apart from this logic, Figure~\ref{fig: ext bd} shows the relative positions of the logics discussed above. It only depicts the inclusions among the selected logics -- it~does not faithfully represent meets and joins.

  The splittings established above can be used to determine which super-Belnap logics satisfy various metalogical properties. This was already done in~\cite{albuquerque+prenosil+rivieccio17} to classify non-trivial super-Belnap logics within the Leibniz and Frege hierarchies of abstract algebraic logic. The only protoalgebraic one is $\CL$, the only Fregean one is also $\CL$, and the only selfextensional ones are $\BD$, $\KO$, and $\CL$.\footnote{Recall that a logic $\logic{L}$ is called \emph{protoalgebraic} if there is a set of formulas in two variables $\Delta(p, q)$ such that $p, \Delta(p, q) \vdash_{\logic{L}} q$ and $\emptyset \vdash_{\logic{L}} \delta(p, p)$ for each $\delta(p, p) \in \Delta(p, p)$. It is \emph{selfextensional} if the equivalence or interderivability relation $\varphi \dashv \vdash_{\logic{L}} \psi$ is a congruence on the algebra of formulas, i.e.\ if replacing a subformula of $\alpha$ by an interderivable formula results in a formula $\beta$ which is interderivable with $\alpha$. The logic is called \emph{Fregean} if this replacement principle  holds for interderivability modulo any set of formulas $\Gamma$, i.e.\ for the relation $\Gamma, \varphi \vdash_{\logic{L}} \psi$ and $\Gamma, \psi \vdash_{\logic{L}} \varphi$.} We~shall provide two alternative proofs of this last fact below.

 Following Cintula \& Noguera~\cite{cintula+noguera13}, we say that a super-Belnap logic $\logic{L}$ enjoys the \emph{(weak) proof by cases property} if (for $\Gamma = \emptyset$)
\begin{align*}
  \Gamma, \varphi \vee \psi \vdash_{\logic{L}} \chi \iff \Gamma, \varphi \vdash_{\logic{L}} \chi \text{ and } \Gamma, \psi \vdash_{\logic{L}} \chi.
\end{align*}

\begin{proposition} \label{prop: pcp}
  The only non-trivial super-Belnap logics which enjoy the (weak) proof by cases property are $\BD$, $\KO$, $\LP$, $\K$, and $\CL$.
\end{proposition}

\begin{proof}
  Let $\logic{L}$ be a proper extension of $\BD$ with the weak proof by cases property. Then $\LP \cap \ECQ \logleq \logic{L}$, i.e.\ $p \wedge \dmneg p \vdash_{\logic{L}} q \vee \dmneg q$. By the weak proof by cases property, $(p \wedge \dmneg p) \vee r \vdash_{\logic{L}} (q \vee \dmneg q) \vee r$, i.e.\ $\KO \logleq \logic{L}$. Moreover, if $\ECQ \logleq \logic{L}$, then $p \wedge \dmneg p \vdash_{\logic{L}} q$ and $(p \wedge \dmneg p) \vee q \vdash_{\logic{L}} q$, i.e.\ $\K \logleq \logic{L}$. But the only non-trivial proper extensions of $\KO$ are $\LP$, $\KO \vee \ECQ$, $\LP \vee \ECQ$, $\K$, and $\CL$. The logics $\KO \vee \ECQ$ and $\LP \vee \ECQ$ do not enjoy the weak proof by cases property: they validate $p \vdash p$ and $(q \wedge \dmneg q) \vdash p$ but not $p \vee (q \wedge \dmneg q) \vdash p$.
\end{proof}

  We say that $\logic{L}$ enjoys the \emph{contraposition property} if $\varphi \vdash_{\logic{L}} \psi$ implies $\dmneg \psi \vdash_{\logic{L}} \dmneg \varphi$.

\begin{proposition} \label{prop:contraposition}
  The only non-trivial super-Belnap logics which enjoy the contra\-position property are $\BD$, $\KO$, and $\CL$.
\end{proposition}

\begin{proof}
  Let $\logic{L}$ be a super-Belnap logic with the contra\-position property. Then $\varphi \vdash_{\logic{L}} \chi$ and $\psi \vdash_{\logic{L}} \chi$ imply $\dmneg \chi \vdash_{\logic{L}} \dmneg \varphi$ and $\dmneg \chi \vdash_{\logic{L}} \dmneg \psi$. Thus $\dmneg \chi \vdash_{\logic{L}} (\dmneg \varphi \wedge \dmneg \psi)$ and $\dmneg (\dmneg \varphi \wedge \dmneg \psi) \vdash_{\logic{L}} \dmneg \dmneg \chi$, hence $\varphi \vee \psi \vdash_{\logic{L}} \chi$. The contraposition property therefore implies the weak proof by cases property. It now suffices to verify that $\K$ and $\LP$ do not satisfy the contra\-position property: we have $p \wedge \dmneg p \vdash_{\K} q$ but $\dmneg q \nvdash_{\K} \dmneg (p \wedge \dmneg p)$, and likewise $q \vdash_{\LP} p \vee \dmneg p$ but $\dmneg (p \vee \dmneg p) \nvdash_{\LP} \dmneg q$.
\end{proof}

\begin{proposition} \label{prop:selfextensional}
  The only non-trivial selfextensional super-Belnap logics are $\BD$, $\KO$, and $\CL$.
\end{proposition}

\begin{proof}
  These logics are selfextensional by virtue of their relation to varieties of Boolean, Kleene, and De~Morgan algebras (Proposition~\ref{fact: semilattice based}). Conversely, if $\logic{L}$ is a selfextensional super-Belnap logic, then it enjoys the contra\-position property: $\varphi \vdash_{\logic{L}} \psi$ implies $\varphi \dashv \vdash_{\logic{L}} \varphi \wedge \psi$, hence $\dmneg (\varphi \wedge \psi) \dashv \vdash_{\logic{L}} \dmneg \varphi$ and $\dmneg \psi \vdash_{\logic{L}} \dmneg (\varphi \wedge \psi) \vdash_{\logic{L}} \dmneg \varphi$. Therefore $\logic{L}$ is one of the logics $\BD$, $\KO$, or $\CL$ by Proposition~\ref{prop:contraposition}.

  Alternatively, suppose that $\logic{L}$ is a selfextensional proper extension of~$\BD$. Then $\LP \cap \ECQ \logleq \logic{L}$, so $(p \wedge \dmneg p) \dashv \vdash_{\logic{L}} (p \wedge \dmneg p) \wedge (q \vee \dmneg q)$. By selfextensionality, $(p \wedge \dmneg p) \vee r \dashv \vdash_{\logic{L}} ((p \wedge \dmneg p) \wedge (q \vee \dmneg q)) \vee r$, therefore $(p \wedge \dmneg p) \vee r \vdash_{\logic{L}} q \vee \dmneg q \vee r$ and $\KO \logleq \logic{L}$. By a similar argument, $\ECQ \logleq \logic{L}$ implies~$\K \logleq \logic{L}$.

  It follows that $\logic{L}$ (if non-trivial) is one of the logics $\KO$, $\LP$, $\K$, $\CL$. It remains to prove that neither $\LP$ nor $\K$ is selfextensional: we have $q \dashv \vdash_{\LP} (p \vee \dmneg p) \vee q$ but $(p \wedge \dmneg p) \vee \dmneg q \nvdash_{\LP} \dmneg q$. Likewise, $q \dashv \vdash_{\K} (p \wedge \dmneg p) \vee q$ but $\dmneg q \nvdash_{\K} p \vee \dmneg p$.
\end{proof}

  A rule $\Gamma \vdash \varphi$ is called \emph{admissible} in a logic $\logic{L}$ if adding it to $\logic{L}$ does not change the set of theorems of $\logic{L}$, or equivalently if for each substitution $\sigma$
\begin{align*}
  \emptyset \vdash_{\logic{L}} \sigma(\gamma) \text{ for each } \gamma \in \Gamma \implies \emptyset \vdash_{\logic{L}} \sigma(\varphi).
\end{align*}
\vskip -0.15pt
% this imperceptible negative vskip was added to make sure he last line would fit on the same page
\noindent A logic $\logic{L}$ is \emph{structurally complete} if each admissible rule of $\logic{L}$ is valid (derivable) in~$\logic{L}$, or equivalently if each proper extension of $\logic{L}$ adds some new theorems to~$\logic{L}$.

\begin{proposition}
  The only non-trivial structurally complete super-Belnap logics are $\K$ and $\CL$.
\end{proposition}

\begin{proof}
  By Proposition~\ref{prop: k lp splitting} each non-trivial super-Belnap logic lies below $\K$ (in which case it has the same theorems as $\BD$) or above $\LP$ (in which case it has the same theorems as $\CL$). Thus $\K$ and $\CL$ are the only non-trivial super-Belnap logics which cannot be properly extended without adding new theorems.
\end{proof}

  While the proof by cases property among super-Belnap logics is only enjoyed by the five logics listed above, we shall see that other super-Belnap logics may satisfy a weaker form of this property. We say that a super-Belnap logic $\logic{L}$ enjoys the \emph{restricted} proof by cases property in case
\begin{align*}
  \Gamma, \varphi \vee \dmneg \varphi \vdash_{\logic{L}} \psi & \iff \Gamma, \varphi \vdash_{\logic{L}} \psi \text{ and } \Gamma, \dmneg \varphi \vdash_{\logic{L}} \psi.
\end{align*}
  For example, $\ETL$ and $\Kminus$ enjoy this property by virtue of the fact that $a \vee \dmneg a = \True$ if and only if $a = \True$ or $\dmneg a = \True$ in $\ETLmatrix$ and $\Kminusmatrix$. We now show that this extends to every extension of $\ETL$ axiomatized by rules of a suitable form.

  The key observation here is that if $a \vee \dmneg a = \True$ in a De~Morgan algebra $\alg{A}$, then $\alg{A}$ is isomorphic to $[\False, a] \times [\False, \dmneg a]$, where the De~Morgan negations on the two intervals are $a \wedge \dmneg x$ and $\dmneg a \wedge \dmneg x$. The isomorphism is given by the maps $x \mapsto \pair{a \wedge x}{\dmneg a \wedge x}$ and $\pair{x}{y} \mapsto x \vee y$.

\begin{proposition}
  Let $\logic{L}$ be an extension of $\ETL$ axiomatized by rules of the form $\Gamma \vdash \varphi$ where $\Gamma$ is not an antitheorem of $\CL$. Then $\logic{L}$ enjoys the restricted proof by cases property.
\end{proposition}

\begin{proof}
   Suppose that $\Gamma, \varphi \vee \dmneg \varphi \nvdash_{\logic{L}} \psi$, as witnessed by a valuation $v$ on a model $\pair{\alg{A}}{\{ \True \}}$ of $\logic{L}$ where $\alg{A}$ is a De~Morgan algebra. Let $a \assign v(\varphi)$. Then $a \vee \dmneg a = \True$, so $\pair{\alg{A}}{\{ \True \}}$ is isomorphic to the binary product of the matrices $\pair{[\False, a]}{\{ a \}}$ and $\pair{[\False, \dmneg a]}{\{ \dmneg a \}}$. By Corollary~\ref{cor: potential antitheorems}, both of these matrices are models of $\logic{L}$ (using the assumption about the axiomatization of $\logic{L}$). This yields two valuations $w_{i} = \pi_{i} \circ v$, where $\pi_{1}(x) \assign a \wedge x$ and $\pi_{2}(x) \assign \dmneg a \wedge x$. Because $v(\psi)$ is not designated in $\pair{\alg{A}}{\{ \True \}}$, it fails to be designated by at least one of these two valuations, say by $w_{1}$. Because each formula in $\Gamma$ is designated in $\pair{\alg{A}}{\{ \True \}}$, it is also designated by $w_{1}$. Finally, $w_{1}(\varphi) = \pi_{1}(v(\varphi)) = \pi_{1}(a) = a$. The~valuation $w_{1}$ thus witnesses that $\Gamma, \varphi \nvdash_{\logic{L}} \psi$.
\end{proof}
\nopagebreak
  In particular, this proposition applies to the logics $\ETLplus_{n}$.

\section{Different frameworks}
\label{sec: different frameworks}

  We now consider what happens if we modify the definition of super-Belnap logics adopted above. The picture only changes marginally if we drop the constants $\True$ and $\False$ from the signature. By contrast, if we use a multiple-conclusion framework instead of the single-conclusion one, then the super-Belnap family reduces to $\BD$, $\KO$, $\K$, $\LP$, and $\CL$.

  Let us first consider \emph{constant-free super-Belnap logics}: extensions of the fragment of $\BD$ without constants. This fragment has no theorems or antitheorems.

\begin{proposition}
  Each super-Belnap logic is axiomatized relative to $\BD$ by a set of rules which do not contain $\True$ and $\False$.
\end{proposition}

\begin{proof}
  Each formula is equivalent in $\BD$ to $\True$ or to $\False$ or to a constant-free formula. (This is an immediate consequence of the fact that each formula can be transformed into an equivalent formula in conjunctive normal form.) But the rule $\Gamma, \True \vdash \varphi$ is equivalent to the rule $\Gamma \vdash \varphi$, the rule $\Gamma \vdash \False$ is equivalent to $\Gamma \vdash p$ for some $p$ not occurring in $\Gamma$ (renaming the variables in $\Gamma$ if necessary), and the rules $\Gamma \vdash \True$ and $\Gamma, \False \vdash \varphi$ hold in $\BD$.
\end{proof}

  Dropping the constants from the signature of $\BD$ means that the undesignated singleton matrix now becomes a submatrix of $\BDmatrix$. Such matrices, where the set of designated elements is empty, will be called \emph{almost trivial}. Each almost trivial matrix determines the \emph{almost trivial logic} axiomatized by $p \vdash q$.

\begin{proposition} \label{prop:constant-free}
  An extension of the constant-free fragment of $\BD$ is the constant-free fragment of an extension of $\BD$ if and only if it is complete with respect to a class of matrices which are not almost trivial.
\end{proposition}

\begin{proof}
  Left to right, no model of $\BD$ is almost trivial. Conversely, let $\logic{L}$ be an extension of the constant-free fragment of $\BD$ complete with respect to a class of matrices which are not almost trivial. Let $\logic{L}_{\True\False}$ be the super-Belnap logic obtained by adding constants $\True$ and $\False$ and the rules $\emptyset \vdash \True$ and $\emptyset \vdash \dmneg \False$ and $\False \vee p \vdash p$ and $\dmneg \True \vee p \vdash p$ to $\logic{L}$. The logic $\logic{L}_{\True\False}$ is a conservative extension of the logic $\logic{L}_{\False}$ which only adds the constant $\False$ and the rules $\emptyset \vdash \dmneg \False$ and $\False \vee p \vdash p$. This is because in each proof in $\logic{L}_{\True\False}$ we can substitute $\dmneg \False$ for $\True$ throughout to obtain a proof in $\logic{L}_{\False}$. We now show that $\logic{L}_{\False}$ is a conservative extension of~$\logic{L}$.

  Let us fix a variable $p$. It will suffice to prove that $\Gamma \vdash_{\logic{L}_{\False}} \varphi$ implies $\Gamma \vdash_{\logic{L}} \varphi$ for $\Gamma$ and $\varphi$ where $p$ does not occur as a subformula. To see this, consider substitutions $\sigmapushp$ and $\sigmapopp$ such that $(\sigmapopp \circ \sigmapushp) (\varphi) = \varphi$ for each $\varphi$ and moreover $\sigmapushp(\varphi)$ never contains $p$ as a subformula. If $\Gamma \vdash_{\logic{L}_{\False}} \varphi$, then $\sigmapushp[\Gamma] \vdash_{\logic{L}_{\False}} \sigmapushp (\varphi)$, hence, supposing that conservativity holds in the special case where $p$ does not occur in $\Gamma$ and $\varphi$, $\sigmapushp[\Gamma] \vdash_{\logic{L}} \sigmapushp(\varphi)$ and $\Gamma = (\sigmapopp \circ \sigmapushp)[\Gamma] \vdash_{\logic{L}} (\sigmapopp \circ \sigmapushp) (\varphi) = \varphi$.

  Now consider a proof of $\varphi$ from $\Gamma$ in $\logic{L}_{\False}$, where $\Gamma$ and $\varphi$ are constant-free and the variable $p$ does not occur in $\Gamma$ and $\varphi$. Then $\varphi$ has a proof from $\Gamma \cup \{ \dmneg \False \}$ which only uses the (constant-free) rules of $\logic{L}$ and instances the rule $\False \vee p \vdash p$. We now prove that all applications of this rule are redundant.

  Suppose therefore that $\False \vee \psi$ has a proof from $\Gamma \cup \{ \dmneg \False \}$ in $\logic{L}_{\False}$ which only uses the rules of~$\logic{L}$. Then $\dmneg p, \Gamma \vdash_{\logic{L}} p \vee \psi$, where $p$ does not occur in $\Gamma$ or~$\psi$, since we can uniformly replace $\False$ by $p$ in the proof (each step of the proof is a substitution instance of a constant-free rule). We need to show that in fact $\Gamma \vdash_{\logic{L}} \psi$.

  If $\Gamma \nvdash_{\logic{L}} \psi$, there is a non-trivial model $\pair{\alg{A}}{F}$ of $\logic{L}$ and a valuation $v$ on $\alg{A}$ such that $v[\Gamma] \subseteq F$ and $v(\psi) \notin F$. Since $\logic{L}$ is complete with respect to a class of matrices which are not almost trivial, we may assume that $\pair{\alg{A}}{F}$ is not almost trivial. Let $v(\psi) = a$. It suffices to find $b \in \alg{A}$ such that $\dmneg b \in F$ and $b \vee a \notin F$, since taking $v(p) = b$ then witnesses that $\dmneg p, \Gamma \nvdash_{\logic{L}} p \vee \psi$. But we can always take $b = a \wedge \dmneg c$ for some $c \in F$ (which exists because $\pair{\alg{A}}{F}$ is not almost trivial).

  This proves, by induction over well-founded trees, that every proof of $\varphi$ from $\Gamma$ in $\logic{L}_{\False}$ can be transformed into a proof of $\varphi$ from $\Gamma$ in $\logic{L}$.
\end{proof}

 Intersecting the almost trivial logic (axiomatized by the rule $p \vdash q$) with the constant-free fragments of $\LP$, $\LP \vee \ECQ$, and $\CL$ yields the following three logics: $\LP^{-}$, axiomatized by $p \vdash q \vee \dmneg q$ relative to the constant-free fragment of~$\BD$, $\LP^{-} \vee \ECQ$, and $\CL^{-} = \LP^{-} \vee \ETL$. We can observe that $\Gamma \vdash_{\LP^{-}} \varphi$ if and only if $\Gamma$ is non-empty and $\Gamma \vdash_{\LP} \varphi$, and likewise for $\LP^{-} \vee \ECQ$ and $\CL^{-}$.

\begin{theorem} \label{thm:constant-free}
  There are exactly four extensions of the constant-free fragment of $\BD$ which are not constant-free fragments of extensions of $\BD$, namely the logics $\LP^{-}$, $\LP^{-} \vee \ECQ$, $\CL^{-} = \LP^{-} \vee \ETL$, and the almost trivial logic.
\end{theorem}

\begin{proof}
  Let $\logic{L}$ be a constant-free super-Belnap logic which is neither trivial nor almost trivial. If $\logic{L}$ lies below the constant-free reduct of $\K$, then the constant-free reduct of $\Kmatrix$ is a model of $\logic{L}$. An almost trivial matrix is a submatrix of this reduct. Thus each constant-free super-Belnap logic $\logic{L}$ below the constant-free fragment of $\K$ is complete with respect to a class of matrices which are not almost trivial. It~thus constitutes the constant-free fragment of some super-Belnap logic by the previous proposition.

  On the other hand, if $\logic{L}$ does not lie below the constant-free reduct of $\K$, then $\LP^{-} \logleq \logic{L}$. This is because if $\logic{L}$ invalidates the rule $p \vdash_{\logic{L}} q \vee \dmneg q$, then the constant-free reduct of $\Kmatrix$ is a model of $\logic{L}$. (The argument is identical to the proof that $\LP$ and $\K$ form a splitting pair.) But each model of $\LP^{-}$ which is not almost trivial validates the rule $\emptyset \vdash p \vee \dmneg p$, i.e.\ it is a model of the constant-free fragment of $\LP$. Each extension of $\LP^{-}$ is thus either an extension of the constant-free fragment of $\LP$ or its intersection with the almost trivial logic. But it was shown by Pynko~\cite{pynko00} that the only non-trivial extensions of the constant-free fragment of $\LP$ are the constant-free fragments of $\LP$, $\LP \vee \ECQ$, and $\CL$.
\end{proof}

  In the constant-free framework, the lattice of super-Belnap logics therefore has two co-atoms, namely classical logic and the almost trivial logic.

  Moving to a multiple-conclusion setting has more profound consequences. \mbox{Recall} (e.g.~from~\cite{shoesmith+smiley78}) that a \emph{multiple-conclusion consequence relation} is a \mbox{relation} between sets of formulas, written $\Gamma \vdash \Delta$, which satisfies the following:
\begin{itemize}
\item $\varphi \vdash_{\logic{L}} \varphi$ (reflexivity),
\item if $\Gamma \vdash_{\logic{L}} \Delta$, then $\Gamma, \Gamma' \vdash_{\logic{L}} \Delta, \Delta'$ (monotonicity),
\item if $\Gamma, \Phi_{1} \vdash_{\logic{L}} \Delta$ and $\Gamma \vdash_{\logic{L}} \Delta, \Phi_{2}$ whenever $\Phi_{1} \cup \Phi_{2} = \Phi$, then $\Gamma \vdash_{\logic{L}} \Delta$ (cut),
\item if $\Gamma \vdash_{\logic{L}} \Delta$, then $\sigma[\Gamma] \vdash_{\logic{L}} \sigma[\Delta]$ for each substitution $\sigma$ (structurality).
\end{itemize}
  The multiple-conclusion logic determined by a class of matrices $\class{K}$ is defined as expected. That is, $\Gamma \vdash \Delta$ if and only if for each valuation $v$ on a matrix $\pair{\alg{A}}{F} \in \class{K}$ we have that $v(\delta) \in F$ for some $\delta \in \Delta$ whenever $v[\Gamma] \subseteq F$.

  By the multiple-conclusion versions of $\BD$, $\LP$, $\K$, $\CL$, denoted $\BDmc$, $\LPmc$, $\Kmc$, $\CLmc$, we mean the multiple-conclusion logics defined semantically via the matrices $\BDmatrix$, $\LPmatrix$, $\Kmatrix$, and $\CLmatrix$. The multiple-conclusion version of $\KO$ is defined as $\KOmc = \LPmc \cap \Kmc$. The multiple-conclusion version of the trivial logic is defined as the logic axiomatized by the rule $\emptyset \vdash \emptyset$. (Note that this logic is only complete with respect to the empty class of matrices.)

  The designated sets of the above finite matrices form prime filters, therefore $\Gamma \vdash \Delta$ holds in the multiple-conclusion version of one of the logics above if and only if $\Gamma \vdash \bigvee \Delta'$ holds in the single-conclusion version for some finite $\Delta' \subseteq \Delta$. We can infer that the logic $\BDmc$ is axiomatized by the (``positive'') rules
\begin{align*}
  & p, q \vdash p \wedge q,  & & p \wedge q \vdash p, & & p \wedge q \vdash q, \\
  & p \vee q \vdash p, q, & & p \vdash p \vee q, & & q \vdash p \vee q,
\end{align*}
  and the (``negative'') rules
\begin{align*}
  & \dmneg p, \dmneg q \vdash \dmneg (p \vee q), & & \dmneg (p \vee q) \vdash \dmneg p, & & \dmneg (p \vee q) \vdash \dmneg q, \\
  & \dmneg (p \wedge q) \vdash \dmneg p, \dmneg q,  & & \dmneg p \vdash \dmneg (p \wedge q), & & \dmneg q \vdash \dmneg (p \wedge q),
\end{align*}
  and the four rules
\begin{align*}
  & p \vdash \dmneg \dmneg p, & & \dmneg \dmneg p \vdash p, & & \emptyset \vdash \True, & & \False \vdash \emptyset.
\end{align*}
  The logic $\LP_{mc}$ extends $\BDmc$ by the rule $\emptyset \vdash p, \dmneg p$, the logic $\K_{mc}$ extends it by $p, \dmneg p \vdash \emptyset$, the logic $\KO_{mc}$ by $p, \dmneg p \vdash q, \dmneg q$, and $\CL = \LP_{mc} \vee \K_{mc}$ (see~\cite{beall13}).

\begin{theorem}
  The only proper non-trivial multiple-conclusion extensions of $\BDmc$ are $\KO_{mc}$, $\K_{mc}$, $\LP_{mc}$, and $\CL_{mc}$.
\end{theorem}

\begin{proof}
  We show that each multiple-conclusion rule $\Gamma \vdash \Delta$ is equivalent over $\BDmc$ to one of the rules
\begin{align*}
  \emptyset & \vdash \emptyset, & p, \dmneg p & \vdash \emptyset, & p, \dmneg p & \vdash q, \dmneg q, & p & \vdash p.
\end{align*}
   Since each formula is equivalent over $\BD$ to a formula in conjunctive normal form and a formula in disjunctive normal form, we may assume by appeal to cut that all formulas in $\Gamma$ and $\Delta$ are either atoms or negated atoms. Consider the substitution which assigns $\True$ to each $p$ such that $p \in \Gamma$, $p \notin \Delta$, $\dmneg p \in \Delta$, $\dmneg p \notin \Gamma$ and $\False$ to each $q$ such that $\dmneg q \in \Gamma$, $\dmneg q \notin \Delta$, $q \in \Delta$, $q \notin \Gamma$. The~effect of this substitution is to erase all such atoms $p$ from the premises and all such atoms $q$ from the conclusions. Moreover, if there is some $p \in \Gamma \cap \Delta$ or some $\dmneg q \in \Gamma \cap \Delta$, then the rule $\Gamma \vdash \Delta$ already holds in $\BD$. We may therefore assume without loss of generality that if $p \in \Gamma$ ($q \in \Delta$), then $p \notin \Delta$ and $\dmneg p \notin \Delta$ ($q \notin \Gamma$ and $\dmneg q \notin \Gamma$).

  If $p, \dmneg p \in \Gamma$ for some $p$ and $q, \dmneg q \in \Delta$ for some $q$, then the substitution which assigns $p$ to each variable in $\Gamma$ and $q$ to each variable in $\Delta$ shows that the rule $\Gamma \vdash \Delta$ is equivalent to $p, \dmneg p \vdash q, \dmneg q$. If $p, \dmneg p \in \Gamma$ for some $p$ and there is no $q$ such that $q, \dmneg \in \Delta$, then the substitution which assigns $p$ to each variable in $\Gamma$ and $\False$ ($\True$) to each (negated) atom in $\Delta$ shows that the rule $\Gamma \vdash \Delta$ is equivalent to $p, \dmneg p \vdash \emptyset$. Dually, if $q, \dmneg q \in \Delta$ for some $q$ and there is no $p$ such that $p, \dmneg p \in \Gamma$, then the rule $\Gamma \vdash \Delta$ is equivalent to $\emptyset \vdash q, \dmneg q$. Finally, if at most one of the formulas $p$, $\dmneg p$ occurs in $\Gamma$ and at most one the formulas $q, \dmneg q$ occurs in $\Delta$, a~suitable substitution again shows that $\Gamma \vdash \Delta$ is equivalent to $\emptyset \vdash \emptyset$.
\end{proof}

  The above argument does not depend essentially on the presence of $\True$ and~$\False$. Dropping the constants from the signature would merely complicate the picture by forcing us to distinguish (i) between the rules $\emptyset \vdash \emptyset$, $p \vdash \emptyset$, $\emptyset \vdash q$, and $p \vdash q$, (ii) between the rules $\emptyset \vdash p, \dmneg p$ and $q \vdash p, \dmneg p$, and (iii) between the rules $p, \dmneg p \vdash \emptyset$ and $p, \dmneg p \vdash q$. It would not, however, yield any substantially new logic.

\section{Constructing finite reduced models of \texorpdfstring{$\BD$}{BD} from graphs}
\label{sec: graph duality}

  In the second half of this paper, we study the relationship between finitary extensions of~$\BD$ and finite graphs. It turns out that graphs naturally come into play when we restrict the duality for De~Morgan algebras due to Cornish \& Fowler~\cite{cornish+fowler77} to finite reduced models of $\BD$. This allows us to describe each finite reduced model of $\BD$ up to isomorphism by a triple $\langle G, H, k \rangle$ where $G$ and $H$ are graphs and $k \in \omega$. Conversely, each such triple gives rise to a finite reduced model $\boldmu(G, H, k)$ of $\BD$. Moreover, $\boldmu(G, H, k)$ is a model of $\ETL$ if and only if $H = \emptyset$.

  Our first task will be to review the duality for finite De~Morgan algebras and extend it to finite \emph{De~Morgan matrices}, i.e.\ De~Morgan algebras equipped with a lattice filter. In particular, we need to describe the dual counter\-parts of strict homo\-morphisms and reduced matrices.

  The duality for finite De~Morgan algebras expands the Birkhoff duality for finite distributive lattices by an order-inverting involution on both sides. On the one side, we have \emph{involutive posets}, i.e.\ posets $\langle P, \leq \rangle$ equipped with an order-inverting involution $\dual$. Their homo\-morphisms (embeddings) are monotone maps (order embeddings) which commute with the involutions.

  To obtain a De~Morgan algebra from an involutive poset, we take the bounded distributive lattice of upsets and expand it by the operation $\dmneg U = P \setminus \dual[U]$. This results in the \emph{complex algebra} of the involutive poset $P$, denoted $P^{+}$ here. Each homomorphism of involutive posets $f\colon P \to Q$ then yields a homomorphism of De~Morgan algebras $f^{+} \assign f^{-1}\colon Q^{+} \to P^{+}$.

  On the other side, we have De~Morgan algebras and their homomorphisms. To obtain an involutive poset from a De~Morgan algebra~$\alg{A}$, we take the poset of prime filters ordered by inclusion and expand it by the operation $\dual \filter{U} = \alg{A} \setminus \dmneg[\filter{U}]$. (Recall that a prime filter is a proper lattice filter $\filter{U}$ such that $a \vee b \in \filter{U}$ if and only if $a \in \filter{U}$ or $b \in \filter{U}$.) This results in the \emph{dual poset} of $\alg{A}$, denoted $\alg{A}_{+}$ here. A~homomorphism of De~Morgan algebras $h\colon \alg{A} \to \alg{B}$ then yields a homomorphism of involutive posets $h_{+} \assign h^{-1}\colon \alg{B}_{+} \to \alg{A}_{+}$.

  The map $\unit(a) = \set{a \in \filter{U}}{\filter{U} \in \alg{A}_{+}}$ embeds the De~Morgan algebra $\alg{A}$ into $(\alg{A}_{+})^{+}$. This embedding is an isomorphism if $\alg{A}$ is finite. Conversely, the map $\counit(u) = \set{u \in U}{U \in P^{+}}$ embeds the involutive poset $P$ into $(P^{+})_{+}$. This embedding is also an isomorphism if $P$ is finite.

\begin{theorem} \label{thm: algebra duality}
  The complex algebra and dual involutive poset constructions are functors which form a dual equivalence between the categories of finite De Morgan algebras and finite involutive posets, with unit $\unit$ and counit $\counit$.
\end{theorem}

\begin{fact} \label{fact: embeddings of algebras}
  Embeddings (surjective homomorphisms) between finite De~Morgan algebras are precisely the duals of sur\-jective homomorphisms (embeddings) between finite involutive posets.
\end{fact}

  The duals of De~Morgan matrices are involutive posets expanded by an upset of \emph{designated points}. For the sake of brevity, let us simply call such structures \emph{frames}. Frames will be denoted by $P$ or $Q$ and the upset of designated points of $P$ ($Q$) will be denoted $D_{P}$ ($D_{Q}$). Homomorphisms of frames are defined as homomorphisms of involutive posets which preserve designation. That is, $f\colon P \to Q$ is a homomorphism of frames if $f$ is a homomorphism of involutive posets and $D_{P} \subseteq h^{-1}[D_{Q}]$.

  The \emph{complex matrix} of a frame $P$ is the complex algebra $P^{+}$ of the involutive poset reduct of $P$ equipped with the principal filter generated by the upset $D \in P^{+}$. Conversely, the \emph{dual} of a De~Morgan matrix $\pair{\alg{A}}{F}$ is the involutive poset $\alg{A}_{+}$ expanded by the upset $D \assign \set{G \supseteq F}{G \in \alg{A}_{+}}$.

\begin{theorem} \label{thm: matrix duality}
  The complex matrix and dual frame constructions are functors which form a dual equivalence between the categories of finite De Morgan matrices and finite frames, with unit $\unit$ and counit $\counit$.
\end{theorem}

\begin{proof}
  Given Theorem~\ref{thm: algebra duality}, it suffices to make the following two observations. If $h\colon \pair{\alg{A}}{F} \rightarrow \pair{\alg{B}}{G}$ is a homomorphism of De~Morgan matrices and $\filter{V}$ is a prime filter of $\alg{B}$ such that $\filter{V} \supseteq G$, then $h^{-1}[\filter{V}]$ is a prime filter of $\alg{A}$ and $h^{-1}[\filter{V}] \supseteq h^{-1}[G] \supseteq F$, therefore the map $h^{-1}$ is a homomorphism of frames.

  Conversely, if $f\colon P \to Q$ is a homomorphism of frames and $U$ is an upset of $Q$ such that $U \supseteq D_{Q}$, then $f^{-1}[U] \supseteq f^{-1}[D_{Q}] \supseteq D_{P}$, therefore the map $ f^{-1}$ is a homomorphism of De~Morgan matrices.
\end{proof}

  We now describe the duals of strict homomorphisms. The following notation will be useful: ${\uparrow} X$ will denote the upward closure of a set $X$ in some given poset, and $\min X$ ($\max X$) will denote the set of minimal (maximal) elements of $X$.

\begin{proposition} \label{prop: strict duality}
  Let $f\colon P \to Q$ be a homomorphism of finite frames. Then $f^{+}$ is strict if and only if $D_{Q} = {\uparrow} f[D_{P}]$, or equivalently $D_{Q} \subseteq {\uparrow} f[D_{P}]$.
\end{proposition}

\begin{proof}
  The inclusion ${\uparrow} f[D_{P}] \subseteq D_{Q}$ holds for each homomorphism of frames. Suppose therefore that $D_{Q} \subseteq {\uparrow} f[D_{P}]$ and consider an upset $U$ of $Q$ such that $f^{+}(U)$ is designated in $P^{+}$. Then $f^{-1}[U] \supseteq D_{P}$, hence $U \supseteq f[D_{P}]$ and $U = {\uparrow} U \supseteq {\uparrow} f[D_{P}] \supseteq D_{Q}$. The homomorphism $f^{+}$ is therefore strict.

  Conversely, let $U \assign {\uparrow} f[D_{P}]$. Then $f^{+}(U) = f^{-1}[{\uparrow} f[D_{P}]] \supseteq f^{-1}[f[D_{P}]] \supseteq D_{P}$, hence $f^{+}(U)$ is designated in $P^{+}$. If $f$ is strict, then $U$ is designated in~$Q^{+}$, therefore ${\uparrow} f[D_{P}] \supseteq D_{Q}$.
\end{proof}

  In view of the above proposition, let us call a homomorphism $f\colon P \to Q$ of finite frames \emph{strict} if $D_{Q} = {\uparrow} f[D_{P}]$, or equivalently if $D_{Q} \subseteq {\uparrow} f[D_{P}]$. 

  The duals of strict homomorphic images of $P^{+}$ can be identified with certain substructures of $P$. We say that a \emph{subframe} of a finite frame $P$ is a subposet $Q$ closed under $\dual$ with $D_{Q} \assign D_{P} \cap Q$. A \emph{strict subframe} of $P$ is then a subframe $Q$ such that $D_{P} = {\uparrow} D_{Q}$, or equivalently ${\min D_{P} \subseteq Q}$.

\begin{proposition}
  Let $P$ be a finite frame. Up to isomorphism, the strict homomorphic images of $P^{+}$ are the complex algebras of strict subframes of~$P$.
\end{proposition}

\begin{proof}
  By Proposition~\ref{prop: strict duality} and Fact~\ref{fact: embeddings of algebras}, $Q^{+}$ is a strict homomorphic image of $P^{+}$ if and only if there is a homomorphism of frames $f\colon Q \to P$ such that $f$ is an order embedding and $D_{P} = {\uparrow} f [D_{Q}]$. Such a map $f$ reflects designation: if $f(u) \in D_{P}$, then there is $v \in D_{Q}$ such that $f(v) \leq f(u)$, thus $v \leq u$ because $f$ is an order embedding, and $u \in D_{Q}$ because $D_{Q}$ is an upset. The map $f$ is therefore an isomorphism between $Q$ and a strict subframe of $P$.
\end{proof}

\begin{proposition} \label{prop: leibniz subframes}
  Let $P$ be a finite frame and $Q$ be the subframe $P$ over $\min D_{P} \cup \dual[\min D_{P}]$. Then the Leibniz reduct of $P^{+}$ is isomorphic to~$Q^{+}$.
\end{proposition}

\begin{proof}
  The Leibniz reduct of $P^{+}$ is the smallest strict homomorphic image of $P^{+}$. The claim now follows from the previous proposition and the observation that the subframe $Q$ is strict if and only if $\min D_{P} \subseteq Q$.
\end{proof}

\begin{corollary} \label{cor: reduced frames}
  Let $P$ be a finite frame. Then $P^{+}$ is a reduced matrix if and only if $P = \min D_{P} \cup \dual [\min D_{P}]$.
\end{corollary}

  Accordingly, we call a finite frame $P$ \emph{reduced} if $P = \min D_{P} \cup \dual [\min D_{P}]$, and we call the subframe $Q$ of $P$ over $\min D_{P} \cup \dual[\min D_{P}]$ the \emph{Leibniz subframe} of~$P$.

  Each frame~$P$ is a disjoint union of its components, where a \emph{component} of $P$ is a subframe whose underlying set is closed upward as well as downward in $P$ (in~addition to being closed under $\dual$). Each component of $P$ is a frame in its own right. We call $P$ \emph{connected} if it has no non-empty proper components.

\begin{proposition} \label{prop: reduced models}
  If $P$ is a finite reduced frame, then $P = \min P \cup \max P$. If~$P$ is moreover connected, then either $P = \{ u, \dual u \}$ for some $u \in P$ or $\min P$ and $\max P$ are disjoint. In the latter case either $D_{P} = P$ or $D_{P} = \max P$.
\end{proposition}

\begin{proof}
  Because $\dual u \in \max P$ if and only if $u \in \min P$, for the first claim it suffices to prove that for each $u \in P$ either $u \in \max P$ or $\dual u \in \max P$. Suppose therefore that there are $v > u$ and $w > \dual u$. If~$u \notin \min D_{P}$, then $\dual u \in \min D_{P}$, so $w \notin \min D_{P}$ and $\dual w \in \min D_{P}$. But $\dual w < u < v$, therefore $v \notin \min D_{P}$ and $\dual v \in \min D_{P}$. This is a contradiction: $\dual v \in \min D_{P}$ and $\dual u \in \min D_{P}$ but $\dual v < \dual u$. This proves that $P = \min P  \cup \max P$.

    If $u \in \min P \cap \max P$, then there is no $v > u$ and no $v < u$. Because $P$ is connected, $P = \{ u, \dual u \}$. Let us therefore assume that $\min P \cap \max P = \emptyset$.

  Suppose that $u \notin D_{P}$ for some $u \in P$. Then $\dual u \in \min D_{P}$, therefore $v > \dual u$ implies $v \notin \min D_{P}$ and $\dual v \in \min D_{P}$. But $\dual v < u \notin D_{P}$, contradicting the fact that $D_{P}$ is an upset. It follows that there can be no $v > \dual u$ if $u \notin D_{P}$. In other words, if $u \notin D_{P}$, then $u \in \min P$. Equivalently, $\max P \subseteq D_{P}$.

  Conversely, suppose that $u \in D_{P}$ for some $u \in \min P$. We wish to show that $w \in D_{P}$ for each $w \in \min P$. By connectedness, it suffices to prove that (i)~$w \in D_{P}$ if $u < \dual w \in D_{P}$, and (ii) $w \in D_{P}$ if $u < v > w$ for some $v \in P$. The~first claim holds because $u < \dual w$ implies $\dual w \notin \min D_{P}$, so $w \in D_{P}$. To~prove the second claim, we apply (i) twice: $u < v$ implies $\dual v \in D_{P}$, and now $\dual v < \dual w$ implies $w = \dual \dual w \in D_{P}$.
\end{proof}

  We can now determine every finite reduced frame up to isomorphism by a pair of graphs $G$, $H$ and a natural number $k$. The graph $G$ will describe the non-singleton components where each element is designated, the graph $H$ will describe the non-singleton components where exactly one of the elements $u$, $\dual u$ is designated, and $k$ will specify the number of singleton components.

  Let us first clarify our terminology. By a \emph{graph} we mean a finite symmetric graph with loops allowed, i.e.\ a finite (possibly empty) set $X$ equipped with a symmetric binary relation $R$. Vertices $u \in X$ and $v \in X$ will be called \emph{neighbors} or \emph{adjacent} vertices if $u R v$ (allowing for $u = v$). A vertex $u \in X$ is \emph{reflexive} if $u R u$, otherwise it is \emph{irreflexive}. The vertex $u$ is \emph{isolated} if it has no neighbors. (No reflexive vertex is isolated.) The \emph{complete graph} $K_{n}$ on $n$ vertices is the set $\{ 1, \dots, n \}$ equipped with the inequality relation. Thus $K_{1}$ is the irreflexive singleton graph and $K_{2}$ is a single edge between two vertices. We prefer the more suggestive notation ${\bullet} \assign K_{1}$. The \emph{disjoint union} of two graphs will be denoted $G \sqcup H$. A \emph{homomorphism} $h\colon G \to H$ of graphs $G = \pair{X}{R}$ and $H = \pair{Y}{S}$ is a map $h\colon X \to Y$ such that $u R v$ implies $f(u) S f(v)$.

  A De~Morgan matrix can be constructed from a finite graph $G = \pair{X}{R}$ as follows. Let $X \sqcup \dual X$ be the disjoint union of two copies of $X$, denoted $X$ and $\dual X$, with $\dual$ being an involution in $X \sqcup \dual X$ switching between the two copies. That is, each element of $X \sqcup \dual X$ has one of the forms $u$ or $\dual u$ for some $u \in X$. We define a partial order on this set:
\begin{align*}
  u \leq_{G} v \iff \text{either } u = v \text{ or }v = \dual w \text{ for some } w \in X \text{ such that }u R w.
\end{align*}
  This partial order with the involution $\dual$ defines the involutive poset $P(G)$.

  We may now equip $P(G)$ with two different sets of designated elements:
\begin{align*}
  D_{+}(G) & \assign X \cup \dual X, & D_{-}(G) & \assign \dual X.
\end{align*}
  This results in the two frames $P_{+}(G)$ and $P_{-}(G)$, respectively. The frame $P(G, H, k)$ is then defined as the disjoint union of the frames $P_{+}(G)$ and $P_{-}(H)$ and $k$ singleton reduced frames. An example is shown in Figure~\ref{fig: mu construction}, where $G_{2}$ is the graph consisting of a reflexive vertex $u$, an irreflexive vertex $v$, and an edge between $u$ and $v$, and $K_{2}$ is the complete graph on $2$ vertices, i.e.\ it consists of two irreflexive vertices connected by an edge. Observe that the complex matrix of $P_{+}(G_{2})$ is precisely the matrix $\Kminusmatrix$ of Figure~\ref{fig: kminus matrix}.

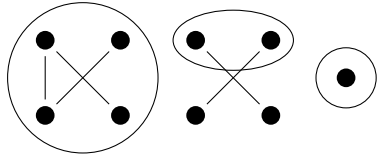
\begin{figure}
\caption{The frame $P(G_{2}, K_{2}, 1)$}
\label{fig: mu construction}

\bigskip

\begin{center}
\begin{tikzpicture}[scale=1,
  dot/.style={circle,fill,inner sep=2.5pt,outer sep=2.5pt}]
  \node (u) at (0,0) [dot] {};
  \node (du) at (0,1) [dot] {};
  \node (v) at (1,0) [dot] {};
  \node (dv) at (1,1) [dot] {};
  \node (p) at (2,0) [dot] {};
  \node (dp) at (2,1) [dot] {};
  \node (q) at (3,0) [dot] {};
  \node (dq) at (3,1) [dot] {};
  \node (b) at (4,0.5) [dot] {};
  \draw[-] (u) edge (du) edge (dv);
  \draw[-] (v) edge (du);
  \draw[-] (p) edge (dq);
  \draw[-] (q) edge (dp);
  \draw (0.5,0.5) ellipse (1 and 1);
  \draw (2.5,1) ellipse (0.8 and 0.4);
  \draw (4, 0.5) ellipse (0.4 and 0.4);
\end{tikzpicture}
\end{center}
\end{figure}

  The complex matrices of the frames $P_{+}(G)$, $P_{-}(H)$, and $P(G, H, k)$ will be denoted $\boldmu_{+}(G)$, $\boldmu_{-}(H)$, and $\boldmu(G, H, k)$, respectively. That is:
\begin{align*}
  \boldmu(G, H, k) \cong \boldmu_{+}(G) \times \boldmu_{-}(H) \times \CLmatrix^{k}.
\end{align*}
  In particular, the matrix $\boldmu(\emptyset, \emptyset, k)$ is isomorphic to $\CLmatrix^{k}$. The matrices $\boldmu(G, H, 1)$ and $\boldmu(G, H, n)$ are logically equivalent for~$n \geq 1$.\footnote{The notation $\boldmu_{+}(G)$ is unrelated to the notation $\matr{A}_{+}$ for the dual frame of a De~Morgan matrix. No confusion threatens: the notation $\boldmu(G)$ does not mea anything here.}

\begin{fact} \label{fact: disjoint unions}
  $\boldmu(G \sqcup G', H \sqcup H', i + i')$ is isomorphic to $\boldmu(G, H, i) \times \boldmu(G', H', i')$.
\end{fact}

\begin{fact}
  $\boldmu_{+}(G)$ is a model of $\ETL$. $\boldmu_{-}(G)$ is not a model of $\ECQ$.
\end{fact}

\begin{proof}
  $\boldmu_{+}(G)$ is a model of $\ETL$ because only its top element is designated. The rule $p, \dmneg p \vdash \emptyset$ fails in $\boldmu_{-}(G)$ if $p$ is interpreted by $D_{-}(G)$.
\end{proof}

  In particular, $\Log \boldmu_{-}(G) \logleq \LP$ because $[\BD, \CL]$ splits into $[\BD, \LP]$ and $[\ECQ, \CL]$. Let us explicitly compute the logics of these matrices:
\begin{align*}
  \Log \boldmu(G, H, 0) & = (\Log \boldmu_{+}(G) \cap \Log \boldmu_{-}(H)) \cup \Exp_{\BD} \Log \boldmu_{+}(G), \\
  \Log \boldmu(G, H, 1) & = (\Log \boldmu_{+}(G) \cap \Log \boldmu_{-}(H)) \cup \ECQ_{\omega}.
\end{align*}
  Because $\Exp_{\BD} \Log \boldmu_{-}(H) \logleq \Exp_{\BD} \LP = \BD$, we have
\begin{align*}
  \Exp_{\BD} \Log \boldmu(G, H, 0) & = \Exp_{\BD} \Log \boldmu_{+}(G), \\
  \Exp_{\BD} \Log \boldmu(G, H, 1) & = \ECQ_{\omega}.
\end{align*}

\begin{fact} \label{fact: reduced models of exp ext}
  Let $\logic{L} \in \Exp \Ext \BD$ be non-trivial. Then $\boldmu(G, H, 0)$ is a model of $\logic{L}$ if and only if $\boldmu_{+}(G)$ is a model of $\logic{L}$. Each $\boldmu(G, H, 1)$ is a model of $\logic{L}$.
\end{fact}

  The above construction covers all finite reduced models of $\BD$. 

\begin{theorem}
  The finite reduced models of~$\BD$ are, up to isomorphism, precisely the matrices $\boldmu(G, H, k)$ for some graphs $G$ and $H$ and some~$k \in \omega$.
\end{theorem}

\begin{proof}
  Let $P$ be a connected finite reduced frame which is not a singleton. Take $X \assign \min P$ and define the binary relation $R$ on $X$ as folows: $u R v$ if and only if $u \leq \dual v$. By Proposition~\ref{prop: reduced models} the frame $P$ is isomorphic either to $P_{+}(G)$ or $P_{-}(G)$ where $G = \pair{X}{R}$.

  If $P$ is an arbitrary finite reduced frame, then $P$ is the disjoint union of its components, which either have the forms $P_{+}(G)$ or $P_{-}(G)$ or they are designated singletons. Taking $G$ to be the disjoint union of all graphs $G$ such that $P_{+}(G)$ is a component of $P$, $H$ to be the disjoint union of all graphs $H$ such that $P_{-}(H)$ is a component of $P$, and $k$ to be the number of designated singleton components, we see that $P$ is isomorphic to $P(G, H, k)$.
\end{proof}

\begin{theorem}
  The finite reduced models of $\ETL$ are, up to isomorphism, precisely the matrices $\boldmu(G, \emptyset, k)$ for some graph $G$ and some $k \in \omega$.
\end{theorem}

\begin{proof}
  Recall Proposition~\ref{prop: reduced models of etl}: a reduced model $\pair{\alg{A}}{F}$ of $\BD$ is a model of $\ETL$ if and only if $F = \{ \True \}$.
\end{proof}

\section{Graph-theoretic completeness theorems}
\label{sec: graph completeness}

  Because the finite reduced models of $\BD$ correspond precisely to triples $\langle G, H, k \rangle$ where $G$ and $H$ are graphs and $k \in \omega$, a completeness theorem for a super-Belnap logic may take the form of specifying a class of such triples. Fortunately, for many logics, including the logics $\ECQ_{n}$, $\ETL_{n}$, and $\ETLplus_{n}$, it will suffice to only consider triples of the form $\langle G, \emptyset, 0 \rangle$. This will yield genuine graph-theoretic completeness theorems for such logics. For example, $\ETL_{n}$ turns out to be complete, in a suitable sense, with respect to the class of all non-$n$-colorable graphs.

  Recall that a logic $\logic{L}$ is said to be $\omega$-complete with respect to a class of matrices $\class{K}$ if it is complete with respect to $\class{K}$ as a finitary logic, i.e.\ if $\logic{L} = \Log_{\omega} \class{K}$.

\begin{theorem} \label{thm: completeness for exp ext etl}
  Let $\logic{L} < \ETL_{\omega}$ be a finitary explosive extension of $\ETL$. Then $\logic{L}$ is $\omega$-complete with respect to all matrices of the form $\boldmu_{+}(G) \times \ETLmatrix$ where $\boldmu_{+}(G)$ is a model of $\logic{L}$ and $G$ has no isolated vertices.
\end{theorem}

\begin{proof}
  Each non-trivial reduced model of $\logic{L}$ is logically equivalent to a model of the form $\boldmu_{+}(G)$ for $G$ non-empty or a model of the form $\boldmu_{+}(G) \times \CLmatrix$.

  Among models of $\logic{L}$ of the form $\boldmu_{+}(G)$, we may restrict to those where $G$ has exactly one isolated vertex, since $\Log \boldmu_{+}(G \sqcup \bullet \sqcup \bullet) = \Log \boldmu_{+}(G \sqcup \bullet) \logleq \Log \boldmu_{+}(G)$. This holds because $\Log \boldmu_{+}(G) \times \CLmatrix = \Log \boldmu_{+}(G) \cup \ETL_{\omega}$ and $\Log \boldmu_{+}(G \sqcup \bullet) = \Log \boldmu_{+}(G) \times \ETLmatrix = \ETL \cup \Exp_{\ETL} \Log \boldmu_{+}(G)$. But the models $\boldmu_{+}(G)$ where $G$ has exactly one isolated vertex are up to isomorphism precisely the models $\boldmu_{+}(G) \times \ETLmatrix$ where $G$ has no isolated vertices.

  Among models of $\logic{L}$ of the form $\boldmu_{+}(G) \times \CLmatrix$, we may restrict to $\ETLmatrix \times \CLmatrix = \boldmu_{+}(\bullet) \times \CLmatrix$. This is because $\Log \boldmu_{+}(G) \times \CLmatrix = \Log \boldmu_{+}(G) \cup \ETL_{\omega} \loggeq \ETL_{\omega} = \Log \ETLmatrix \times \CLmatrix$. Finally, $\logic{L} < \ETL_{\omega}$, so $\logic{L}$ must have at least one model of the form $\boldmu_{+}(G) \times \ETLmatrix$. But then $\Log \boldmu_{+}(G) \times \ETLmatrix = \ETL \cup \Exp_{\ETL} \Log \boldmu_{+}(G) \logleq \ETL_{\omega}$, therefore we may also disregard the model $\ETLmatrix \times \CLmatrix$.
\end{proof}

\begin{theorem} \label{thm: completeness for exp ext bd}
  Let $\logic{L}$ be a proper finitary explosive extension of $\BD$ such that ${\logic{L} < \ECQ_{\omega}}$. Then $\logic{L}$ is $\omega$-complete with respect to all matrices of the form $\boldmu_{+}(G) \times \BDmatrix$ where $\boldmu_{+}(G)$ is a model of $\logic{L}$ and $G$ has no isolated vertices.
\end{theorem}

\begin{proof}
  Recall that $\logic{L} = \Exp_{\BD} (\ETL \vee \logic{L})$ and $\ETL \vee \logic{L} = \ETL \cup \logic{L}$ for $\logic{L}$ in $\Exp \Ext \ECQ$ by Theorem~\ref{thm: explosive isomorphism}. The inequality $\logic{L} < \ECQ_{\omega}$ implies $\logic{L} < \ETL_{\omega}$, since $\ETL \vee \logic{L} = \ETL \cup \logic{L}$. The logic $\ETL \vee \logic{L}$ is $\omega$-complete with respect to all of its models of the form $\boldmu_{+}(G) \times \ETLmatrix$ where $G$ has no isolated vertices, or equivalently with respect to matrices $\boldmu_{+}(G) \times \ETLmatrix$ where $\boldmu_{+}(G)$ is a model of $\logic{L}$ and $G$ has no isolated vertices. Let us call this class of graphs $\class{K}$. Then
\begin{align*}
  \logic{L} & = \Exp_{\BD} (\ETL \vee \logic{L}) \approx \Exp_{\BD} \bigcap_{G \in \class{K}} \Log \boldmu_{+}(G) \times \ETLmatrix \\
  & = \bigcap_{G \in \class{K}} \Exp_{\BD} \Log \boldmu_{+}(G) \times \ETLmatrix = \bigcap_{G \in \class{K}} \Log \boldmu_{+}(G) \times \BDmatrix,
\end{align*}
  where the last equality holds because
\begin{align*}
  \Exp_{\BD} \Log \boldmu_{+}(G) \times \ETLmatrix & = \Exp_{\BD} (\ETL \cup \Exp_{\BD} \boldmu_{+}(G)) \\
  & = \ECQ \cup \Exp_{\BD} \boldmu_{+}(G) \\
  & = \Exp_{\BD} \boldmu_{+}(G) \\
  & = \Log \boldmu_{+}(G) \times \BDmatrix,
\end{align*}
  and by $\logic{L}_{1} \approx \logic{L}_{2}$ we mean that the finitary parts of $\logic{L}_{1}$ and $\logic{L}_{2}$ coincide.
\end{proof}

  To obtain a completeness theorem for a finitary explosive extension of $\BD$ or $\ETL$, it thus suffices to describe the class of graphs $G$ without isolated vertices for which $\boldmu_{+}(G)$ is a model of $\logic{L}$. The following theorem records another situation where satisfactory graph-theoretic completeness results may be obtained.

\begin{theorem} \label{thm: completeness not below etl omega}
  Let $\logic{L}$ be a finitary non-classical extension of $\ETL$ (of $\BD$) by a set of rules of the form $\Gamma \vdash \varphi$ where $\Gamma$ is not an antitheorem of $\CL$. Then $\logic{L}$ is $\omega$-complete with respect a class of matrices of the form $\boldmu_{+}(G)$ (or $\boldmu_{-}(H)$).
\end{theorem}

\begin{proof}
  It suffices to observe that Corollary~\ref{cor: potential antitheorems} applies to $\logic{L}$, so $\boldmu(G, H, k)$ is a model of $\logic{L}$ if and only if $\boldmu_{+}(G)$, $\boldmu_{-}(H)$, and $\CLmatrix$ (if $k \geq 1$) are.
\end{proof}

  This applies in particular to the logic introduced in Proposition~\ref{prop: etl omega splitting} as the smallest super-Belnap logic which does not lie below $\ETL_{\omega}$, axiomatized by the rule $(p \wedge \dmneg p) \vee q \vee \dmneg q, (q \wedge \dmneg q) \vee p \vee \dmneg p \vdash p \vee \dmneg p$. Proving a non-trivial completeness theorem for this logic is a problem that we leave open.

  To understand which matrices of the form $\boldmu_{+}(G)$ are models of a given logic, it will be helpful to consider a different, simpler matrix based on the graph $G$. This matrix $\boldgamma(G)$ is the bounded distributive lattice of subsets of $X$ equipped with the operation $\dmneg U \assign X \setminus R[U]$ for $U \subseteq X$ and with the set of designated values $\{ X \}$, where $R[U] \assign \set{v \in X}{u R v \text{ for some } u \in U}$. In other words, $\dmneg U$~is the set of all vertices of $G$ which are not neighbors of any vertex in $U$.\footnote{The reader will observe that the operation $\dmneg x$ can be expressed as $\Box \neg x$ where $\neg$ is the Boolean negation and $\Box$ is the usual box operator of classical modal logic. The matrices $\boldgamma(G)$ are thus reducts of matrices which arise in the context of classical modal logic.} 

  The matrices $\boldgamma(G)$ themselves are almost never models of $\BD$: the equality $\dmneg (U \cap V) = \dmneg U \cup \dmneg V$ will fail unless each vertex has at most one neighbor. However, the following lemma demonstrates the value of these matrices for under\-standing which rules hold in the matrices $\boldmu_{+}(G)$.

\begin{lemma} \label{lemma: gamma construction}
  Let $\Gamma \cup \{ \varphi \}$ be a set of formulas where negation is only applied directly to atoms, and let $G = \pair{X}{R}$ be a graph without isolated vertices. If the rule $\Gamma \vdash \varphi$ is valid in $\boldmu_{+}(G)$, then it is valid in $\boldgamma(G)$. If $\varphi$ does not contain negation, then the opposite implication also holds.
\end{lemma}

\begin{proof}
  Recall that we identify the set $X$ with a subset of $P_{+}(G)$. We define the maps ${\uparrow}\colon \boldgamma(G) \rightarrow \boldmu_{+}(G)$ and ${\downarrow}\colon \boldmu_{+}(G) \rightarrow \boldgamma(G)$ as follows: ${\uparrow} U$ for $U \subseteq X$ is the upward closure of $U$ in $P_{+}(G)$, while ${\downarrow} V$ for $V \subseteq P_{+}(G)$ is the restriction of $V$ to $X$ as a subset of $P_{+}(G)$, i.e.\ $V \cap X$. These are not homomorphisms, but they enjoy some of the useful properties of homomorphisms:
\begin{align*}
  {\downarrow} (U \cup V) & = {\downarrow} U \cup {\downarrow V}, & {\downarrow} \dmneg {\uparrow} U & = \dmneg U, \\
  {\downarrow} (U \cap V) & = {\downarrow} U \cap {\downarrow V}, & {\downarrow} \dmneg U & \subseteq \dmneg {\downarrow} U.
\end{align*}
  Moreover, $U$ is designated in $\boldgamma(G)$ if and only if ${\uparrow} U$ is designated in $\boldmu_{+}(G)$. Conversely, $V$ is designated in $\boldmu_{+}(G)$ if and only if ${\downarrow} V$ is designated in $\boldgamma(G)$. (Here we use the assumption that $G$ does not contain isolated vertices.)

  Given a valuation $v$ on $\boldgamma(G)$, we define a valuation $u$ on $\boldmu_{+}(G)$ such that $u(p) \assign {\uparrow} v(p)$. We prove by induction over the complexity of the formula $\psi$ that ${\downarrow} u(\psi) = v(\psi)$. The inductive steps for meets and joins are trivial, as are the base cases for $\True$ and $\False$. The only non-trivial cases are
\begin{align*}
  & {\downarrow} u(p) = {\downarrow} {\uparrow} v(p) = v(p), & & {\downarrow} u(\dmneg p) = {\downarrow} \dmneg {\uparrow} v(p) = \dmneg v(p) = v(\dmneg p).
\end{align*}
  Consequently, $v(\psi)$ is designated in $\boldgamma(G)$ if and only if $u(\psi)$ is designated in $\boldmu_{+}(G)$. A counterexample $v$ to the rule $\Gamma \vdash \varphi$ on $\boldgamma(G)$ thus yields a counter\-example $u$ to this rule on $\boldmu_{+}(G)$.

  Conversely, let $v$ be a valuation on $\boldmu_{+}(G)$ which invalidates the rule $\Gamma \vdash \varphi$. Let $w$ be the valuation on $\boldgamma(G)$ such that $w(p) = {\downarrow} v(p)$. We prove by induction over the complexity of the formula $\psi$ that ${\downarrow} v(\psi) \subseteq w(\psi)$. The inductive steps for meets and joins are again trivial, as are the base cases for $\True$ and $\False$ and for atoms $p$. The only non-trivial case is ${\downarrow} \dmneg v(p) \subseteq \dmneg {\downarrow} v(p)$. Similarly, if $\psi$ does not contain negation, we prove by induction over the complexity of $\psi$ that ${\downarrow} v(\psi) = w(\psi)$. Consequently, if $v(\psi)$ is designated in $\boldmu_{+}(G)$, then $w(\psi)$ is designated in $\boldgamma(G)$, and the converse implication holds if $\psi$ does not contain negation. A~counterexample $v$ to the rule $\Gamma \vdash \varphi$ on $\boldmu_{+}(G)$ thus yields a counterexample $w$ to this rule on $\boldgamma(G)$, provided that negation does not occur in $\psi$.
\end{proof}

  With the help of this lemma, we can easily identify the graph-theoretic counter\-parts of various logical rules considered so far. Recall that a graph $G$ is called \emph{$n$-colorable} if there is an \emph{$n$-coloring} of $G$, i.e.\ a homomorphism $G \to K_{n}$.

\begin{fact} \label{fact: n-colorable matrix}
  The graph $G$ is not $n$-colorable (for $n \geq 2$) if and only if the rule $(p_{1} \wedge \dmneg p_{1}) \vee \dots \vee (p_{n} \wedge \dmneg p_{n}) \vdash \emptyset$ holds in $\boldgamma(G)$.
\end{fact}

\begin{proof}
  There is an immediate bijective correspondence between valuations on $\boldgamma(G)$ which invalidate the rule and $n$-colorings of $G$.
\end{proof}

  The rules axiomatizing the logics $\ETLplus_{n}$ correspond to a stronger property. We~define a \emph{partial homomorphism} $h\colon G \to H$ as a homomorphism $h\colon G' \to H$ where $G'$ is a subgraph of $G$. That is, $G' = \pair{Y}{S}$ where $Y \subseteq X$ and $S = R \cap Y^2$. A \emph{weak $n$-coloring} of $G$ is a partial homomorphism $h\colon G \to K_{n}$ such that for at least one vertex $u$ of $G$ the map $h$ is defined on all the neighbors of $u$. In other words, the set of vertices where $h$ is undefined is small in the very modest sense that not every vertex of $G$ is adjacent to this set.

\begin{fact} \label{fact: weakly n-colorable matrix}
  The graph $G$ is not weakly $n$-colorable (for $n \geq 1$) if and only if the rule $(p_{1} \wedge \dmneg p_{1}) \vee \dots \vee (p_{n} \wedge \dmneg p_{n}) \vee q, \dmneg q \vee r \vdash r$ holds in $\boldgamma(G)$.
\end{fact}

\begin{proof}
  There is again an immediate correspondence between valuations on $\boldgamma(G)$ which invalidate the rule and weak $n$-colorings of $G$.
\end{proof}

  A matrix separating $\ETLplus_{n}$ and $\ETLplus_{n+1}$ may now be supplied.

\begin{fact} \label{fact: etlplus separation}
  $\ETL_{n+2} \nlogleq \ETLplus_{n}$. Consequently, $\ETLplus_{n} < \ETLplus_{n+1}$.
\end{fact}

\begin{proof}
  The graph $K_{n+2}$ is $(n+2)$-colorable but not weakly $n$-colorable, so $\boldmu_{+}(K_{n+2})$ is a model of $\ETLplus_{n}$ but not $\ETL_{n+2}$. Because $\ETL_{n+2} \logleq \ETLplus_{n+1}$, it follows that $\ETLplus_{n} \nlogleq \ETLplus_{n+1}$.
\end{proof}

  The graph-theoretic counterparts of $\ETL_{\omega}$ and $\ETLplus_{\omega}$ can immediately be inferred from the counterparts of $\ETL_{n}$ and $\ETLplus_{n}$.

\begin{fact} \label{fact: contains reflexive vertex}
  The graph $G$ contains a reflexive vertex if and only if the rules $(p_{1} \wedge \dmneg p_{1}) \vee \dots \vee (p_{n} \wedge \dmneg p_{n}) \vdash \emptyset$ hold in $\boldgamma(G)$ for all $n \in \omega$.
\end{fact}

\begin{proof}
  The graph $G$ is not $n$-colorable for any $n$ if and only if it contains a reflexive vertex.
\end{proof}

\begin{fact} \label{fact: graphs for k minus}
  Each irreflexive vertex of $G$ has a reflexive neighbor if and only if $\boldgamma(G)$ validates the rules $(p_{1} \wedge \dmneg p_{1}) \vee \dots \vee (p_{n} \wedge \dmneg p_{n}) \vee q, \dmneg q \vee r \vdash r$ for all $n \in \omega$.
\end{fact}

\begin{proof}
  A graph in which each irreflexive vertex has a reflexive neighbor is not weakly $n$-colorable for any $n$, since a weak coloring is undefined on reflexive vertices. Conversely, if $u$ is a vertex of a graph $G$ with no reflexive neighbors, then for large enough $n$ there is a partial homomorphism $h\colon G \rightarrow K_{n}$ defined on all neighbors of $u$. Such a partial homomorphism is a weak $n$-coloring.
\end{proof}

  Given Theorems~\ref{thm: completeness for exp ext etl} and \ref{thm: completeness for exp ext bd} and Lemma~\ref{lemma: gamma construction}, the above facts yield graph-theoretic completeness theorems for the logics $\ECQ_{n}$, $\ETL_{n}$, and $\ETLplus_{n}$.

\begin{theorem}
  The logic $\ECQ_{n}$ (for $n \geq 2$) is $\omega$-complete with respect to the class of all matrices of the form $\boldmu_{+}(G) \times \BDmatrix$, where $G$ is a graph without isolated vertices which is not $n$-colorable.
\end{theorem}

\begin{theorem}
  The logic $\ETL_{n}$ (for $n \geq 2$) is $\omega$-complete with respect to the class of all matrices of the form $\boldmu_{+}(G) \times \ETLmatrix$, where $G$ is a graph without isolated vertices which is not $n$-colorable.
\end{theorem}

\begin{theorem} \label{thm: completeness for etlplus n}
  The logic $\ETLplus_{n}$ (for $n \geq 1$) is $\omega$-complete with respect to the class of all matrices of the form $\boldmu_{+}(G)$, where $G$ is a graph without isolated vertices which is not weakly $n$-colorable.
\end{theorem}

  This last batch of completeness theorems is perhaps less satisfying than the completeness theorems of $\ECQ$ and $\ETL$, but we shall see in the following section that for the logics $\ECQ_{n}$ and $\ETL_{n}$ this is unavoidable: apart from $\ECQ$ and $\ETL$, none of these logics are complete with respect to a finite set of finite matrices. We do not know whether this holds for $\ETLplus_{n}$.

\section{Super-Belnap logics and the homomorphism order on graphs}
\label{sec: hom order}

  Our description of the finite reduced models of $\BD$ yields a connection between finitary explosive extensions of~$\BD$ and homomorphisms of finite graphs: the lattice of finitary explosive extensions of $\BD$ turns out to be dually isomorphic to the lattice of \emph{homomorphic classes} of non-empty graphs, i.e.\ classes $\class{K}$ such that $H \in \class{K}$ whenever $G \in \class{K}$ and there is a homomorphism of graphs $G \to H$. In~the following, we write simply $G \to H$ to abbreviate the claim that there is a homomorphism of graphs $G \to H$.

  Such classes correspond to upsets in the so-called \emph{homomorphism order} on finite graphs. The relation $G \to H$ between finite non-empty graphs yields a pre-order on the class of all finite graphs. The homomorphism order on finite graphs is obtained by factoring this pre-order on a proper class down to a partially ordered set. The least element of this order is the equivalence class of $\bullet$, consisting of graphs without any edges. The least element above the equivalence class of $\bullet$ is the equivalence class of $K_{2}$, consisting of all bipartite graphs. The top element of this order is the class of all graphs with a loop. The fact that we allow for loops therefore does not have a substantial effect on the homomorphism order.

 The homomorphism order on graphs has been the object of much mathematical attention, see in particular the monograph of Hell \& Ne\v{s}et\v{r}il~\cite{hell+nesetril04}. We shall only need the following property of this partial order, called \emph{countable universality}.

\begin{theorem} \label{thm: countable universality}
  Each countable partial order embeds into the homo\-morphism order on finite graphs.
\end{theorem}

  We now prove that almost every homomorphic class of non-empty graphs arises as the class of all non-empty graphs $G$ such that $\boldmu_{+}(G) \in \Mod \logic{L}$ of some $\logic{L} \in \Exp \Ext_{\omega} \ETL$. Moreover, our proof will be constructive: for each non-empty graph $G$ without isolated vertices we construct a certain rule $(\alpha_{G})$ and we take $\logic{L}$ to be the extension of $\ETL$ by the rules $(\alpha_{G})$ for such $G \notin \class{K}$.

  There is only exception, namely the class of all graphs $G$ such that $K_{2} \to G$, or equivalently the class of all graphs with at least one edge. We denote this class by ${\uparrow} K_{2}$. To~see that it does not correspond to any $\logic{L} \in \Exp \Ext_{\omega} \ETL$, observe that $\ETL$ corresponds to the class of all non-empty graphs, and $\ETL_{2}$ corresponds to the class of all non-$2$-colorable graphs. The logic $\logic{L}$ corresponding to ${\uparrow} K_{2}$ would thus have to lie between $\ETL$ and $\ETL_{2}$. But no such logic exists.

  We now define the rules $(\alpha_{G})$. Consider a non-empty graph $G = \pair{X}{R}$ without isolated vertices. We want to describe the non-empty graphs $H$ without isolated vertices such that $H \notto G$ by means of an explosive rule. Let us assign a propositional atom $p_{u}$ to each $u \in X$ and define the formula $\varphi_{u}$ for $u \in X$ as
\begin{align*}
  \varphi_{u} & \assign p_u \wedge \bigwedge_{{\substack{v \in X \\ \neg u R v}}} \dmneg p_v.
\end{align*}
  In other words, the conjunctive clause $\varphi_{u}$ contains the atom $p_{u}$ and a negated atom for each vertex which is not adjacent to $u$. The explosive rule $(\alpha_{G})$ will be defined as
\begin{align*}
  \bigvee_{u \in X} \varphi_{u} & \vdash \emptyset. \tag{$\alpha_{G}$} 
\end{align*}

\begin{fact} \label{fact: simple homomorphisms}
  Each homomorphism of graphs $G \to H$ yields a homomorphism of matrices $\boldmu_{+}(H) \to \boldmu_{+}(G)$.
\end{fact}

\begin{proof}
  Each graph homomorphism $g\colon G \to H$ extends to a homo\-morphism of frames $\hat{g}\colon P_{+}(G) \rightarrow P_{+}(H)$ such that $\hat{g}(\dual u) = \dual g(u)$, which corresponds dually to a homomorphism of matrices $\boldmu_{+}(H) \rightarrow \boldmu_{+}(G)$ by Theorem~\ref{thm: matrix duality}.
\end{proof}

\begin{lemma} \label{lemma: alpha g}
  Let $G$ and $H$ be non-empty graphs without isolated vertices. Then $(\alpha_{G})$ holds in $\boldmu_{+}(H)$ if and only if $H \notto G$.
\end{lemma}

\begin{proof}
  The valuation $v$ on $\boldgamma(G)$ such that $v(p_{u}) \assign \{ u \}$ witnesses that the rule fails in $\boldgamma(G)$, since $u \in v(\varphi_{u})$. By Lemma \ref{lemma: gamma construction} the rule thus fails in~$\boldmu_{+}(G)$. If there is a graph homomorphism $H \rightarrow G$, then there is a homomorphism of matrices $\boldmu_{+}(G) \rightarrow \boldmu_{+}(H)$. But $(\alpha_{G})$ is an explosive rule, therefore if it fails in~$\boldmu_{+}(G)$ and $\boldmu_{+}(H)$ is non-trivial, then it fails in $\boldmu_{+}(H)$.

  Conversely, suppose that the rule $(\alpha_{G})$ fails in $\boldmu_{+}(H)$. By Lemma~\ref{lemma: gamma construction} it also fails in $\boldgamma(H)$, as witnessed by a valuation $w$. Let $G = \pair{X}{R}$ and $H = \pair{Y}{S}$. Consider the relation $Q \subseteq Y \times X$ so that $u' Q u$ if and only if $u' \in w(\varphi_{u})$ for $u \in X$ and $u' \in Y$. Firstly, each $u' \in Y$ is related to some $u \in X$ by $Q$ because $w(\bigvee_{u \in X} \varphi_{u}) = Y$. Secondly, we claim that $u' Q u$, $v' Q v$, and $u' S v'$ imply $u R v$.

  Suppose therefore that $u' Q u$ and $v' Q v$. If $u$ and $v$ are not adjacent in $G$, then $w(\varphi_{u}) \subseteq w(\dmneg p_{v})$, therefore $u' \in w(\varphi_{u}) \subseteq w(\dmneg p_{v}) = \dmneg w(p_{v})$. On the other hand, $v' \in w(\varphi_{v}) \subseteq w(p_{v})$. Thus $u'$ and $v'$ are not adjacent in $H$: no vertex in $\dmneg w(p_{v})$ can be adjacent to a vertex in $w(p_{v})$.

  Now consider any function $f\colon Y \to X$ whose graph is contained in $Q$, i.e.\ $u'$ and $f(u')$ are related by $Q$. By the first claim made above, such a function exists. By the second claim, it is a graph homomorphism $f\colon H \to G$.
\end{proof}

% page break
\pagebreak

\begin{theorem} \label{thm: homomorphic classes etl}
  The lattice $\Exp \Ext_{\omega} \ETL$ of finitary explosive extensions of $\ETL$ is dually isomorphic to the lattice of homomorphic classes of non-empty graphs $\class{K}$ other than ${\uparrow} K_{2}$ via the maps
\begin{align*}
  \logic{L} & \mapsto \class{K}_{\logic{L}} \assign \set{G}{\boldmu_{+}(G) \in \Mod \logic{L}}, & 
  \class{K} & \mapsto \logic{L}_{\class{K}} \assign \ETL + \set{(\alpha_{G})}{G \notin \class{K}}.
\end{align*}
\end{theorem}

\begin{proof}
  Consider a non-trivial finitary explosive extension $\logic{L}$ of $\ETL$. If $\logic{L} < \ETL_{\omega}$, then Theorem~\ref{thm: completeness for exp ext etl} states that $\logic{L}$ is complete with respect to models of the form $\boldmu_{+}(G) \times \ETLmatrix$ where $G$ is non-empty. Moreover, $\ETL_{\omega}$ is the largest non-trivial explosive extension of $\ETL$ and it is complete with respect to the matrix $\Kmatrix \times \ETLmatrix \cong \boldmu_{+}(H) \times \ETLmatrix$ where $H$ is a reflexive singleton graph. Each non-trivial finitary explosive extension of $\ETL$ is therefore uniquely determined by its models of the form $\boldmu_{+}(G) \times \ETLmatrix$. But if $G$ is non-empty, then $\boldmu_{+}(G) \times \ETLmatrix$ is a model of an explosive extension of $\ETL$ if and only if $\boldmu_{+}(G)$ is. Moreover, the matrix $\boldmu_{+}(\emptyset) \times \ETLmatrix$ is logically equivalent to $\boldmu_{+}(\bullet) \times \ETLmatrix$. Each non-trivial finitary explosive extension of $\ETL$ is thus uniquely determined by the class of all non-empty graphs $G$ such that $\boldmu_{+}(G) \in \Mod \logic{L}$. Let us call this class~$\class{K}_{\logic{L}}$. Observe that $\class{K}_{\logic{L}} = \emptyset$ if and only if $\logic{L}$ is the trivial logic.

  The previous paragraph shows that $\logic{L}_{1} \logleq \logic{L}_{2}$ if and only if $\class{K}_{\logic{L}_{2}} \subseteq \class{K}_{\logic{L}_{1}}$, where $\logic{L}_{1}$ and $\logic{L}_{2}$ are finitary explosive extensions of $\ETL$. Moreover, $\class{K}_{\logic{L}}$ is a homo\-morphic class of graphs: each homo\-morphism of non-empty graphs $G \to H$ yields a homomorphism of non-trivial matrices $\boldmu_{+}(H) \to \boldmu_{+}(G)$, so $\boldmu_{+}(G) \in \Mod \logic{L}$ implies $\boldmu_{+}(H) \in \Mod \logic{L}$ for each explosive extension $\logic{L}$ of $\ETL$.

  It remains to prove that each homomorphic class $\class{K}$ of non-empty graphs other than ${\uparrow} K_{2}$ has the form $\class{K}_{\logic{L}}$ for some suitable $\logic{L}$. For each such $\class{K}$ either $K_{2} \notin \class{K}$ or $\bullet \in \class{K}$. In the latter case, we take $\logic{L} = \ETL$. Let us thus assume that $K_{2} \notin \class{K}$.

  Now consider the extension $\logic{L}_{\class{K}}$ of $\ETL$ by the rules $(\alpha_{G})$ for each non-empty graph $G \notin \class{K}$ without isolated vertices. Let $\overline{G}$ denote the result of removing all isolated vertices from $G$. If $\overline{G}$ is non-empty, then $\boldmu_{+}(G) \in \Mod \logic{L}_{\class{K}}$ if and only if $\boldmu_{+}(\overline{G}) \in \Mod \logic{L}_{\class{K}}$. Moreover, $G \to H$ if and only if $\overline{G} \to H$, and $H \to G$ if and only if $H \to \overline{G}$. Thus $G \in \class{K}$ if and only if $\overline{G} \in \class{K}$, provided that $\overline{G} \neq \emptyset$.

  Suppose first that $\overline{G}$ is non-empty. Then $\boldmu_{+}(G) \in \Mod \logic{L}_{\class{K}}$ is equivalent to $\boldmu_{+}(\overline{G}) \in \Mod \logic{L}_{\class{K}}$, which is equivalent by Lemma~\ref{lemma: alpha g} to the claim that $\overline{G} \notto H$, or equivalently $G \notto H$, for each non-empty $H \notin \class{K}$ without isolated vertices. This is equivalent to the claim that if $\overline{H}$ is non-empty for $H \notin \class{K}$, then $G \notto \overline{H}$, or equivalently $G \notto H$. Because $\overline{G}$ is non-empty, $\overline{H} = \emptyset$ implies that $G \notto H$, therefore we may simplify this claim to: there is no homomorphism $G \to H$ for $H \notin \class{K}$. Because $\class{K}$ is a homomorphic class, this is equivalent simply to $G \in \class{K}$. Thus $\boldmu_{+}(G) \in \Mod \logic{L}_{\class{K}}$ if and only if $G \in \class{K}$, provided that $\overline{G}$ is non-empty.

  On the other hand, suppose that $\overline{G} = \emptyset$. Then $G \notin \class{K}$ because $\bullet \notin \class{K}$. The matrix $\boldmu_{+}(G)$ is logically equivalent to $\ETLmatrix$, therefore it remains to show that $\ETLmatrix \notin \Mod \logic{L}_{\class{K}}$. But $K_{2} \notin \class{K}$, so $\logic{L}_{\class{K}}$ validates the rule $(\alpha_{G})$ for $G = K_{2}$. This rule is precisely the rule $(p \wedge \dmneg p) \vee (q \wedge \dmneg q) \vdash \emptyset$ which axiomatizes $\ETL_{2}$.
\end{proof}

\begin{theorem} \label{thm: homomorphic classes ecq}
  The lattice $\Exp \Ext_{\omega} \ECQ$ of finitary explosive extensions of $\ECQ$ is dually isomorphic to the lattice of homomorphic classes of non-empty graphs $\class{K}$ other than ${\uparrow} K_{2}$ via the maps
\begin{align*}
  \logic{L} & \mapsto \set{G}{\boldmu_{+}(G) \in \Mod \logic{L}}, & 
  \class{K} & \mapsto \ECQ + \set{(\alpha_{G})}{G \notin \class{K}}.
\end{align*}
\end{theorem}

\begin{proof}
  $\Exp \Ext_{\omega} \ECQ$ and $\Exp \Ext_{\omega} \ETL$ are isomorphic via $\logic{L} \mapsto \logic{L} \vee \ETL$ and $\logic{L} \mapsto \Exp_{\BD} \logic{L}$. Concatenating these maps with the isomorphism from the previous theorem yields the isomorphism $\logic{L} \mapsto \class{K}_{\logic{L}}$ and $\class{K} \mapsto \Exp_{\BD} \logic{L}_{\class{K}} = \ECQ + \set{(\alpha_{G})}{G \notin \class{K}}$.
\end{proof}

  The lattice $\Exp \Ext_{\omega} \BD$ consists of $\BD$ and $\Exp \Ext_{\omega} \ECQ$ (Theorem~\ref{thm: iso exp ext etl}), therefore it can be described even more easily.

\begin{theorem} \label{thm: homomorphic classes bd}
  The lattice of finitary explosive extensions of $\BD$ is dually isomorphic to the lattice of homomorphic classes of non-empty graphs.
\end{theorem}

  The countable universality of the homomorphism order (Theorem~\ref{thm: countable universality}), or more precisely the fact that it contains an infinite antichain, yields a continuum of finitary explosive extensions of $\ETL$ and $\BD$.

\begin{corollary}
  The lattices $\Exp \Ext_{\omega} \ETL$ and $\Exp \Ext_{\omega} \BD$ both have the cardinality of the continuum.
\end{corollary}

\begin{corollary} \label{cor: continuum in three intervals}
  Each of the three intervals $[\BD, \LP]$, $[\ECQ, \LP \vee \ECQ]$, and $[\ETL, \CL]$ contains a continuum of finitary logics.
\end{corollary}

\begin{proof}
  The lattices $\Exp \Ext_{\omega} \ETL$, $\Exp \Ext_{\omega} \ECQ \subseteq [\ECQ, \LP \vee \ECQ]$, and $\LP \cap \Exp \Ext_{\omega} \ECQ \subseteq [\BD, \LP]$ are isomorphic by Theorems~\ref{thm: iso exp ext etl} and~\ref{thm: iso exp ext ecq}.
\end{proof}

  We can now answer positively the question posed in~\cite{rivieccio12} whether there are logics strictly between $\ETL_{n}$ and $\ETL_{n+1}$ for some $n$.

\begin{corollary}
  There is an infinite increasing chain of explosive extensions of $\ETL$ strictly between $\ETL_{2}$ and $\ETL_{3}$.
\end{corollary}

\begin{proof}
  It suffices to find in the homomorphism order a decreasing chain of finite $3$-colorable graphs which are not $2$-colorable. The sequence of cycles of lengths $2n+1$ for $n \geq 1$ is an example of such a chain.
\end{proof}

  We can also use the countable universality of the homomorphism order in a more sophisticated way to construct a non-finitary explosive extension of $\ETL$.

\begin{proposition}
  There is a non-finitary explosive extension of $\ETL$.
\end{proposition}

\begin{proof}
  Given a graph $G$, let $\logic{L}'_{G}$ be the extension of $\ETL$ by the rule $(\alpha_{G})$. Given a countable set of graphs $\class{K}$, let $\logic{L}'_{\class{K}} \assign \bigcap_{G \in \class{K}} \logic{L}'_{G}$. By the remarks preceding Propositions~\ref{prop: axiomatizing cap exp} and~\ref{prop: mod cap lexp} this logic is axiomatized by the rule $\set{\varphi_{G}}{G \in \class{K}} \vdash \emptyset$, provided that we use distinct variables in each of the formulas $\varphi_{G}$, and moreover $\Mod \logic{L}'_{\class{K}} = \bigcup_{G \in \class{K}} \Mod \logic{L}'_{G}$. By Lemma~\ref{lemma: alpha g} the matrix $\boldmu_{+}(H)$ fails to be a model of $\logic{L}'_{\class{K}}$ if and only if $H \to G$ for each $G \in \class{K}$. 

  If the logic $\logic{L}'_{\class{K}}$ is finitary, then there is some finite $\class{K}' \subseteq \class{K}$ such that the rule $\set{\varphi_{G}}{G \in \class{K}'} \vdash \emptyset$ axiomatizes $\logic{L}'_{\class{K}}$. In other words, there is some finite $\class{K}'$ such that $\logic{L}'_{\class{K}} = \logic{L}'_{\class{K}'}$. But then these two logics agree on models of the form $\boldmu_{+}(H)$. That is, a graph lies below all the graphs of $\class{K}$ in the homomorphism order whenever it lies below all the graphs in the finite set $\class{K}'$. All we have to do now is use the countable universality of the homomorphism order to pick some $\class{K}$ such that this equivalence does not hold for any finite $\class{K}' \subseteq \class{K}$. For example, consider an embedding of the free countably generated meet-semilattice into the homomorphism order and take $\class{K}$ to be the set of its maximal elements.
\end{proof}

  Some algebraic corollaries concerning antivarieties of De~Morgan algebras may be inferred from the above description of $\Exp \Ext_{\omega} \ETL$. The finitary extensions of $\ETL$ are precisely those finitary logics which are $\omega$-complete with respect to classes of De~Morgan matrices of the form $\pair{\alg{A}}{\{ \True \}}$. They are thus in bijective correspondence with quasivarieties of De~Morgan algebras axiomatized by quasiequations where each equality takes the form $\True \approx u$ for some term $u$. Similarly, the finitary explosive extensions of $\ETL$ are almost in bijective correspondence with antivarieties of De~Morgan algebras axiomatized by negative clauses where each equality takes the form $\True \approx u$: the only difference is that the trivial singleton matrix is a model of each extension of $\ETL$, but it can be excluded by the negative clause $\True \napprox \False$. Finally, observe that the negative clause $\True \napprox u_{1} \text{ or } \dots \text{ or } \True \napprox u_{n}$ is equivalent to $\True \napprox u_{1} \wedge \dots \wedge u_{n}$.

\begin{corollary}
  There are continuum many antivarieties of De~Morgan algebras (axiomatized by negative clauses of the form $\True \napprox u$).
\end{corollary}

\begin{corollary}
  There is a class of De~Morgan algebras axiomatized by an infinitary negative clause which is not an antivariety of De~Morgan algebras.
\end{corollary}

  These results depend essentially on the fact that the constants $\True$ and $\False$ are part of the signature of De~Morgan algebras. If~we drop them from this signature, we obtain the variety of \emph{De Morgan lattices}, which has only finitely many sub\-quasivarieties as shown by Pynko~\cite{pynko99c}. In particular, the only non-empty proper antivariety of De~Morgan lattices is axiomatized by $x \approx \dmneg x$.

  The existence of continuum many antivarieties of De~Morgan algebras complements the result of Adams \& Dziobiak~\cite{adams+dziobiak94} that there are continuum many quasi\-varieties of Kleene algebras. Let us observe, for the sake of completeness, that by contrast the lattice of proper anti\-varieties of Kleene algebras is rather trivial. Here by a \emph{proper} antivariety of Kleene algebras we mean an antivariety strictly included in the whole variety.

\begin{proposition}
  There are only two non-empty proper antivarieties of Kleene algebras, namely those axiomatized by $\True \napprox \False$ and by $x \napprox \dmneg x$.
\end{proposition}

\begin{proof}
  The antivariety axiomatized by $\True \napprox \False$ is the largest proper antivariety of Kleene algebras, since it only excludes the trivial algebra. Conversely, let $\class{K}$ be a non-empty antivariety of Kleene algebras. Then $\class{K}$ contains a non-trivial algebra, therefore $\Btwo \in \class{K}$ and $\Btwo \times \Kthree \in \class{K}$ by closure under homomorphic preimages. Pynko~\cite[Proposition~4.5]{pynko99c} shows that $\Btwo \times \Kthree$ generates the quasivariety of Kleene algebras axiomatized by $x \approx \dmneg x$. The antivariety axiomatized by $x \approx \dmneg x$ is therefore the smallest non-empty antivariety of Kleene algebras. On the other hand, if $\class{K}$ contains a non-trivial Kleene algebra which fails $x \napprox \dmneg x$, then it contains $\Kthree$ and therefore all Kleene algebras.
\end{proof}

  Finally, we can use the isomorphism between finitary explosive super-Belnap logics and homomorphic classes of non-empty graphs to show that certain logics are not complete with respect to any finite set of finite matrices.

  For the purposes of the following theorems, a homomorphic class of non-empty graphs is called \emph{non-exceptional} if it is not ${\uparrow} K_{2}$. Given a class $\class{K}$ of non-empty graphs, the non-exceptional homomorphic class generated by $\class{K}$ is the homo\-morphic class generated by $\class{K}$, except when this class would be ${\uparrow} K_{2}$. In that case, we take it to be the class of all non-empty graphs instead.

\begin{theorem}
  Consider a finitary explosive extension $\logic{L}$ of $\ETL$ and a class of non-empty graphs $\class{K}$. Then $\logic{L}$ is $\omega$-complete with respect to $\boldmu_{+}[\class{K}] \times \ETLmatrix \assign \set{\boldmu_{+}(G) \times \ETLmatrix}{G \in \class{K}}$ if and only if $\class{K}_{\logic{L}}$ is generated as a non-exceptional homomorphic class of non-empty graphs by $\class{K}$.
\end{theorem}

\begin{proof}
  By Theorem~\ref{thm: homomorphic classes etl}, the logic $\logic{L}$ is the largest finitary explosive extension of $\logic{L}$ such that $\boldmu_{+}[\class{K}] \subseteq \Mod \logic{L}$ (or equivalently, $\boldmu_{+}[\class{K}] \times \ETLmatrix \subseteq \Mod \logic{L}$) if and only if $\class{K}_{\logic{L}}$ is the smallest non-exceptional homomorphic class of non-empty graphs such that $\class{K} \subseteq \class{K}_{\logic{L}}$.
\end{proof}

\begin{theorem}
  Consider a proper finitary explosive extension $\logic{L}$ of $\BD$ and a class of non-empty graphs $\class{K}$. Then $\logic{L}$ is $\omega$-complete with respect to $\boldmu_{+}[\class{K}] \times \BDmatrix \assign \set{\boldmu_{+}(G) \times \BDmatrix}{G \in \class{K}}$ if and only if $\class{K}_{\logic{L}}$ is generated as a non-exceptional homomorphic class of non-empty graphs by $\class{K}$.
\end{theorem}

\begin{proof}
  Replace $\ETLmatrix$ by $\BDmatrix$ and Theorem~\ref{thm: homomorphic classes etl} by Theorem~\ref{thm: homomorphic classes ecq} in the previous proof.
\end{proof}

\begin{fact}
  $\ECQ_{n}$ and $\ETL_{n}$ are not complete with respect to any finite set of finite matrices for $n \geq 2$.
\end{fact}

\begin{proof}
  If $\logic{L} = \ETL_{n}$ or $\logic{L} = \ECQ_{n}$, then $\class{K}_{\logic{L}}$ is the class of all non-$n$-colorable graphs. It~suffices to prove that $\class{K}_{\logic{L}}$ is not finitely generated as a homomorphic class of graphs. This is a corollary of the classical theorem of Erd\H{o}s~\cite{erdos59} which states that for each positive $n$ and $g$ there is a graph of girth at least~$g$ which is not $n$-colorable. Here the \emph{girth} of a graph is the length of its shortest cycle ($1$~if the graph contains a loop). If $G \to H$, then the girth of $H$ is at most equal to the girth of $G$. Thus if $\class{K}_{\logic{L}}$ were finitely generated, there would be an upper bound on the girth of graphs in $\class{K}_{\logic{L}}$, contradicting the theorem of Erd\H{o}s.
\end{proof}

\section{A graph-theoretic description of \texorpdfstring{$\Ext_{\omega} \ETL$}{Ext ETL}}
\label{sec: graphs to logics}

  In theory, the whole lattice $\Ext_{\omega} \BD$ may be described in graph-theoretic terms. In practice, such a description is rather cumber\-some and in\-elegant, owing to the fact that we need to deal with triples $\langle G, H, k \rangle$ rather than merely with individual graphs. Fortunately, restricting to $\Ext_{\omega} \ETL$ will allow us to disregard the $H$ component of these triples, and further restricting to the interval $[\ETL, \ETL_{\omega}]$ will allow us to disregard the $k$ component as well. In this section, we work out the graph-theoretic description of $\Ext_{\omega} \ETL$ and its restriction to $[\ETL, \ETL_{\omega}]$.

  The key observation to recall here is that the lattice $\Ext_{\omega} \BD$ ($\Ext_{\omega} \ETL$) is iso\-morphic, by Theorem~\ref{thm: gratzer-quackenbush}, to the lattice of classes of finite reduced models of $\BD$ ($\ETL$) closed under isomorphisms, Leibniz reducts of submatrices ($\AlgS^{*}$), and Leibniz reducts of finite products ($\AlgP_{\omega}^{*}$). Because the class of finite reduced models of $\BD$ ($\ETL$) is closed under finite products, the last condition amounts to closure under finite products ($\AlgP_{\omega}$).

\begin{theorem}
  $\Ext_{\omega} \BD$ ($\Ext_{\omega} \ETL$) is isomorphic via $\logic{L} \mapsto \Mod^{*}_{\omega} \logic{L}$ to the lattice of classes of matrices $\boldmu(G, H, k)$ (with $H = \emptyset$) closed under finite direct products and Leibniz reducts of submatrices.
\end{theorem}

  We already know that finite products correspond dually to finite disjoint unions. We also know that the Leibniz reduct $\matr{A}^{*}$ of a finite De~Morgan matrix $\matr{A}$ corresponds to the Leibniz subframe of the dual frame $\matr{A}_{+}$ of $\matr{A}$ (Proposition~\ref{prop: leibniz subframes}), obtained by restricting to $\min \matr{A}_{+} \cup \dual [\min \matr{A}_{+}]$. The only remaining task is to describe in dual terms the submatrices of a given finite De~Morgan matrix $\matr{A}$.

  Subalgebras correspond dually to quotients of involutive posets in the duality of Cornish \& Fowler~\cite{cornish+fowler77} for De~Morgan algebras. If $\lesssim$ is a \emph{compatible} preorder on $P$, i.e.\ $u \leq v$ implies $u \lesssim v$ and moreover $u \lesssim v$ implies $\dual v \lesssim \dual u$, then the quotient $P / {\lesssim}$ is the involutive poset of equivalence classes of ${\lesssim}$ equipped with the natural order and involution. If the relation ${\lesssim}$ is the smallest compatible preorder on $P$ such that $u \lesssim v$, we say that ${\lesssim}$ is the principal preorder \emph{generated} by $\pair{u}{v}$, or less formally that ${\lesssim}$ is obtained by adding $u \leq v$ to $P$. A \emph{principal quotient} of $P$ is a quotient by a preorder generated by some pair $\pair{u}{v}$.

  If $P$ is moreover a frame and $\lesssim$ is a compatible preorder on $P$, then we can turn $P / {\lesssim}$ into a frame by taking the designated set to be the upward closure of $D_{P}$ with respect to ${\lesssim}$. This is the only way to define a designated set on $P / {\lesssim}$ which makes the canonical map $\pi\colon P \to P / {\lesssim}$ strict. Submatrices of $P^{+}$ thus correspond dually to quotients of~$P$ in this sense, and up to isomorphism $\AlgS^{*}(P^{+})$ consists of complex algebras of Leibniz subframes of quotients of $P$.

  It will be convenient to describe the result of taking the Leibniz reduct of a submatrix, or dually the Leibniz subframe of a quotient frame, by a sequence of simpler constructions.  A \emph{proper immediate submatrix} of a matrix $\matr{A}$ is a sub\-matrix of $\matr{A}$ which is a co-atom in the lattice of sub\-matrices of $\matr{A}$. An \emph{immediate submatrix} of $\matr{A}$ is either $\matr{A}$ itself or a proper immediate submatrix of $\matr{A}$.

  Dually, a \emph{proper immediate quotient} of a finite frame $P$ is a quotient of $P$ with respect to a compatible preorder which is an atom in the lattice of compatible preorders on $P$. Each proper immediate quotient is principal.

\begin{fact}
  If $\matr{A}$ is an immediate submatrix of $\matr{B}$, then $\matr{A}^{*}$ is iso\-morphic to $\matr{C}^{*}$ for some immediate submatrix $\matr{C}$ of $\matr{B}^{*}$.
\end{fact}

\begin{proof}
  Let $\matr{B} = \pair{\alg{B}}{G}$ and $\matr{A} = \pair{\alg{A}}{F}$ with $\alg{A} \leq \alg{B}$ and $F = G \cap \alg{B}$. Let $\theta$ be the Leibniz congruence $\Leibniz{\alg{B}}{G}$. Then the restriction of $\theta$ to $\alg{A}$ is compatible with $F$, hence $\matr{A}^{*}$ is isomorphic to the Leibniz reduct of the submatrix $\matr{C} \assign \pair{\alg{A} / \theta}{F / \theta}$ of $\matr{B}^{*}$. But $\matr{C}$ is an immediate submatrix of $\matr{B}^{*}$: if $\alg{A} / \theta \leq \alg{D} \leq \alg{B} / \theta$, then $\alg{A} \leq \pi^{-1}[\alg{D}] \leq \alg{B}$, where $\pi$ is the projection map $\pi\colon \alg{B} \to \alg{B} / \theta$, so $\pi^{-1}[\alg{D}] = \alg{A}$ or $\pi^{-1}[\alg{B}]$, and thus $\alg{D} = \alg{A} / \theta$ or $\alg{D} = \alg{B} / \theta$.
\end{proof}

\begin{lemma}
  Let $\matr{A}$ and $\matr{B}$ be finite reduced models of $\BD$. Then $\matr{A} \in \AlgS^{*}(\matr{B})$ if and only if $\matr{A}$ can be obtained from $\matr{B}$ by repeatedly taking Leibniz reducts of proper immediate submatrices. Equivalently, $\matr{A} \in \AlgS^{*}(\matr{B})$ if and only if $\matr{A}_{+}$ can be obtained from $\matr{B}_{+}$ by taking Leibniz subframes of proper immediate quotients.
\end{lemma}

\begin{proof}
  The second claim is simply a dual translation of the first claim. The right-to-left direction holds because $\AlgS^{*} \AlgS^{*} (\class{K}) = \AlgS^{*} (\class{K})$. Conversely, $\matr{A} \in \AlgS^{*}(\matr{B})$ if and only if there is a matrix $\matr{C} \leq \matr{B}$ such that $\matr{A} = \matr{C}^{*}$. But $\matr{C} \leq \matr{B}$ if and only if there is a sequence of matrices $\matr{C}_{0} \assign \matr{C} \leq \matr{C}_{1} \leq \dots \leq \matr{C}_{k} \assign \matr{B}$ such that $\matr{C}_{i}$ is an immediate submatrix of $\matr{C}_{i+1}$. By the previous lemma, $\matr{C}^{*}_{i}$ is the Leibniz reduct of an immediate submatrix of $\matr{C}^{*}_{i+1}$. But $\matr{C}^{*}_{0} = \matr{A}$ and $\matr{C}^{*}_{k} = \matr{B}$.
\end{proof}

  The task of describing $\Ext_{\omega} \ETL$ has therefore been reduced to the task of describing the Leibniz subframes of immediate quotients of the frames $P(G, \emptyset, k)$. This can be achieved by a straightforward case analysis.

  Here by a \emph{homomorphic image} of a graph $G$ we mean a graph $H$ such that there is a surjective homomorphism $G \to H$. In other words, $H$ can be obtained from $G$ by identifying (collapsing) certain vertices and adding some edges. This should be contrasted with the previous section, where we considered closure in the homomorphism order, with no requirement of surjectivity.

\begin{lemma} \label{lemma: s star dual}
  $\boldmu(H, \emptyset, j) \in \AlgS^{*} (\boldmu(G, \emptyset, i))$ if and only if the pair $\pair{H}{j}$ can be obtained from $\pair{G}{i}$ by some sequence of operations of the following types:
\begin{enumerate}
\item $\pair{G}{i} \mapsto \pair{H}{i}$ if $H$ is a homomorphic image of $G$,
\item $\pair{G \sqcup K_{2}}{i} \mapsto \pair{G \sqcup \bullet}{i}$,
\item $\pair{G \sqcup H}{i} \mapsto \pair{G}{i+1}$,
\item $\pair{G}{i+1} \mapsto \pair{G}{i}$ if $i \geq 1$,
\item $\pair{G}{1} \mapsto \pair{G}{0}$ if there is a loop in $G$.
\end{enumerate}
\end{lemma}

\begin{proof}
  The following case analysis shows that each of these operations can be obtained by taking Leibniz subframes of immediate quotients repeatedly, and conversely that taking the Leibniz subframe of an immediate quotient cor\-responds to some sequences of these operations.

  Consider points $x$ and $y$ in the frame $P(G, \emptyset, i)$ with $x \nleq y$ such that the pair $\pair{x}{y}$ generates an immediate quotient of $P(G, \emptyset, i)$. Let $P(H, \emptyset, j)$ be the Leibniz subframe of this immediate quotient.

  If $x = u$ and $y = \dual v$ for $u, v \in G$, then $H$ is obtained from $G$ by adding an edge between $u$ and $v$. If $x = u$ and $y = v$, then the quotient is immediate only if each neighbor of $v$ in $G$ is also a neighbor of $u$. If $v$ has no neighbors, then $H$ is obtained by adding an edge between $u$ and $v$. Otherwise, $H$ is obtained by removing $v$ from $G$. This is equivalent to identifing $u$ and $v$ in $G$ (even if $u$ and $v$ are neighbors or if $v$ has a loop).

  If $x = \dual u$ and $y = v$, then the quotient is immediate only if each neighbor of $u$ is adjacent to each neighbor of $v$. If $u = v$ is an isolated vertex, then $H$ is obtained from $G$ by adding a loop on this vertex. If $u = v$ has a loop but no neighbors other than itself, then $H$ is obtained from $G$ by removing this loop and $j = i + 1$. If $u = v$ has neighbors other than itself, then $H$ is obtained from $G$ by removing $u$. Otherwise, we may assume that $u \neq v$.

  If the only neighbor of $u$ is $v$ (if the only neighbor of $v$ is $u$), then $H$ is obtained by removing $u$ (removing $v$) from $G$. If $v$ ($u$) has neighbors other than $u$ ($v$), this amounts to identifying $u$ ($v$) with some neighbor of $v$ ($u$). Otherwise, this amounts to replacing a $K_{2}$ component by $\bullet$. Finally, if $u$ has a neighbor other than $v$ and vice versa, then $H$ is obtained by removing $u$ and $v$. This amounts to identifying $u$ with some neighbor of $v$ and vice versa.

  If $x = u$ for $u \in G$ and $y = \dual y$, then the quotient is immediate only if $u$ has a loop. Then $H = G$ and $j = i - 1$. If $x = \dual u$ and $y = \dual y$, then the quotient is never immediate, being included in the quotient generated by $\pair{\dual u}{u}$.

  Finally, if $x = \dual x$ and $y = \dual y$, then $H = G$ and $j = i - 1$.
\end{proof}

  We now have all the necessary ingredients to describe the lattice $\Ext_{\omega} \ETL$ in terms of graphs. In the following, we refer to the operation of replacing $G \sqcup K_{2}$ by $G \sqcup \bullet$ as \emph{contracting an isolated edge}.

\begin{theorem} \label{thm: ext etl graph theoretically}
  The map
\begin{align*}
  \logic{L} & \mapsto \pair{\set{G \neq \emptyset}{\boldmu_{+}(G) \in \Mod \logic{L}}}{\set{H \neq \emptyset}{\boldmu_{+}(H) \times \CLmatrix \in \Mod \logic{L}}}
\end{align*}
  is a dual isomorphism between $\Ext_{\omega} \ETL$ and the lattice (ordered by component\-wise inclusion) of pairs of classes of non-empty graphs $\pair{\class{K}_{0}}{\class{K}_{1}}$ with $\class{K}_{0} \subseteq \class{K}_{1}$ such that $\class{K}_{0}$ and $\class{K}_{1}$ are closed under taking homomorphic images and disjoint unions and under contracting isolated edges, and moreover
\begin{enumerate}
\item if $G \sqcup H \in \class{K}_{1}$, then $G \in \class{K}_{1}$,
\item if $G \in \class{K}_{1}$ and $G$ has a loop, then $G \in \class{K}_{0}$.
\end{enumerate}
\end{theorem}

\begin{proof}
  By Lemma~\ref{lemma: s star dual} this map yields a pair of classes of non-empty graphs satisfying the required conditions. (The condition $\class{K}_{0} \subseteq \class{K}_{1}$ follows from closure under products and the fact that $\CLmatrix$ is a model of each non-trivial super-Belnap logic.) Each finitary extension of $\ETL$ is complete with respect to its models of the forms $\boldmu_{+}(G)$ and $\boldmu_{+}(H) \times \CLmatrix$ for $G$ non-empty, therefore the map is an order embedding. (The trivial logic is complete with respect to the empty class of such models.) Finally, consider a pair of classes of non-empty graphs $\pair{\class{K}_{0}}{\class{K}_{1}}$ satisfying these conditions. We prove that this pair arises from the logic $\logic{L}$ determined by the matrices $\boldmu_{+}(G)$ for $G \in \class{K}_{0}$ and $\boldmu_{+}(H) \times \CLmatrix$ for $H \in \class{K}_{1}$.

  By Theorem~\ref{thm: gratzer-quackenbush} we know that the class of all finite reduced models of $\logic{L}$ is obtained by taking the Leibniz reducts of submatrices of finite direct products of the above matrices $\boldmu_{+}(G)$ and $\boldmu_{+}(H) \times \CLmatrix$. We must therefore show that this does not result in any new matrices of the forms $\boldmu_{+}(G)$ and $\boldmu_{+}(H) \times \CLmatrix$.

  If the direct product only contains matrices of the form $\boldmu_{+}(G)$, then the closure of $\class{K}_{0}$ under disjoint unions and the operations mentioned in Lemma~\ref{lemma: s star dual} ensures this. If the product contains at least one matrix of the form $\boldmu_{+}(H) \times \CLmatrix$, then by the inclusion $\class{K}_{0} \subseteq \class{K}_{1}$ and the closure of $\class{K}_{0}$ and $\class{K}_{1}$ under disjoint unions the product has the form $\boldmu_{+}(H) \times \CLmatrix^{k}$ for some $H \in \class{K}_{1}$ and some $k \geq 1$. The closure of $\pair{\class{K}_{0}}{\class{K}_{1}}$ under the conditions listed in Lemma~\ref{lemma: s star dual} again ensures that $\boldmu_{+}(G) \in \AlgS^{*}(\boldmu_{+}(H) \times \CLmatrix^{k})$ implies that $G \in \class{K}_{0}$, and $\boldmu_{+}(G) \times \CLmatrix \in \AlgS^{*}(\boldmu_{+}(H) \times \CLmatrix^{k})$ implies that $G \in \class{K}_{1}$.
\end{proof}

  We shall not explicitly state the analogue of Theorem~\ref{thm: ext etl graph theoretically} for the whole lattice $\Ext_{\omega} \BD$, on account of it being too cumbersome. However, it is clear how such a theorem would be obtained: one would merely extend the case analysis of Lemma~\ref{lemma: s star dual} to all matrices of the form $\boldmu(G, H, k)$. (There are no technical obstacles to be overcome here, merely some tedious case analysis.) Instead of talking about pairs of classes of graphs, one would talk about pairs of classes of pairs of graphs $\pair{G}{H}$.

  Instead of going in the direction of increased complexity, let show how this isomorphism can be simplified if we restrict to the interval $[\ETL, \ETL_{\omega}]$. For such logics $\logic{L}$ it suffices to record the non-empty \emph{loopless} graphs $G$ for which $\boldmu_{+}(G)$ is a model of $\logic{L}$. This yields a much neater description of $[\ETL, \ETL_{\omega}]$.

% page break
\pagebreak

\begin{theorem} \label{thm: etl to etl omega graph theoretically}
  The map
\begin{align*}
  \logic{L} & \mapsto \set{G \neq \emptyset}{\boldmu_{+}(G) \in \Mod \logic{L} \text{ and $G$ has no loops}}
\end{align*}
  is a dual isomorphism between the interval $[\ETL, \ETL_{\omega}] \subseteq \Ext_{\omega} \BD$ and the lattice (ordered by inclusion) of classes of non-empty graphs without loops closed under taking homo\-morphic images, disjoint unions, and contracting isolated edges.
\end{theorem}

\begin{proof}
  Each of the matrices $\boldmu_{+}(H) \times \CLmatrix$ is a model of $\ETL_{\omega}$, while $\boldmu_{+}(G) \in \Mod \ETL_{\omega}$ if and only if $G$ contains a loop (Fact~\ref{fact: contains reflexive vertex}). Theorem~\ref{thm: ext etl graph theoretically} therefore yields a dual isomorphism between $[\ETL, \ETL_{\omega}]$ and the lattice of pairs $\pair{\class{K}_{0}}{\class{K}_{1}}$ satisfying the conditions of Theorem~\ref{thm: ext etl graph theoretically} such that $\class{K}_{0}$ contains each graph with a loop and $\class{K}_{1}$ contains each non-empty graph. Now consider the map which assigns to $\pair{\class{K}_{0}}{\class{K}_{1}}$ the restriction $\class{L}_{0}$ of $\class{K}_{0}$ to loopless graphs. The closure conditions of Theorem~\ref{thm: ext etl graph theoretically} for $\pair{\class{K}_{0}}{\class{K}_{1}}$ imply the required closure conditions for~$\class{L}_{0}$. Conversely, if $\class{L}_{0}$ satisfies the closure conditions of the current theorem and we take $\class{K}_{1}$ to be the class of all non-empty graphs and $\class{K}_{0}$ to be the union of $\class{L}_{0}$ and the class of all graphs with at least one loop, then $\pair{\class{K}_{0}}{\class{K}_{1}}$ rather trivially satisfies the closure conditions of Theorem~\ref{thm: ext etl graph theoretically}.
\end{proof}

  Of course, closure under homomorphic images is interpreted here as closure restricted to the class of loopless graphs.

  We end with an example of how looking at super-Belnap logics from this dual, graph-theoretic perspective can simplify our proofs. Namely, we provide an alternative, and perhaps more transparent, proof of the completeness theorem for the logic~$\Kminus$ (Proposition~\ref{prop: kminus}) defined semantically by the matrix $\Kminusmatrix$ shown in Figure~\ref{fig: kminus matrix}.

\begin{proposition}
  $\Kminus = \Log \Kminusmatrix$.
\end{proposition}

\begin{proof}
  Observe that $\Kminusmatrix = \boldmu_{+}(G_2)$, where $G_{2}$ is obtained by adding a loop to $K_{2}$, i.e.\ the graph $G_{2}$ consists of a reflexive and an irreflexive vertex which are neighbors. By Theorem~\ref{thm: completeness not below etl omega} and Fact~\ref{fact: graphs for k minus}, the logic $\Kminus$ is complete with respect to the class of all matrices $\boldmu_{+}(G)$ such that each irreflexive vertex of~$G$ has a reflexive neighbor. We must therefore show that if $\boldmu_{+}(G_{2})$ is a model of a super-Belnap logic $\logic{L}$, then so is each such matrix $\boldmu_{+}(G)$. By Theorem~\ref{thm: ext etl graph theoretically} it suffices to show that each such graph $G$ can be obtained from $G_{2}$ by means of the operations allowed by this theorem. Indeed, we can take a copy of $G_{2}$ for each irreflexive vertex, a reflexive singleton for each reflexive vertex, and consider their disjoint union $H$. A graph homomorphism from $H$ onto $G$ is easily constructed using the assumption that each irreflexive vertex has a reflexive neighbor.
\end{proof}

% \bibliographystyle{plain}
% \bibliography{references}

\end{document}